\documentclass[1996/10/24]{amsart}

\usepackage{color}

\allowdisplaybreaks[3]

\usepackage{amsfonts}
\usepackage{epsf,latexsym,amsfonts,amsbsy,mathrsfs}
\usepackage{amsmath, amssymb, amsthm,amscd,amsxtra}
\usepackage[]{graphicx,subfigure}
\usepackage{epstopdf}
\usepackage{float}
\usepackage{multirow}
\usepackage{bbm}
\usepackage{stmaryrd}
\usepackage{epsfig}
\usepackage{tikz}
\usepackage{dsfont}
\usetikzlibrary{positioning} 
\usepackage{xcolor}

\usepackage{amssymb,stmaryrd,color}
\usepackage{lipsum}
\usepackage{amsfonts}
\usepackage{graphicx}
\usepackage{epstopdf}
\usepackage{algorithmic}
\usepackage{bm}
\usepackage{multirow}

\ifpdf
  \DeclareGraphicsExtensions{.eps,.pdf,.png,.jpg}
\else
  \DeclareGraphicsExtensions{.eps}
\fi

\newtheorem{theorem}{Theorem}[section]
\newtheorem{lemma}[theorem]{Lemma}

\theoremstyle{definition}

\theoremstyle{remark}
\newtheorem{remark}[theorem]{Remark}

\theoremstyle{corollary}
\newtheorem{corollary}[theorem]{Corollary}

\numberwithin{equation}{section}

\newcommand{\ignore}[1]{}

\begin{document}


\title[$W^{1,p}$ error estimates for the Poisson problem]{$W^{1,p}$ error estimates of an IPDG and a class of HDG methods for the Poisson problem on convex polyhedral domains}


\author{Haitao Leng}
\address{School of Mathematics and Information Science, Guangzhou University, Guangzhou 510006, Guangdong, China}
\email{htleng@m.scnu.edu.cn}

\author{Weifeng Qiu}
\address{Department of Mathematics, City University of Hong Kong, Hong Kong, China}
\email{weifeqiu@cityu.edu.hk}

\thanks{The second author is corresponding author.}


\date{}

\begin{abstract}
This paper analyzes a symmetric interior penalty discontinuous Galerkin (IPDG) method, a hybrid mixed DG method, and two representative hybridizable discontinuous Galerkin (HDG) methods for the Poisson problem on convex polyhedral domains. For any fixed $\epsilon\in (0,\frac{1}{2})$, we first derive optimal discrete $W^{1,p}$ error estimates for $2\leq p<\frac{1}{\epsilon}$ by combining the interface Green's functions developed in \cite{lq2026} with the Fubini-type weighted estimates established in this paper. These error estimates lead to discrete $W^{1,q}$-$W^{1,p}$ inf-sup stability for $2\leq q<\frac{1}{\epsilon}$, where $q$ is the conjugate number of $p$ such that $\frac{1}{p}+\frac{1}{q}=1$. Using this stability result and a suitably constructed dual problem, we extend the optimal error estimates to $\frac{1}{1-\epsilon}<p<2$ and subsequently establish the corresponding inf-sup stability for $\frac{1}{1-\epsilon}< q<2$. Thus, both optimal discrete $W^{1,p}$ error estimates and discrete $W^{1,q}$-$W^{1,p}$ inf-sup stability hold throughout the range $\frac{1}{1-\epsilon}<p<\frac{1}{\epsilon}$, with $q=p/(p-1)$. Note that all of the aforementioned results are free of logarithmic factors and also hold on polyhedral meshes (except for the hybrid mixed DG method). Compared with the conforming finite element method (FEM), establishing such error estimates and inf-sup stability results for DG methods on convex polyhedral domains is more delicate because their analysis involves mesh-dependent norms. Consequently, the operator interpolation theory employed for conforming FEMs cannot be directly applied to DG methods. Finally, a numerical example is presented to illustrate the theoretical results.
\end{abstract}

\maketitle

\section{Introduction}
In the numerical analysis of partial differential equations (PDEs), the $H^1$ norm provides a natural framework for energy-based
stability and error estimates, with trial and test functions typically chosen from the same finite-dimensional space. This framework, however, may be insufficient for problems with limited solution regularity, since the solution may belong to $W^{1,p}$ for some $p<2$ without belonging to $H^1$ \cite{c1985, lc2022, loq2021}. In addition,
$W^{1,p}$ stability and error estimates play an important role in establishing weak discrete maximum principles \cite{ll2020, s1980}, analyzing the convergence of numerical methods for parabolic problems \cite{mss2005}, and deriving optimal $W^{2,p}$ error estimates for elliptic equations in non-divergence form \cite{fns2018, nsz2017}. These considerations motivate us to conduct the present study of optimal $W^{1,p}$ error estimates.

For the conforming FEM applied to second-order elliptic equations, optimal $W^{1,p}$ error estimates \cite{rs1982} for $1<p<\infty$ can be obtained by combining $H^1$ and $W^{1,\infty}$ error estimates \cite{dlsw2012, drs2024, glrs2009} with operator interpolation theory \cite{bs2008} and duality arguments. For DG methods, however, establishing such error estimates is more delicate because their analysis involves mesh-dependent norms, and consequently, the operator interpolation theory used for the conforming FEM cannot be directly extended to DG methods.

Over the last two decades, DG methods have become one of the most successful classes of finite element methods owing to their flexibility on general meshes, local conservation, suitability for $hp$-adaptivity, and excellent parallel scalability; see, for example, \cite{c2002,de2011}. Numerous DG methods have subsequently been developed for second-order elliptic problems, including IPDG methods \cite{abcm2002, mss2017}, local DG methods \cite{c2006, ccps2000, ckps2001}, hybrid mixed DG methods \cite{es2010}, HDG methods \cite{cgl2009, cgnps2011, lc2022, qs2016}, and hybrid high-order methods \cite{be2018, cel2019}. For maximum-norm error estimates for IPDG, HDG, and local DG methods, here we refer the interested reader to \cite{c2005, cc2005, cmz2021, g2006, lc2026, lq2026}.

To the best of our knowledge, one of the earliest systematic analyses of discrete $W^{1,p}$ error estimates for DG methods was given in \cite{fns2018}. In that work, the authors established discrete $W^{1,p}~(1<p<\infty)$ error estimates and discrete $W^{1,q}$-$W^{1,p}~(2\leq p<\infty)$ inf-sup stability for the symmetric IPDG method applied to second-order elliptic equations on smooth domains, using the approach developed in \cite{c2005, cc2005} together with Fubini-type weighted estimates. More specifically, they proved the following estimates (see \cite[Theorem 3.1 and Lemma 3.5]{fns2018}):
\begin{align}
|||u-u_h|||_{W_{hh}^{1,p}}\leq Ch|\ln h|^{t}\inf_{\chi\in V_h^{\rm DG}}|u-\chi|_{W_h^{2,p}(\Omega)},\quad 1<p<\infty,\label{W1p-re:1}
\end{align}
and
\begin{align}
|||v_h|||_{W_{hh}^{1,p}(\Omega)}\leq C|\ln h|^t\sup_{0\neq w_h\in V_h^{\rm DG}}\frac{a_{h}(v_h,w_h)}{|||w_h|||_{W_{hh}^{1,q}(\Omega)}},\quad \forall v_h\in V_h^{\rm DG},~2\leq p<\infty,\label{W1p-re:2}
\end{align}
where $u$ is the solution of the model problem, $a_{h}$ is the bilinear form associated with the IPDG method of the model problem, $u_h\in V_h^{\rm DG}$ is the IPDG approximation of $u$, $V_h^{\rm DG}$ is the discrete space defined in Subsection \ref{sec3:1} consisting of piecewise polynomials with degree less than or equal to $k\geq 1$, $|||\cdot|||_{W_{hh}^{1,p}}$ is a mesh-dependent norm defined in Corollary \ref{W1p-IPDG-2}, $q$ is the conjugate number of $p$ such that $\frac{1}{p}+\frac{1}{q}=1$, and $t=(p+1)/p$ if $k=1$ and $t=0$ if $k\geq 2$. Using these $W^{1,p}$ error estimates and a freezing coefficient technique, the authors further derived error estimates in a discrete $W^{2,p}$ norm for the IPDG method with $k\geq 2$ applied to elliptic equations in non-divergence form. For $k=1$, however, the logarithmic factor in (\ref{W1p-re:1}) prevents a direct application of the same argument.

On convex polyhedral domains, extending the analysis in \cite{fns2018} introduces a logarithmic factor even for higher polynomial degrees. A central difficulty is the control of the jump contribution:
$$\sum_{F\in\mathcal{E}_h}h_F^{1-p}\|[u-u_h]\|_{L^p(F)}^p,$$
where $\mathcal{E}_h$ denotes the set of mesh faces and $[u-u_h]$ is the jump defined in Subsection \ref{sec3:1}.
More specifically, the authors of \cite{fns2018} used a trace inequality to obtain
\begin{align*}
\sum_{F\in\mathcal{E}_h}h_F^{1-p}\|[u-u_h]\|_{L^p(F)}^p\leq C\left(\|\nabla(u-u_h)\|_{L_h^p(\Omega)}^p+h^{-1}\|u-u_h\|_{L^p(\Omega)}\right).
\end{align*}
To estimate $\|u-u_h\|_{L^p(\Omega)}$, they employed a technique similar to those used to bound $\|u-u_h\|_{L^\infty(\Omega)}$ in \cite{c2005, cc2005, cmz2021, g2006, lc2026}. This technique arises from pointwise error estimates and inevitably leads to a logarithmic factor on convex polyhedral domains. In particular, the resulting logarithmic loss persists even for higher-degree polynomial approximations, so this approach does not yield logarithm-free estimates for the jump contribution. Thus, obtaining logarithm-free estimates requires a different treatment of the above interface jumps.

One possible route to removing this logarithmic loss is to exploit weak discrete
maximum principles, which have been established for FEMs in polygonal domains \cite{s1980} and convex polyhedral domains \cite{ll2020} for polynomial degrees greater than or equal to one. For DG methods, a weak discrete maximum principle would take the form:
\begin{align}
\|u_h^{\rm nm}\|_{L^{\infty}(\Omega)}\leq C\|u\|_{L^{\infty}(\partial\Omega)},\label{wdmp}
\end{align}
where $u$ is a harmonic function and $u_h^{\rm nm}$ is the corresponding discrete approximation obtained by the DG method. However, to the best of our knowledge, a weak discrete maximum principle of the form (\ref{wdmp}) has not been established for DG methods. We note that the weak discrete maximum principle for IPDG methods established in \cite{ChibaSaito2019} differs from (\ref{wdmp}), as it replaces $\Vert u\Vert_{L^{\infty}(\partial\Omega)}$ with $\Vert u_{h}^{{\rm nm}}\Vert_{L^{\infty}(\partial\Omega)}$, which is not known a priori. Furthermore, adapting the analysis of the weak discrete maximum principle in \cite{ll2020} to DG methods requires optimal discrete $W^{1,p}$ error estimates without logarithmic factors for some $p$ in a neighborhood of $2$. Thus, establishing such error estimates is an important step toward proving weak discrete maximum principles for DG methods, but it is not clear how to obtain these error estimates.

The primary objective of this paper is to address this problem.

In this paper, we consider a symmetric IPDG method, a hybrid mixed DG method, and two representative HDG methods for the Poisson problem on convex polyhedral domains. Using the interface Green's functions developed in \cite{lq2026} to represent and estimate jumps across the mesh interfaces, we reduce the pointwise error estimates for the gradients and jumps to weighted interpolation error estimates in the $L^p$ norm for $2\leq p<\frac{1}{\epsilon}$ (see (\ref{rG:2}), (\ref{rG:3}), (\ref{rG-HDG12:1}), and (\ref{rG-HDG0:1})), where $\epsilon\in(0,\frac{1}{2})$ is any given constant. Combining these pointwise error estimates with the Fubini-type weighted estimates established in Lemma \ref{fubini} and Lemma \ref{fubini1}, we obtain optimal discrete $W^{1,p}$ error estimates for $2\leq p<\frac{1}{\epsilon}$. These error estimates imply discrete $W^{1,q}$-$W^{1,p}$ inf-sup stability for $2\leq q<\frac{1}{\epsilon}$. Using this stability result and a suitably constructed dual problem, we then derive optimal discrete $W^{1,p}$ error estimates for $\frac{1}{1-\epsilon}<p<2$, which in turn imply discrete $W^{1,q}$-$W^{1,p}$ inf-sup stability for $\frac{1}{1-\epsilon}< q<2$. The optimal discrete $W^{1,p}$ error estimates for the IPDG and HDG methods are summarized in Theorem \ref{W1p-IPDG} and Theorem \ref{W1p-HDG-pre}, with proofs provided in Subsection \ref{sec5:1} and Subsection \ref{sec5:2}, respectively. Beyond their role in the error analysis, the discrete $W^{1,q}$-$W^{1,p}$ inf-sup stability results
for the IPDG and HDG methods, presented in Lemma \ref{inf-sup} and Lemma \ref{inf-sup-HDG0} respectively, are of independent interest in numerical analysis.
Notably, both the optimal discrete $W^{1,p}$ error estimates and the discrete $W^{1,q}$-$W^{1,p}$ inf-sup stability are logarithm-free and also hold on polyhedral meshes (except for the hybrid mixed DG method). They can therefore be used to prove $W^{2,p}$ error estimates for DG approximations of elliptic equations in non-divergence form and to establish weak discrete maximum principles for DG methods on convex polyhedral domains.

The rest of this paper is organized as follows. Section \ref{sec2} and Section \ref{sec3} introduce the notation and mesh assumptions, present the IPDG and HDG approximations of the Poisson problem, and review some auxiliary results. Section \ref{sec4} summarizes the main results of this paper, and Section \ref{sec5} provides detailed proofs of these results. Section \ref{sec6} presents a numerical example to illustrate
the theoretical results. Finally, we end this paper with some conclusions in Section \ref{sec7}.

\section{Notations}\label{sec2}
Let $\Omega\subset \mathbb{R}^3$ be an open bounded domain that is convex with polyhedral boundary $\partial\Omega$. For any nonnegative integer $t$, an open subset $D\subset\Omega$ and $1\leq p\leq \infty$, let $W^{t,p}(D)$ be the Sobolev spaces (cf. \cite{a1975}) with
norms
\begin{align*}
\|v\|_{W^{t,p}(D)}^p=\sum_{i=0}^t|v|_{W^{i,p}(D)}^p\quad 1\leq p<\infty,\quad \|v\|_{W^{t,\infty}(D)}=\max_{0\leq i\leq t}|v|_{W^{i,\infty}(D)},
\end{align*}
and seminorms
\begin{align*}
|v|^p_{W^{i,p}(D)}=\sum_{|\alpha|=i}\int_D\left|\frac{\partial^{\alpha}v}{\partial x^{\alpha}}\right|^p dx\quad 1\leq p<\infty,\quad
|v|_{W^{i,\infty}(D)}=\max_{|\alpha|=i}\left\|\frac{\partial^{\alpha}v}{\partial x^{\alpha}}\right\|_{L^{\infty}(D)},
\end{align*}
where
\begin{align*}
\frac{\partial^{\alpha}v}{\partial x^{\alpha}}=\frac{\partial^{|\alpha|}v}{\partial x_1^{\alpha_1}\partial x_2^{\alpha_2}\partial x_3^{\alpha_3}}
\end{align*}
for the multi-index $\alpha=(\alpha_1,\alpha_2,\alpha_3),~\alpha_1\geq 0,~\alpha_2\geq 0,~\alpha_3\geq 0$, and $|\alpha|=\alpha_1+\alpha_2+\alpha_3$.
Denote by $W_0^{t,p}(D)~(1<p<\infty)$ the completion of $\mathcal{C}_0^{\infty}(D)$ according to the norm $\|\cdot\|_{W^{t,p}(D)}$, where
$\mathcal{C}_0^{\infty}(D)$ is the space of functions with continuous derivatives of arbitrary order and compact support in $D$. If $p=2$, the
Sobolev spaces $W^{t,2}(D)$ are denoted by $H^t(D)$ with norms $\|\cdot\|_{H^t(D)}$ and seminorms $|\cdot|_{H^t(D)}$. If $t=0$, we set $L^p(D)=W^{0,p}(D)$
with norm $\|\cdot\|_{L^p(D)}$.

Let $\mathcal{T}_h$ be a conforming and quasi-uniform mesh of the domain $\Omega$. The mesh $\mathcal{T}_h$ may consist of either simplices or general
polyhedral elements. For a general polyhedral mesh, we assume that for each polyhedral element $K$, there is a shape-regular sub-triangulation and the number of sub-elements is uniformly bounded by a positive constant $\mathcal{N}$ independent of the mesh size. The meshes associated with the different numerical methods will be specified when the methods are introduced.
Denote by $\mathcal{E}_h^o$ the set of all interior faces of $\mathcal{T}_h$ and $\mathcal{E}_h^{\partial}$ the set of all boundary faces.
We define $\mathcal{E}_h=\mathcal{E}_h^o\cup\mathcal{E}_h^{\partial}$ and $\partial\mathcal{T}_h=\{\partial K: K\in\mathcal{T}_h\}$, where $\partial K$ denotes
the boundary of the element $K$. For each element $K\in\mathcal{T}_h$ and face $F\in\mathcal{E}_h$, let $h_K$ and $h_F$ be the diameters of $K$ and $F$, respectively. Moreover, we set $h=\max_{K\in\mathcal{T}_h}h_K$.

Next, we introduce mesh-dependent Sobolev spaces:
\begin{align*}
W^{t,p}_h(D)=\{v: v\in W^{t,p}(K\cap D), ~\forall K\in\mathcal{T}_h\},
\end{align*}
with norms
\begin{align*}
\|v\|_{W_h^{t,p}(D)}^p=\sum_{i=0}^t|v|_{W^{i,p}_h(D)}^p\quad 1\leq p<\infty,\quad \|v\|_{W_h^{t,\infty}(D)}=\max_{0\leq i\leq t}|v|_{W_h^{i,\infty}(D)},
\end{align*}
and seminorms
\begin{align*}
|v|_{W_h^{i,p}(D)}^p=\sum_{K\in\mathcal{T}_h}|v|^p_{W^{i,p}(K\cap D)}\quad 1\leq p<\infty,\quad |v|_{W^{i,\infty}_h(D)}=\max_{K\in\mathcal{T}_h}|v|_{W^{i,\infty}(K\cap D)}.
\end{align*}
If $p=2$, the mesh-dependent Sobolev spaces $W_h^{t,2}(D)$ are denoted by $H_h^t(D)$ with norms $\|\cdot\|_{H_h^t(D)}$ and  seminorms $|\cdot|_{H_h^t(D)}$.
If $t=0$, the mesh-dependent Sobolev spaces $W_h^{0,p}(D)$ are abbreviated by $L_h^p(D)$ with norms $\|\cdot\|_{L_h^p(D)}$.

\section{IPDG and HDG methods}\label{sec3}
In this section, we consider the following Poisson problem in the convex polyhedral domain $\Omega\subset \mathbb{R}^3$:
\begin{align}
-\Delta u=f,\quad \text{in}~\Omega,\quad \quad u=0,\quad \text{on}~\partial\Omega.\label{model}
\end{align}
\begin{lemma}[Regularity]\label{model-r}
(1) There exists a constant $p_0>2$, depending on the domain $\Omega$, such that for any $f\in L^p(\Omega)~(1<p<p_0)$, problem (\ref{model}) has a unique solution $u\in W^{2,p}(\Omega)\cap W_0^{1,p}(\Omega)$ which satisfies
\begin{align}
\|u\|_{W^{2,p}(\Omega)}\leq C\|f\|_{L^p(\Omega)}.\label{model-r:1}
\end{align}

(2) For any $f\in W^{-s,p}(\Omega)$ with $0<s\leq 1$ and $1<p<\frac{2}{1-s}$, problem (\ref{model}) has a unique solution $u\in W^{1,p}_0(\Omega)\cap W^{2-s,p}(\Omega)$ which satisfies
\begin{align}
\|u\|_{W^{2-s,p}(\Omega)}\leq C\|f\|_{W^{-s,p}(\Omega)},\label{model-r:2}
\end{align}
where $W^{-s,p}(\Omega)$ denotes the dual space of $W_0^{s,p^{\prime}}(\Omega)$ and $p^{\prime}$ is the conjugate number of $p$ such that
$\frac{1}{p}+\frac{1}{p^{\prime}}=1$.
\end{lemma}
\begin{proof}
The results in (1) follow from \cite[Corollary 3.12]{d1992} and \cite[Lemma 2.1]{ll2020}, while those in (2)
can be obtained from \cite[Corollary 1]{f1993}.
\end{proof}

For any given constant $\epsilon\in (0,\frac{1}{2})$, by choosing $s=1-\frac{1}{p}-\epsilon$ with $\frac{1}{1-\epsilon}<p<\frac{1}{\epsilon}$ in (2) of Lemma \ref{model-r}, we know that for
any $f\in W^{\frac{1}{p}+\epsilon-1,p}(\Omega)$, problem (\ref{model}) has a unique solution $u\in W_0^{1,p}(\Omega)\cap W^{1+\frac{1}{p}+\epsilon,p}(\Omega)$
which satisfies
\begin{align}
\|u\|_{W^{1+\frac{1}{p}+\epsilon,p}(\Omega)}\leq C\|f\|_{W^{\frac{1}{p}+\epsilon-1,p}(\Omega)}.\label{model-r:3}
\end{align}

By setting ${\bm q}=-\nabla u$, problem (\ref{model}) can be rewritten as the following mixed formulation:
\begin{align}\label{mixed}
{\bm q}+\nabla u=0\quad {\rm in}~\Omega,\quad \nabla\cdot{\bm q}=f\quad {\rm in}~\Omega,\quad
u=0\quad {\rm on}~\partial\Omega.
\end{align}

\subsection{IPDG methods for (\ref{model})}\label{sec3:1}
Based on the general polyhedral mesh $\mathcal{T}_h$, we define the following discrete space
\begin{align*}
V_h^{{\rm DG}}=\{v_h\in L^2(\Omega): v_h|_K\in\mathcal{P}^k(K),~ \forall K\in\mathcal{T}_h\},\quad k\geq 1,
\end{align*}
where $\mathcal{P}^k(D)$ denotes the set of polynomials with degree less than or equal to $k$ on the domain $D$.
Then, the IPDG discretization of (\ref{model}) reads as follows: Find $u_h^{\rm DG}\in V_h^{{\rm DG}}$ such that
\begin{align}
a_h(u_h^{{\rm DG}},v_h)=(f,v_h)_{\mathcal{T}_h},\quad \forall v_h\in V_h^{{\rm DG}},\label{IPDG}
\end{align}
where the bilinear form $a_h$ is defined as:
\begin{align*}
a_h(u_h^{{\rm DG}},v_h)=&(\nabla u_h^{{\rm DG}},\nabla v_h)_{\mathcal{T}_h}-\sum_{F\in\mathcal{E}_h}\int_F\left(\{\nabla u_h^{{\rm DG}}\cdot{\bm n}\}[v_h]+\{\nabla v_h\cdot{\bm n}\}[u_h^{{\rm DG}}]\right) ds\\
&+\sum_{F\in\mathcal{E}_h}\int_F\frac{\lambda}{h_F}[u_h^{{\rm DG}}][v_h] ds.
\end{align*}
Here, $(\cdot,\cdot)_{\mathfrak{D}}=\sum_{K\in\mathfrak{D}}(\cdot,\cdot)_K$ for any $\mathfrak{D}\subset\mathcal{T}_h$, $(\cdot,\cdot)_K$ denotes the inner product in $L^2(K)$, $\lambda$ is the stabilization parameter which is independent of the mesh size and usually chosen to be large enough to ensure the stability of (\ref{IPDG}). For each interior face
$F\in\mathcal{E}_h^o$, $\{\nabla v_h\cdot{\bm n}\}$ and $[v_h]$ denote the average of the normal derivative and the jump of $v_h$ defined by
\begin{align*}
\{\nabla v_h\cdot{\bm n}\}=\frac{1}{2}\left(\nabla v_h^+\cdot{\bm n}_F+\nabla v_h^-\cdot{\bm n}_F\right),\quad [v_h]=v_h^+-v_h^-,
\end{align*}
where $v_h^+$ and $v_h^-$ are the restrictions of $v_h$ on the elements $K^+$ and $K^-$ that share the common face $F$, and
${\bm n}_F$ is the unit normal vector on $F$ pointing from $K^+$ to $K^-$. For each boundary face $F\in\mathcal{E}_h^{\partial}$, the average and jump are defined as $\{\nabla v_h\cdot{\bm n}\}=\nabla v_h\cdot{\bm n}$ and $[v_h]=v_h$, where ${\bm n}$ denotes the outward unit normal vector to $\partial\Omega$.

\subsection{HDG methods for (\ref{mixed})}\label{sec3:2}
Based on the mesh $\mathcal{T}_h$, we define the following discrete spaces:
\begin{align*}
{\bm W}_h=&\{{\bm w}_h\in (L^2(\Omega))^3: {\bm w}_h|_K\in {\bm W}(K),~\forall K\in\mathcal{T}_h\},\\
V_h=&\{v_h\in L^2(\Omega): v_h|_K\in V(K),~\forall K\in\mathcal{T}_h\},\\
\widehat{V}_h=&\{\widehat{v}_h\in L^2(\mathcal{E}_h): \widehat{v}_h|_F\in\mathcal{P}^k(F)~\forall F\in\mathcal{E}_h,~\widehat{v}_h|_F=0~\forall F\in\mathcal{E}_h^{\partial}\},
\end{align*}
where the local discrete spaces $W(K)$ and $V(K)$ are defined as follows:
\begin{align*}
({\bm W}(K), V(K))=\begin{cases}
\left((\mathcal{P}^k(K))^3+{\bm x}\mathcal{P}^k(K), \mathcal{P}^k(K)\right),~ &{\rm hybrid~mixed~DG},\quad{\rm on~simplicial~meshes},\\
\left((\mathcal{P}^k(K))^3,\mathcal{P}^k(K)\right), &{\rm HDG}-1,\quad{\rm on~general~polyhedral~meshes},\\
\left((\mathcal{P}^k(K))^3,\mathcal{P}^{k+1}(K)\right), &{\rm HDG}-2,\quad{\rm on~general~polyhdedal~meshes}.
\end{cases}
\end{align*}
Then, the HDG discretization of (\ref{mixed}) reads as follows: Find $({\bm q}_h,u_h,\widehat{u}_h)\in {\bm W}_h\times V_h\times \widehat{V}_h$ such that
\begin{subequations}\label{HDG}
\begin{align}
({\bm q}_h,{\bm w}_h)_{\mathcal{T}_h}-(u_h,\nabla\cdot{\bm w}_h)_{\mathcal{T}_h}+\langle\widehat{u}_h,{\bm w}_h\cdot{\bm n}\rangle_{\partial\mathcal{T}_h}=0,\label{HDG:1}\\
-({\bm q}_h,\nabla v_h)_{\mathcal{T}_h}+\langle\widehat{\bm q}_h\cdot{\bm n},v_h\rangle_{\partial\mathcal{T}_h}=(f,v_h)_{\mathcal{T}_h},\label{HDG:2}\\
\langle\widehat{\bm q}_h\cdot{\bm n},\widehat{v}_h\rangle_{\partial\mathcal{T}_h}=0,\label{HDG:3}\\
\widehat{\bm q}_h\cdot{\bm n}={\bm q}_h\cdot{\bm n}+\tau(\mathcal{L}u_h-\widehat{u}_h),\quad {\rm on}~\partial\mathcal{T}_h,\label{HDG:4}
\end{align}
\end{subequations}
for any $({\bm w}_h,v_h,\widehat{v}_h)\in {\bm W}_h\times V_h\times \widehat{V}_h$, where ${\bm n}$ is the outward unit normal vector to $\partial K$, $\langle\cdot,\cdot\rangle_{\partial K}$ denotes the inner product in $L^2(\partial K)$, and $\langle\cdot,\cdot\rangle_{\partial \mathfrak{D}}=\sum_{K\in \mathfrak{D}} \langle\cdot,\cdot\rangle_{\partial K}$
for any $\mathfrak{D}\subset\mathcal{T}_h$. Here, $\tau$ denotes the stabilization parameter. On each $K\in\mathcal{T}_h$, it is chosen as:
\begin{align*}
\begin{cases}
\tau=0~{\rm on}~\partial K,\quad & {\rm hybrid~mixed~DG},\\
\tau=h_K^{-1}~{\rm on}~F^{\ast}_K~{\rm and}~\tau=0~{\rm on}~\partial K\backslash F_K^{\ast},\quad & {\rm HDG-1},\\
\tau=h_K^{-1}~{\rm on}~\partial K,\quad &{\rm HDG-2},
\end{cases}
\end{align*}
where $F_K^{\ast}$ denotes a fixed face of the element $K$ such that $K$ lies on one side of the hyperplane containing $F_K^{\ast}$. $\mathcal{L}$ is a projection operator defined as:
\begin{align*}
\mathcal{L}=\begin{cases}
\mathcal{I},\quad & {\rm hybrid~mixed~DG},~{\rm HDG-1},\\
\Pi_{\widehat{V}},\quad &{\rm HDG-2},
\end{cases}
\end{align*}
where $\mathcal{I}$ is the identity mapping and $\Pi_{\widehat{V}}$ denotes the standard $L^2$-orthogonal projection onto $\mathcal{P}^k(F)$ on each
$F\in\mathcal{E}_h$. Summing the equations in (\ref{HDG}) together, we obtain an equivalent formulation of (\ref{HDG}): Find $({\bm q}_h,u_h,\widehat{u}_h)\in {\bm W}_h\times V_h\times \widehat{V}_h$ such that
\begin{align}
\mathcal{B}({\bm q}_h,u_h,\widehat{u}_h;{\bm w}_h,v_h,\widehat{v}_h)=-(f,v_h)_{\mathcal{T}_h},\quad \forall ({\bm w}_h,v_h,\widehat{v}_h)\in {\bm W}_h\times V_h\times \widehat{V}_h,\label{eHDG}
\end{align}
where
\begin{align*}
\mathcal{B}({\bm q}_h,u_h,\widehat{u}_h;{\bm w}_h,v_h,\widehat{v}_h)=&({\bm q}_h,{\bm w}_h)_{\mathcal{T}_h}-(u_h,\nabla\cdot{\bm w}_h)_{\mathcal{T}_h}
+\langle\widehat{u}_h,{\bm w}_h\cdot{\bm n}\rangle_{\partial\mathcal{T}_h}\\
&-(\nabla\cdot{\bm q}_h,v_h)_{\mathcal{T}_h}-\langle\tau(\mathcal{L}u_h-\widehat{u}_h),\mathcal{L}v_h-\widehat{v}_h\rangle_{\partial \mathcal{T}_h}
+\langle{\bm q}_h\cdot{\bm n},\widehat{v}_h\rangle_{\partial\mathcal{T}_h}.
\end{align*}

The hybrid mixed DG, HDG-1, and HDG-2 methods are introduced in \cite{es2010, cfq2018, qs2016}, respectively. In fact, the hybrid mixed DG method provides the
same approximations as the mixed method which approximates ${\bm q}$ by $H$(div)-conforming Raviart-Thomas elements.

\subsection{Auxiliary results}\label{sec3:3}
We conclude this section by introducing some results associated to the discrete problems (\ref{IPDG}) and (\ref{HDG}),
and some well-known results that will be used frequently in the subsequent proofs.

\begin{lemma}\label{dfp}
Let $({\bm q},u)$ be the solution of problem (\ref{mixed}), then we have
\begin{align}
a_h(u,v_h)=(f,v_h)_{\mathcal{T}_h},\quad \forall v_h\in V_h^{{\rm DG}},
\end{align}
and
\begin{align}
\mathcal{B}({\bm q},u,u;{\bm w}_h,v_h,\widehat{v}_h)=-(f,v_h)_{\mathcal{T}_h},\quad \forall ({\bm w}_h,v_h,\widehat{v}_h)\in {\bm W}_h\times V_h\times \widehat{V}_h.
\end{align}
\end{lemma}
\begin{proof}
This lemma can be proved directly by the integration by parts.
\end{proof}

\begin{lemma}\label{inequality}
Let $({\bm q}_h,u_h,\widehat{u}_h)\in {\bm W}_h\times V_h\times \widehat{V}_h$ satisfy (\ref{HDG:1}), then we have
\begin{align}
\|\nabla u_h\|^2_{L^2(K)}+h_K^{-1}\|u_h-\widehat{u}_h\|_{L^2(\partial K)}^2\leq C\left(\|{\bm q}_h\|^2_{L^2(K)}+\|\sqrt{\tau}(\mathcal{L}u_h-\widehat{u}_h)\|
^2_{L^2(\partial K)}\right),
\end{align}
for each $K\in\mathcal{T}_h$.
\end{lemma}
\begin{proof}
The result of this lemma follows from \cite[Theorem 2.3]{cfq2018}, \cite[Lemma 3.2]{qs2016}, and \cite[Lemma 3.1]{es2010} for HDG-1, HDG-2, and hybrid mixed DG methods, respectively.
\end{proof}

\begin{lemma}[Local energy error estimates]\label{lee}
Let $\Omega_0\subset\Omega_1\subset\Omega$ satisfy $\rho={\rm dist}(\partial\Omega_0,\partial\Omega_1\backslash\partial\Omega)\geq rh$ for $r>1$ sufficiently large.

(1) Let $u\in H_0^1(\Omega)\cap H^2(\Omega)$ and $u_h^{{\rm DG}}\in V_h^{\rm DG}$ be the solutions of problems (\ref{model}) and (\ref{IPDG}), then we have
\begin{align}
&\Big(\sum_{F\in\mathcal{E}_h}\Big(h_F^{-1}\|[u-u_h^{{\rm DG}}]\|^2_{L^2(F\cap\Omega_0)}+h_F\|\{\nabla(u-u_h^{{\rm DG}})\cdot{\bm n}\}\|^2_{L^2(F\cap\Omega_0)}\Big)\Big)^{1/2}\nonumber\\
&+\|\nabla(u-u_h^{{\rm DG}})\|_{L_h^2(\Omega_0)}\leq Ch|u-\Pi_ku|_{H_h^2(\Omega_1)}+C\rho^{-1}\|u-u_h^{{\rm DG}}\|_{L^2(\Omega_1)},
\end{align}
where $\Pi_k$ denotes the standard $L^2$-orthogonal projection onto $V_h^{{\rm DG}}$ on each $K\in\mathcal{T}_h$.

(2) Let $({\bm q},u)\in (H^1(\Omega))^3\times H_0^1(\Omega)\cap H^2(\Omega)$ and $({\bm q}_h,u_h,\widehat{u}_h)\in{\bm W}_h\times V_h\times \widehat{V}_h$ be the solutions of problems (\ref{mixed}) and (\ref{HDG}), then we have
\begin{align}
&\|{\bm q}-{\bm q}_h\|_{L^2(\Omega_0)}+\|\nabla(u-u_h)\|_{L_h^2(\Omega_0)}+\Big(\sum_{K\in\mathcal{T}_h}\|\sqrt{\tau}(\mathcal{L}u_h-\widehat{u}_h)\|^2_{L^2(\partial K\cap\Omega_0)}\Big)^{1/2}\nonumber\\
\leq& Ch\rho^{-1}\Big(\|{\bm q}-{\bm q}_h\|_{L^2(\Omega_1)}+\Big(\sum_{K\in\mathcal{T}_h}\|\sqrt{\tau}(\mathcal{L}u_h-\widehat{u}_h)\|^2_{L^2(\partial K\cap\Omega_1)}\Big)^{1/2}\Big)\\
&+C\rho^{-1}\|u-u_h\|_{L^2(\Omega_1)}+C\big(\|\nabla(u-\Pi_Vu)\|_{L_h^2(\Omega_1)}+h\|\nabla({\bm q}-\Pi_W{\bm q})\|_{L_h^2(\Omega_1)}\big),\nonumber
\end{align}
where $\Pi_V$ and $\Pi_W$ denote the standard $L^2$-orthogonal projections onto $V(K)$ and ${\bm W}(K)$ on each $K\in\mathcal{T}_h$.
\end{lemma}
\begin{proof}
The local energy error estimates presented in (1) and (2) follow from \cite[Lemma 4.1]{cc2005} and \cite[Lemma 3.4]{clq2026} respectively.
\end{proof}

\begin{lemma}[Energy error estimates]\label{leee}
(1) Let $u\in H_0^1(\Omega)\cap H^2(\Omega)$ and $u_h^{{\rm DG}}\in V_h^{\rm DG}$ be the solutions of problems (\ref{model}) and (\ref{IPDG}), then we have
 \begin{align}
 &\|\nabla(u-u_h^{{\rm DG}})\|_{L_h^2(\Omega)}+\Big(\sum_{F\in\mathcal{E}_h}h_F^{-1}\|[u-u_h^{{\rm DG}}]\|^2_{L^2(F)}\Big)^{1/2}\nonumber\\
 \leq&C\|\nabla(u-\Pi_ku)\|_{L_h^2(\Omega)}+C\Big(\sum_{F\in\mathcal{E}_h}h_F\|\{\nabla(u-\Pi_ku)\cdot{\bm n}\}\|^2_{L^2(F)}\Big)^{1/2},
 \end{align}
 where $\Pi_k$ denotes the standard $L^2$-orthogonal projection onto $V_h^{{\rm DG}}$ on each $K\in\mathcal{T}_h$.

(2) Let $({\bm q},u)\in (H^1(\Omega))^3\times H_0^1(\Omega)\cap H^2(\Omega)$ and $({\bm q}_h,u_h,\widehat{u}_h)\in{\bm W}_h\times V_h\times \widehat{V}_h$ be the solutions of problems (\ref{mixed}) and (\ref{HDG}), then we have
\begin{align}
&\|{\bm q}-{\bm q}_h\|_{L^2(\Omega)}+\|\nabla(u-u_h)\|_{L^2_h(\Omega)}+\Big(\sum_{K\in\mathcal{T}_h}\|\sqrt{\tau}(\mathcal{L}u_h-\widehat{u}_h)\|^2_{L^2(\partial K)}\Big)^{1/2}\nonumber\\
\leq&C\Big(\|\nabla(u-\Pi_Vu)\|_{L_h^2(\Omega)}+\|{\bm q}-\Pi_W{\bm q}\|_{L^2(\Omega)}\Big)
+C\Big(\sum_{K\in\mathcal{T}_h}h_K\|({\bm q}-\Pi_W{\bm q})\cdot{\bm n}\|_{L^2(\partial K)}^2\Big)^{1/2},
\end{align}
where $\Pi_V$ and $\Pi_W$ denote the standard $L^2$-orthogonal projections onto $V(K)$ and ${\bm W}(K)$ on each $K\in\mathcal{T}_h$.
\end{lemma}
\begin{proof}
The result in (1) can be found in \cite[Theorem 4.17]{de2011}, while the results in (2) for the hybrid mixed DG, HDG-1 and HDG-2 methods can be found in
\cite{es2010, cfq2018, qs2016}, respectively.
\end{proof}

\begin{lemma}[Inverse inequality]\label{inv}
Let $D$ be an element of $\mathcal{T}_h$ or a face of $\mathcal{E}_h$, then the following estimates hold for $0\leq j\leq i\leq l+1$ and
$1\leq p, q\leq \infty$:
\begin{align}
|v_h|_{W^{i,p}(D)}\leq Ch_D^{-(i-j)+\frac{d_D}{p}-\frac{d_D}{q}}|v_h|_{W^{j,q}(D)},\quad \forall v_h\in \mathcal{P}^l(D),
\end{align}
where $h_D$ and $d_D$ denote the diameter and dimension of the domain $D$.
\end{lemma}
\begin{proof}
This lemma follows from \cite[Lemma 4.5.3]{bs2008}.
\end{proof}
\begin{lemma}[Trace inequality]\label{tra}
For each $K\in\mathcal{T}_h$ and $1<p<\infty$, we have
\begin{align}
\|v\|_{L^p(\partial K)}\leq C\left(h_K^{-\frac{1}{p}}\|v\|_{L^p(K)}+h_K^{\beta}|v|_{W^{\frac{1}{p}+\beta,p}(K)}\right),\quad \forall v\in W^{\frac{1}{p}+\beta,p}(K),
\end{align}
where $\beta\in (0,1-\frac{1}{p}]$ is a real number and $C$ is a positive constant depending on $\beta$ and $p$. For $\beta\in (0,1-\frac{1}{p})$,
the seminorm of $W^{\frac{1}{p}+\beta,p}(K)$ is defined as
\begin{align*}
|v|_{W^{\frac{1}{p}+\beta,p}(K)}^p=\int_K\int_K\frac{|v(x)-v(y)|^p}{|x-y|^{4+p\beta}} dxdy.
\end{align*}
\end{lemma}
\begin{proof}
For each polyhedral element $K$, let $\{K_i\}_{i=1}^{\mathcal{N}_K}$ denote the set of sub-elements, and set $F_i=\partial K\cap\partial K_i$ for $i=1,\cdots,\mathcal{N}_K$, where $\mathcal{N}_K\leq \mathcal{N}$ denotes the number of sub-elements in $K$. Let $\widehat{\mathcal{K}}$ be the reference element of $K_i$, and let the affine mapping be given by $\overline{K}_i\ni x={\bm B}_{K_i}\widehat{x}+{\bm b}_{K_i}$. Then, we define $\widehat{v}_i(\widehat{x})=v_i({\bm B}_{K_i}\widehat{x}+{\bm b}_{K_i})=v_i(x)$ and have
\begin{align*}
\|v_i\|_{L^p(F_i)}\leq Ch^{\frac{2}{p}}_K\|\widehat{v}_i\|_{L^p(\widehat{F}_i)},
\end{align*}
and
\begin{align*}
\|\widehat{v}_i\|_{L^p(\widehat{\mathcal{K}})}\leq Ch_K^{-\frac{3}{p}}\|v_i\|_{L^p(K_i)},\quad |\widehat{v}_i|_{W^{\frac{1}{p}+\beta,p}(\widehat{\mathcal{K}})}
\leq Ch_K^{\beta-\frac{2}{p}}|v_i|_{W^{\frac{1}{p}+\beta,p}(K_i)},
\end{align*}
for $i=1,\cdots,\mathcal{N}_K$, where $v_i$ denotes the restriction of $v$ in $K_i$. The above three estimates, together with the standard trace inequality in the reference element $\widehat{\mathcal{K}}$, lead to
\begin{align*}
\|v\|_{L^p(\partial K)}^p=& \sum_{i=1}^{\mathcal{N}_K}\|v_i\|_{L^p(F_i)}^p\leq C\sum_{i=1}^{\mathcal{N}_K}h_K^2\|\widehat{v}_i\|_{L^p(\widehat{F}_i)}^p\\
\leq& C\sum_{i=1}^{\mathcal{N}_K}h_K^2\|\widehat{v}_i\|_{W^{\frac{1}{p}+\beta,p}(\widehat{\mathcal{K}})}^p\\
\leq& C\sum_{i=1}^{\mathcal{N}_K}\left(h_K^{-1}\|v_i\|_{L^p(K_i)}^p+h_K^{p\beta}|v_i|_{W^{\frac{1}{p}+\beta,p}(K_i)}^p\right)\\
\leq& C\left(h_K^{-1}\|v\|_{L^p(K)}^p+h_K^{p\beta}|v|_{W^{\frac{1}{p}+\beta,p}(K)}^p\right),
\end{align*}
where $C$ is a positive constant depending on $\beta$ and $p$. We conclude the proof.
\end{proof}

\section{Main results for IPDG and HDG methods}\label{sec4}
For any given constant $\epsilon\in (0,\frac{1}{2})$ and $\frac{1}{1-\epsilon}<p<\frac{1}{\epsilon}$, the goal of this paper is to establish optimal discrete $W^{1,p}$ error estimates for the symmetric IPDG method, the hybrid mixed DG method, and two representative HDG methods for the Poisson problem on a convex polyhedral domain $\Omega\subset\mathbb{R}^3$. More specifically, we establish the following theorem for the IPDG method by using the interface Green's functions developed in \cite{lq2026}, the Fubini-type weighted estimates presented in Lemma \ref{fubini1}, and the discrete $W^{1,q}$-$W^{1,p}$ inf-sup stability established in Lemma \ref{inf-sup}, where $q$ is the conjugate number of $p$ such that $\frac{1}{p}+\frac{1}{q}=1$. Note that, in the proof, the optimal discrete $W^{1,p}$ error estimates for $2\leq p<\frac{1}{\epsilon}$ are first established using the interface Green's functions and the Fubini-type weighted estimates. These error estimates are then used to obtain the discrete $W^{1,q}$-$W^{1,p}$ inf-sup stability for $2\leq q<\frac{1}{\epsilon}$. Using this stability result and a suitably constructed dual problem, we extend the optimal error estimates to $\frac{1}{1-\epsilon}<p<2$ and subsequently establish the corresponding inf-sup stability for $\frac{1}{1-\epsilon}<q<2$. The detailed proofs are given in Subsection \ref{sec5:1}.

\begin{theorem}\label{W1p-IPDG}
For any given constant $\epsilon\in(0,\frac{1}{2})$ and $\frac{1}{1-\epsilon}<p<\frac{1}{\epsilon}$,
let $u\in W^{1,p}_0(\Omega)\cap W^{1+\frac{1}{p}+\epsilon,p}(\Omega)$ be the solution of problem (\ref{model}), and let $u_h^{{\rm DG}}$ be its IPDG approximation defined in (\ref{IPDG}). Then, we have
\begin{align}
|||u-u_h^{{\rm DG}}|||_{W_h^{1,p}(\Omega)}\leq C\Big(\|\nabla(u-\Pi_ku)\|_{L_h^p(\Omega)}
+h^{\frac{1}{p}+\epsilon}|\nabla(u-\Pi_ku)|_{W_h^{\frac{1}{p}+\epsilon,p}(\Omega)}\Big),\label{W1p-IPDG:1}
\end{align}
where the mesh-dependent norm $|||v|||_{W_h^{1,p}(\Omega)}$ and seminorm $|\nabla v|_{W_h^{\frac{1}{p}+\epsilon,p}(\Omega)}$
are defined by
\begin{align*}
|||v|||_{W_h^{1,p}(\Omega)}^p=\|\nabla v\|_{L_h^p(\Omega)}^p+\sum_{F\in\mathcal{E}_h}h_F^{1-p}\|[v]\|_{L^p(F)}^p,
\end{align*}
and
\begin{align*}
|\nabla v|_{W_h^{\frac{1}{p}+\epsilon,p}(\Omega)}^p=\sum_{K\in\mathcal{T}_h}|\nabla v|^p_{W^{\frac{1}{p}+\epsilon,p}(K)},\quad
|\nabla v|^p_{W^{\frac{1}{p}+\epsilon,p}(K)}=\int_K\int_K\frac{|v(x)-v(y)|^p}{|x-y|^{4+p\epsilon}}dxdy.
\end{align*}
\end{theorem}

Based on Theorem \ref{W1p-IPDG}, we derive the following corollary for the IPDG method.

\begin{corollary}\label{W1p-IPDG-2}
For any given constant $\epsilon\in(0,\frac{1}{2})$ and $\frac{1}{1-\epsilon}< p<\frac{1}{\epsilon}$,
let $u\in W^{1,p}_0(\Omega)\cap W^{1+\frac{1}{p}+\epsilon,p}(\Omega)$ be the solution of problem (\ref{model}), and let $u_h^{{\rm DG}}$ be its IPDG approximation defined in (\ref{IPDG}). Then, we have
\begin{align}
|||u-u_h^{{\rm DG}}|||_{W_{hh}^{1,p}(\Omega)}\leq& C\Big(\|\nabla(u-\Pi_ku)\|_{L_h^p(\Omega)}
+h^{\frac{1}{p}+\epsilon}|\nabla(u-\Pi_ku)|_{W_h^{\frac{1}{p}+\epsilon,p}(\Omega)}\Big),\label{W1p-IPDG-2:1}
\end{align}
where the mesh-dependent norm $|||v|||_{W_{hh}^{1,p}(\Omega)}$ is defined by
\begin{align*}
|||v|||_{W_{hh}^{1,p}(\Omega)}^p=|||v|||_{W_h^{1,p}(\Omega)}^p+\sum_{F\in\mathcal{E}_h}h_F\|\{\nabla v\cdot{\bm n}\}\|^p_{L^p(F)},
\end{align*}
and the mesh-dependent seminorm $|\cdot|_{W_h^{\frac{1}{p}+\epsilon,p}(\Omega)}$ is defined in Theorem \ref{W1p-IPDG}.
\end{corollary}
\begin{proof}
By Lemma \ref{inv} and Lemma \ref{tra} with $\beta=\epsilon$, we have
\begin{align*}
\Big(\sum_{F\in\mathcal{E}_h}h_F\|\{\nabla (u-u_h^{{\rm DG}})\cdot{\bm n}\}\|^p_{L^p(F)}\Big)^{1/p}\leq& C\left(\|\nabla(u-u_h^{{\rm DG}})\|_{L^p_h(\Omega)}
+h^{\frac{1}{p}+\epsilon}|\nabla(u-u_h^{{\rm DG}})|_{W_h^{\frac{1}{p}+\epsilon,p}(\Omega)}\right)\\
\leq&C\Big(\|\nabla(u-u_h^{{\rm DG}})\|_{L^p_h(\Omega)}+\|\nabla(u-\Pi_ku)\|_{L_h^p(\Omega)}\\
&+h^{\frac{1}{p}+\epsilon}|\nabla(u-\Pi_ku)|_{W_h^{\frac{1}{p}+\epsilon,p}(\Omega)}\Big),
\end{align*}
which, together with Theorem \ref{W1p-IPDG}, yields the desired result (\ref{W1p-IPDG-2:1}).
\end{proof}

Next, we establish optimal discrete $W^{1,p}$ error estimates for the hybrid mixed DG method and two representative HDG methods in the following theorem.
Note that the proof strategy is very similar with that for the IPDG method, but the technical details differ considerably because the IPDG method is based on the primary formulation, whereas the hybrid mixed DG and HDG methods are based on the mixed formulation. The detailed proofs are given in Subsection \ref{sec5:2}.

\begin{theorem}\label{W1p-HDG-pre}
For any given $\epsilon\in (0,\frac{1}{2})$ and $\frac{1}{1-\epsilon}< p<\frac{1}{\epsilon}$, let $(u,{\bm q})\in W^{1+\frac{1}{p}+\epsilon,p}(\Omega)\cap W_0^{1,p}(\Omega)
\times (W^{\frac{1}{p}+\epsilon,p}(\Omega))^3$ be the solution of problem (\ref{mixed}), and let $({\bm q}_h,u_h,\widehat{u}_h)$ be its HDG approximation defined in (\ref{HDG}). Then, we have
\begin{align}
&\|{\bm q}-{\bm q}_h\|_{L^p(\Omega)}+|||u-u_h|||_{W_h^{1,p}(\Omega)}\nonumber\\
\leq& C\Big(\|\nabla(u-\Pi_Vu)\|_{L_h^p(\Omega)}+\|{\bm q}-\Pi_W{\bm q}\|_{L^p(\Omega)}+h^{\frac{1}{p}+\epsilon}|{\bm q}-\Pi_W{\bm q}|_{W_h^{\frac{1}{p}+\epsilon,p}(\Omega)}\Big),\label{W1p-HDG-pre:1}
\end{align}
where the mesh-dependent norm $|||\cdot|||_{W_h^{1,p}(\Omega)}$ and seminorm $|\cdot|_{W_h^{\frac{1}{p}+\epsilon,p}(\Omega)}$ are defined in Theorem \ref{W1p-IPDG}.
\end{theorem}

\begin{remark}
For the HDG-1 method, if let $\Pi_V$ be a projection operator from $H^1(\Omega)$ onto $V_h$ such that for each element $K\in\mathcal{T}_h$,
\begin{subequations}\label{proj_HDG-1}
\begin{align}
(\Pi_Vu,v_h)_K=(u,v_h)_K,\quad \forall v_h\in\mathcal{P}^{k-1}(K),\\
\langle\Pi_Vu,\widehat{v}_h\rangle_{F_K^{\ast}}=\langle u,\widehat{v}_h\rangle_{F_K^{\ast}},\quad \forall\widehat{v}_h\in\mathcal{P}^k(F_K^{\ast}),
\end{align}
\end{subequations}
where $F_K^{\ast}$ is a fixed face of the element $K$ such that $K$ lies on one side of the hyperplane containing $F_K^{\ast}$.
Then, for the HDG-1 method, the estimate for $\|{\bm q}-{\bm q}_h\|_{L^p(\Omega)}$ in (\ref{W1p-HDG-pre:1}) can be improved to
\begin{align}
\|{\bm q}-{\bm q}_h\|_{L^p(\Omega)}\leq C\left(\|{\bm q}-\Pi_W{\bm q}\|_{L^p(\Omega)}+h^{\frac{1}{p}+\epsilon}|{\bm q}-\Pi_W{\bm q}|_{W_h^{\frac{1}{p}+\epsilon,p}}\right),\label{W1p-HDG-pre:2}
\end{align}
where $\frac{1}{1-\epsilon}< p<\frac{1}{\epsilon}$ and $\epsilon\in (0,\frac{1}{2})$ is any given constant.
In fact, the projection defined in \cite[($4.8$)]{cfq2018} is the vectorization of $\Pi_{V}$ defined in (\ref{proj_HDG-1}).
Moreover, we refer to \cite[Proposition~$4.3$]{cfq2018} for the approximation properties of the projection operator $\Pi_{V}$.
A brief proof of (\ref{W1p-HDG-pre:2}) is given in Appendix \ref{apendixC:1}.
\end{remark}

With Theorem \ref{W1p-HDG-pre} at hand, we derive the following two corollaries for the hybrid mixed DG method and two representative HDG methods.
The detailed proofs are given in Appendix \ref{apendixC:2} and Appendix \ref{apendixC:3}, respectively.

\begin{corollary}\label{W1p-HDG-c}
For any given $\epsilon\in (0,\frac{1}{2})$ and $\frac{1}{1-\epsilon}< p<\frac{1}{\epsilon}$, let $(u,{\bm q})\in W^{1+\frac{1}{p}+\epsilon,p}(\Omega)\cap W_0^{1,p}(\Omega)
\times (W^{\frac{1}{p}+\epsilon,p}(\Omega))^3$ be the solution of problem (\ref{mixed}), and let $({\bm q}_h,u_h,\widehat{u}_h)$ be its HDG approximation defined in (\ref{HDG}). Then, we have
\begin{align}
&\|{\bm q}-{\bm q}_h\|_{L^p(\Omega)}+\|\nabla(u-u_h)\|_{L_h^p(\Omega)}+\Big(\sum_{K\in\mathcal{T}_h}\tau^{p-1}\|\mathcal{L}u_h-\widehat{u}_h\|_{L^p(\partial K)}^p\Big)^{1/p}\nonumber\\
\leq& C\Big(\|\nabla(u-\Pi_Vu)\|_{L_h^p(\Omega)}+\|{\bm q}-\Pi_W{\bm q}\|_{L^p(\Omega)}+h^{\frac{1}{p}+\epsilon}|{\bm q}-\Pi_W{\bm q}|_{W_h^{\frac{1}{p}+\epsilon,p}(\Omega)}\Big),\label{W1p-HDG-c:1}
\end{align}
where the mesh-dependent seminorm $|\cdot|_{W_h^{\frac{1}{p}+\epsilon,p}(\Omega)}$ is defined in Theorem \ref{W1p-IPDG}.
\end{corollary}

\begin{corollary}\label{W1p-HDG-c1}
For any given $\epsilon\in (0,\frac{1}{2})$ and $\frac{1}{1-\epsilon}< p<\frac{1}{\epsilon}$, let $(u,{\bm q})\in W^{1+\frac{1}{p}+\epsilon,p}(\Omega)\cap W_0^{1,p}(\Omega)
\times (W^{\frac{1}{p}+\epsilon,p}(\Omega))^3$ be the solution of problem (\ref{mixed}), and let $({\bm q}_h,u_h,\widehat{u}_h)$ be its HDG approximation defined in (\ref{HDG}). Then, we have
\begin{align}
|||({\bm q}-{\bm q}_h,u-u_h,u-\widehat{u}_h)|||_{W_{h^{\prime}}^{1,p}(\Omega)}
\leq& C\Big(\|\nabla(u-\Pi_Vu)\|_{L_h^p(\Omega)}+\|\nabla(u-\Pi_ku)\|_{L_h^p(\Omega)}\nonumber\\
&+\|{\bm q}-\Pi_W{\bm q}\|_{L^p(\Omega)}+h^{\frac{1}{p}+\epsilon}|{\bm q}-\Pi_W{\bm q}|_{W_h^{\frac{1}{p}+\epsilon,p}(\Omega)}\Big),\label{W1p-HDG-c1:1}
\end{align}
where the mesh-dependent norm $|||({\bm w},v,\widehat{v})|||_{W_{h^{\prime}}^{1,p}(\Omega)}$ is defined by
\begin{align*}
|||({\bm w},v,\widehat{v})|||_{W_{h^{\prime}}^{1,p}(\Omega)}^p=\|{\bm w}\|_{L^p(\Omega)}^p+\|\nabla v\|_{L_h^p(\Omega)}^p
+\sum_{K\in\mathcal{T}_h}h_K^{1-p}\|v-\widehat{v}\|_{L^p(\partial K)}^p,
\end{align*}
and the mesh-dependent seminorm $|\cdot|_{W_h^{\frac{1}{p}+\epsilon,p}(\Omega)}$ is defined in Theorem \ref{W1p-IPDG}.
\end{corollary}

\section{Proofs of Theorem \ref{W1p-IPDG} and Theorem \ref{W1p-HDG-pre}}\label{sec5}
In this section, we are devoted to proving Theorem \ref{W1p-IPDG} and Theorem \ref{W1p-HDG-pre}. To this end,
for $z\in\Omega$, let $\delta_z$ be the Dirac delta function such that $(\delta_z(x), v(x))=v(z)$ for any $v(x)\in\mathcal{C}(\Omega)$. By virtue of $\delta_z$, we define the Green's function by
\begin{align}
-\Delta G(z,x)=\delta_z\quad {\rm in}~\Omega,\quad G(z,x)=0\quad {\rm on}~\partial\Omega.\label{G}
\end{align}
From \cite[Theorem 2]{c1985}, we know that for each $z\in\Omega$, the above problem has a unique weak solution $G(z,x)\in W_0^{1,q}(\Omega)$ for $1\leq q<\frac{3}{2}$. Moreover, from \cite[Theorem 5.1.1 and Theorem 5.1.8]{mr2010}, \cite[Lemma 3.2]{dlsw2012}, \cite[Theorem 1]{glrs2009}, and
\cite[Proposition 4.1]{drs2024} we have
\begin{align}
\left|\frac{\partial^{\alpha+\beta}G(z,x)}{\partial z^{\alpha}\partial x^{\beta}}\right|\leq C|x-z|^{-1-|\alpha|-|\beta|},\quad 0\leq|\alpha+\beta|\leq 2,~
0\leq |\alpha|,|\beta|\leq 1,\label{pe-G}
\end{align}
for $z\neq x$, and
\begin{align}
\frac{|\partial_{z_i}G(z,x)-\partial_{y_i}G(y,x)|}{|z-y|^{\gamma}}\leq C\left(|x-z|^{-2-\gamma}+|x-y|^{-2-\gamma}\right),\label{pe-G:1}
\end{align}
for all $x,y,z\in \Omega$ and $y\neq z$, where $\gamma\in(0,1)$ is a constant depending on the domain $\Omega$ and $\partial_{z_i}$ denotes the partial derivative with respect to the $i$-th component of $z$.

In addition to the Dirac delta function $\delta_z$, for each $z\in \overline{K}_z$ with $K_z\in\mathcal{T}_h$, we define a smooth discrete delta function $\delta_{h,z}$ supported in $K_z$ (cf. \cite[Appendix A]{sw1995}), which satisfies
\begin{align}
\chi(z)=\int_{K_z}\delta_{h,z}(x)\chi(x)dx\quad \forall \chi\in \mathcal{P}^m(K_z),\quad \int_{K_z}\delta_{h,z}(x)dx=1,\label{ddd:1}
\end{align}
and
\begin{align}
|\delta_{h,z}|_{W^{l,p}(\Omega)}\leq Ch^{-l-3\left(1-\frac{1}{p}\right)}\quad \forall 1\leq p\leq \infty,~l=0,1,\label{ddd:2}
\end{align}
where $m=k$ for the IPDG, hybrid mixed DG, and HDG-1 methods, and $m=k+1$ for the HDG-2 method.

\subsection{Proof of Theorem \ref{W1p-IPDG}}\label{sec5:1}
For $z\in K_z$ with $K_z\in\mathcal{T}_h$, we define the following regularized Green's function:
\begin{align}
-\Delta\Gamma(z,x)=\nabla\cdot({\bm e}_i\delta_{h,z})\quad {\rm in}~\Omega,\quad \Gamma(z,x)=0\quad {\rm on}~\partial\Omega,\label{rG}
\end{align}
and let $\Gamma_h^{{\rm DG}}\in V_h^{{\rm DG}}$ be the IPDG approximation of $\Gamma$ such that
\begin{align}
a_h(\Gamma_h^{{\rm DG}},v_h)=(\nabla\cdot({\bm e}_i\delta_{h,z}),v_h)_{\mathcal{T}_h},\quad \forall v_h\in V_h^{{\rm DG}},\label{rG-IPDG}
\end{align}
where ${\bm e}_i$ denotes the $i$-th standard basis vector of $\mathbb{R}^3$.
Then, using  Lemma \ref{dfp}, (\ref{ddd:2}), and the definitions of $a_h$ and $\delta_{h,z}$, we have for $2\leq p<\frac{1}{\epsilon}$
\begin{align}
&-({\bm e}_i\cdot\nabla(\Pi_ku-u_h^{{\rm DG}}))(z)=(\nabla\cdot({\bm e}_i\delta_{h,z}),\Pi_ku-u_h^{{\rm DG}})_{\mathcal{T}_h}\nonumber\\
=&a_h(\Gamma_h^{{\rm DG}}-\Gamma,\Pi_ku-u)+a_h(\Gamma,\Pi_ku-u)\nonumber\\
=&(\nabla(\Gamma_h^{{\rm DG}}-\Gamma),\nabla(\Pi_ku-u))_{\mathcal{T}_h}-\sum_{F\in\mathcal{E}_h}\int_F\Big(\{\nabla(\Gamma_h^{{\rm DG}}-\Gamma)\cdot{\bm n}\}[\Pi_ku-u]\nonumber\\
&+\{\nabla(\Pi_ku-u)\cdot{\bm n}\}[\Gamma_h^{{\rm DG}}-\Gamma]-\lambda h_F^{-1}[\Gamma_h^{{\rm DG}}-\Gamma][\Pi_ku-u]\Big) ds\label{rG:1}\\
&+(\nabla\cdot({\bm e}_i\delta_{h,z}),\Pi_ku-u)\nonumber\\
\leq&C\Big(\Big(\sum_{F\in\mathcal{E}_h}h_F^{1-p}\|\sigma_z^s[\Pi_ku-u]\|_{L^p(F)}^p+
h_F\|\sigma_z^s\{\nabla(\Pi_ku-u)\cdot{\bm n}\}\|_{L^p(F)}^p\Big)^{1/p}\nonumber\\
&+\|\sigma_z^s\nabla(\Pi_ku-u)\|_{L_h^p(\Omega)}\Big)\Big(h^{-\frac{3}{p}}+\|\sigma_z^{-s}\nabla(\Gamma_h^{{\rm DG}}-\Gamma)\|_{L^q_h(\Omega)}\nonumber\\
&+\Big(\sum_{F\in\mathcal{E}_h}h_F\|\sigma_z^{-s}\{\nabla(\Gamma_h^{{\rm DG}}-\Gamma)\cdot{\bm n}\}\|_{L^q(F)}^q+h_F^{1-q}\|\sigma_z^{-s}[\Gamma_h^{{\rm DG}}-\Gamma]\|_{L^q(F)}^q\Big)^{1/q}\Big),\nonumber
\end{align}
where $q$ is the conjugate number of $p$ such that $\frac{1}{p}+\frac{1}{q}=1$, $s$ is any positive real number, and $\sigma_z$ is a weight function defined by
\begin{align}
\sigma_z(x)=\frac{h}{|x-z|+h}.\label{swf}
\end{align}

For each $x_0\in F$ with $F\in\mathcal{E}_h^o$, let $K^+$ and $K^-$ be the elements sharing the common face $F$. Then, we define the following interface
Green's function:
\begin{align}
-\Delta\Gamma_0(x_0,x)=\delta_{h,K^+}-\delta_{h,K^{-}}\quad {\rm in}~\Omega,\quad \Gamma_0(x_0,x)=0\quad {\rm on}~\partial\Omega,\label{rG0}
\end{align}
and let $\Gamma_{0,h}^{{\rm DG}}\in V_h^{{\rm DG}}$ be its IPDG approximation such that
\begin{align}
a_h(\Gamma_{0,h}^{{\rm DG}},v_h)=(\delta_{h,K^+}-\delta_{h,K^{-}},v_h)_{\mathcal{T}_h},\quad \forall v_h\in V_h^{{\rm DG}},\label{rG0-IPDG}
\end{align}
where $\delta_{h,K^+}$ and $\delta_{h,K^-}$ denote the discrete delta function $\delta_{h,x_0}$ supported in $K^+$ and $K^-$, respectively.
Using  Lemma \ref{dfp}, (\ref{ddd:2}), (\ref{rG0}), and (\ref{rG0-IPDG}), we obtain for $2\leq p<\frac{1}{\epsilon}$
\begin{align}
&[\Pi_ku-u_h^{{\rm DG}}](x_0)=(\delta_{h,K^+}-\delta_{h,K^{-}},\Pi_ku-u_h^{{\rm DG}})_{\mathcal{T}_h}\nonumber\\
=&a_h(\Gamma_{0,h}^{{\rm DG}}-\Gamma_0,\Pi_ku-u)+a_h(\Gamma_0,\Pi_ku-u)\nonumber\\
\leq&C\Big(\Big(\sum_{F\in\mathcal{E}_h}h_F^{1-p}\|\sigma_{x_0}^s[\Pi_ku-u]\|_{L^p(F)}^p+
h_F\|\sigma_{x_0}^s\{\nabla(\Pi_ku-u)\cdot{\bm n}\}\|_{L^p(F)}^p\Big)^{1/p}\label{rG0:1}\\
&+\|\sigma_{x_0}^s\nabla(\Pi_ku-u)\|_{L_h^p(\Omega)}+h^{-1}\|\sigma_{x_0}^s(\Pi_ku-u)\|_{L^p(\Omega)}\Big)
\Big(h^{1-\frac{3}{p}}+\|\sigma_{x_0}^{-s}\nabla(\Gamma_{0,h}^{{\rm DG}}-\Gamma_0)\|_{L^q_h(\Omega)}\nonumber\\
&+\Big(\sum_{F\in\mathcal{E}_h}h_F\|\sigma_{x_0}^{-s}\{\nabla(\Gamma_{0,h}^{{\rm DG}}-\Gamma_0)\cdot{\bm n}\}\|_{L^q(F)}^q+h_F^{1-q}\|\sigma_{x_0}^{-s}[\Gamma_{0,h}^{{\rm DG}}-\Gamma_0]\|_{L^q(F)}^q\Big)^{1/q}\Big).\nonumber
\end{align}

Next, for each $x_{\ast}\in F$ with $F\in\mathcal{E}_h^{\partial}$, let $K_{\ast}$ be the element including the face $F$. Then, we define the following interface Green's function:
\begin{align}
-\Delta\Gamma_{\ast}(x_{\ast},x)=\delta_{h,x_{\ast}}\quad {\rm in}~\Omega,\quad \Gamma_{\ast}(x_{\ast},x)=0\quad {\rm on}~\partial\Omega,\label{rG1}
\end{align}
and let $\Gamma_{\ast,h}^{{\rm DG}}\in V_h^{{\rm DG}}$ be its IPDG approximation such that
\begin{align}
a_h(\Gamma_{\ast,h}^{{\rm DG}},v_h)=(\delta_{h,x_{\ast}},v_h)_{\mathcal{T}_h},\quad \forall v_h\in V_h^{{\rm DG}}.\label{rG1-IPDG}
\end{align}
Similarly, using Lemma \ref{dfp}, (\ref{ddd:2}), (\ref{rG1}), and (\ref{rG1-IPDG}), we have for $2\leq p<\frac{1}{\epsilon}$
\begin{align}
&[\Pi_ku-u_h^{{\rm DG}}](x_{\ast})=(\delta_{h,x_{\ast}},\Pi_ku-u_h^{{\rm DG}})_{\mathcal{T}_h}\nonumber\\
=&a_h(\Gamma_{\ast,h}^{{\rm DG}}-\Gamma_{\ast},\Pi_ku-u)+a_h(\Gamma_{\ast},\Pi_ku-u)\nonumber\\
\leq&C\Big(\Big(\sum_{F\in\mathcal{E}_h}h_F^{1-p}\|\sigma_{x_{\ast}}^s[\Pi_ku-u]\|_{L^p(F)}^p+
h_F\|\sigma_{x_{\ast}}^s\{\nabla(\Pi_ku-u)\cdot{\bm n}\}\|_{L^p(F)}^p\Big)^{1/p}\label{rG1:1}\\
&+\|\sigma_{x_{\ast}}^s\nabla(\Pi_ku-u)\|_{L_h^p(\Omega)}+h^{-1}\|\sigma_{x_{\ast}}^s(\Pi_ku-u)\|_{L^p(\Omega)}\Big)
\Big(h^{1-\frac{3}{p}}+\|\sigma_{x_{\ast}}^{-s}\nabla(\Gamma_{\ast,h}^{{\rm DG}}-\Gamma_{\ast})\|_{L^q_h(\Omega)}\nonumber\\
&+\Big(\sum_{F\in\mathcal{E}_h}h_F\|\sigma_{x_{\ast}}^{-s}\{\nabla(\Gamma_{\ast,h}^{{\rm DG}}-\Gamma_{\ast})\cdot{\bm n}\}\|_{L^q(F)}^q+h_F^{1-q}\|\sigma_{x_{\ast}}^{-s}[\Gamma_{\ast,h}^{{\rm DG}}-\Gamma_{\ast}]\|_{L^q(F)}^q\Big)^{1/q}\Big).\nonumber
\end{align}

To further estimate (\ref{rG:1}), (\ref{rG0:1}) and (\ref{rG1:1}), we introduce a dyadic decomposition of the domain $\Omega$. Without loss of generality, we assume that ${\rm diam}(\Omega)\leq 1$. Then, for a fixed point $z\in\overline{\Omega}$, we define
\begin{align*}
\Omega_{\ast}(z)=\{x\in\Omega: |x-z|\leq Mh\},\quad \Omega_j(z)=\{x\in\Omega: d_{j+1}\leq |x-z|\leq d_j\},
\end{align*}
and
\begin{align*}
\Omega_j^1(z)=\Omega_{j-1}(z)\cup\Omega_j(z)\cup\Omega_{j+1}(z),\\
\Omega_j^2(z)=\Omega_{j-2}(z)\cup\Omega_j^1(z)\cup\Omega_{j+2}(z),\\
\Omega_j^3(z)=\Omega_{j-3}(z)\cup\Omega_j^2(z)\cup\Omega_{j+3}(z),\\
\Omega_j^4(z)=\Omega_{j-4}(z)\cup\Omega_j^3(z)\cup\Omega_{j+4}(z),
\end{align*}
where $d_j=2^{-j}$ for $j=0,1,\cdots,I$ with $I\thickapprox |\ln h|$ such that $d_{I+1}\leq Mh< d_I$. Here, $M\geq 3$ is a positive constant that will be determined later. Therefore,
\begin{align*}
\Omega=\Omega_{\ast}(z)\cup\bigcup_{j=0}^I\Omega_j(z).
\end{align*}
\begin{lemma}\label{rG-lee}
For $z\in K_z$ with $K_z\in\mathcal{T}_h$, let $\Gamma(z,x)$ be the solution of problem (\ref{rG}). Then, we have
\begin{align}
|\Gamma|_{H^2(\Omega_j^2(z))}\leq Cd_j^{-\frac{5}{2}}.\label{rG-lee:1}
\end{align}
Furthermore, for $x_0\in F$ with $F\in\mathcal{E}_h^o$ and $x_{\ast}\in F$ with $F\in\mathcal{E}_h^{\partial}$,
let $\Gamma_0(x_0,x)$ and $\Gamma_{\ast}(x_{\ast},x)$ be the solutions of problems (\ref{rG0}) and (\ref{rG1}). Then, we have
\begin{align}
|\Gamma_0|_{H^2(\Omega_j^2(x_0))}\leq Chd_j^{-\frac{5}{2}},\quad |\Gamma_{\ast}|_{H^2(\Omega_j^2(x_{\ast}))}\leq Chd_j^{-\frac{5}{2}}.\label{rG-lee:2}
\end{align}
\end{lemma}
\begin{proof}
The results presented in (\ref{rG-lee:1}) and (\ref{rG-lee:2}) have been proved in \cite[Lemma 3.7]{lc2026} and \cite[Lemma 4.6]{lq2026}, respectively.
\end{proof}

Based on Lemma \ref{rG-lee}, we now prove the following results.

\begin{lemma}\label{rG-L1-IPDG}
For $z\in K_z$ with $K_z\in\mathcal{T}_h$, let $\Gamma(z,x)$ and $\Gamma_h^{{\rm DG}}(z,x)$ be the solutions of problems (\ref{rG}) and (\ref{rG-IPDG}).
Then, we have
\begin{align}
&\Big(\sum_{F\in\mathcal{E}_h}h_F\|\sigma_z^{-s}\{\nabla(\Gamma_h^{{\rm DG}}-\Gamma)\cdot{\bm n}\}\|_{L^q(F)}^q+h_F^{1-q}\|\sigma_z^{-s}[\Gamma_h^{{\rm DG}}-\Gamma]\|_{L^q(F)}^q\Big)^{1/q}\nonumber\\
&+\|\sigma_z^{-s}\nabla(\Gamma_h^{{\rm DG}}-\Gamma)\|_{L^q_h(\Omega)}\leq Ch^{-\frac{3}{p}}.
\end{align}
Furthermore, for $x_0\in F$ with $F\in\mathcal{E}_h^o$ and $x_{\ast}\in F$ with $F\in\mathcal{E}_h^{\partial}$, let $\Gamma_0(x_{0},x)$ and $\Gamma_{\ast}(x_{\ast},x)$ be the solutions of problems (\ref{rG0}) and (\ref{rG1}), and let $\Gamma_{0,h}^{{\rm DG}}$ and $\Gamma_{\ast,h}^{{\rm DG}}$
be their IPDG approximations defined in (\ref{rG0-IPDG}) and (\ref{rG1-IPDG}). Then, we have
\begin{align}
&\Big(\sum_{F\in\mathcal{E}_h}h_F\|\sigma_{x_l}^{-s}\{\nabla(\Gamma_{l,h}^{{\rm DG}}-\Gamma_l)\cdot{\bm n}\}\|_{L^q(F)}^q+h_F^{1-q}\|\sigma_{x_l}^{-s}[\Gamma_{l,h}^{{\rm DG}}-\Gamma_l]\|_{L^q(F)}^q\Big)^{1/q}\nonumber\\
&+\|\sigma_{x_l}^{-s}\nabla(\Gamma_{l,h}^{{\rm DG}}-\Gamma_l)\|_{L^q_h(\Omega)}\leq Ch^{1-\frac{3}{p}},\quad {\rm for}~l\in\{0,\ast\}.
\end{align}
Here, $s$ is a real number such that $\frac{3}{p}<s<\gamma+\frac{3}{p}$ for $2\leq p<\frac{1}{\epsilon}$, $\gamma\in(0,1)$ is a positive constant depending on the domain $\Omega$, and $q$ is the conjugate number of $p$ such that $\frac{1}{p}+\frac{1}{q}=1$.
\end{lemma}
\begin{proof}
In the proof, we let $x_1=z$, $\Gamma_{1,h}^{{\rm DG}}=\Gamma_h^{{\rm DG}}$ and $\Gamma_1=\Gamma$. Then, for $l\in\{0,1,\ast\}$, according to the definition (\ref{swf}) of the weight function $\sigma_{x_l}(x)$, we have
\begin{align}
\sigma_{x_l}(x)=\frac{h}{|x-x_l|+h}\geq \frac{h}{Mh+h}=\frac{1}{M+1}\quad \forall x\in\Omega_{\ast}(x_l),\label{rG-L1-IPDG-proof:ad2}
\end{align}
and
\begin{align}
\sigma_{x_l}(x)=\frac{h}{|x-x_l|+h}\geq\frac{h}{d_j+h}\geq \frac{1}{2}hd_j^{-1},\quad \forall x\in\Omega_j(x_l).\label{rG-L1-IPDG-proof:ad3}
\end{align}
Consequently, the dyadic decomposition of the domain $\Omega$ and the above two estimates yield
\begin{align}
&\Big(\sum_{F\in\mathcal{E}_h}h_F\|\sigma_{x_l}^{-s}\{\nabla(\Gamma_{l,h}^{{\rm DG}}-\Gamma_l)\cdot{\bm n}\}\|_{L^q(F)}^q+h_F^{1-q}\|\sigma_{x_l}^{-s}[\Gamma_{l,h}^{{\rm DG}}-\Gamma_l]\|_{L^q(F)}^q\Big)^{1/q}\nonumber\\
&+\|\sigma_{x_l}^{-s}\nabla(\Gamma_{l,h}^{{\rm DG}}-\Gamma_l)\|_{L^q_h(\Omega)}\leq CM^s\Big(\|\nabla(\Gamma_{l,h}^{{\rm DG}}-\Gamma_l)\|_{L^q_h(\Omega_{\ast}(x_l))}+\label{rG-L1-IPDG-proof:1}\\
&\Big(\sum_{F\in\mathcal{E}_h}h_F\|\{\nabla(\Gamma_{l,h}^{{\rm DG}}-\Gamma_l)\cdot{\bm n}\}\|_{L^q(F\cap\Omega_{\ast}(x_l))}^q+h_F^{1-q}\|[\Gamma_{l,h}^{{\rm DG}}-\Gamma_l]\|_{L^q(F\cap\Omega_{\ast}(x_l))}^q\Big)^{1/q}\Big)\nonumber\\
&+C\sum_{j=0}^I(hd_j^{-1})^{-s}\Big(\|\nabla(\Gamma_{l,h}^{{\rm DG}}-\Gamma_l)\|_{L^q_h(\Omega_{j}(x_l))}+\Big(\sum_{F\in\mathcal{E}_h}
h_F\|\{\nabla(\Gamma_{l,h}^{{\rm DG}}-\Gamma_l)\cdot{\bm n}\}\|_{L^q(F\cap\Omega_j(x_l))}^q\nonumber\\
&+h_F^{1-q}\|[\Gamma_{l,h}^{{\rm DG}}-\Gamma_l]\|_{L^q(F\cap\Omega_j(x_l))}^q\Big)^{1/q}\Big)=J_1+J_2.\nonumber
\end{align}

For $J_1$, Lemma \ref{leee} and Lemma \ref{tra} with $\beta=1-\frac{1}{q}$ yield
\begin{align}
J_1\leq& CM^s\Big(\|\nabla(\Gamma_{l,h}^{{\rm DG}}-\Gamma_l)\|_{L_h^q(\mathcal{T}_h(\Omega_{\ast}(x_l)))}+h|\Gamma_{l,h}^{{\rm DG}}-\Gamma_l|_{W_h^{2,q}(\mathcal{T}_h(\Omega_{\ast}(x_l)))}\nonumber\\
&+\Big(\sum_{F\in\mathcal{E}_h}h_F^{1-q}\|[\Gamma_{l,h}^{{\rm DG}}-\Gamma_l]\|_{L^q(F\cap\Omega_{\ast}(x_l))}^q\Big)^{1/q}\Big)\nonumber\\
\leq& CM^s(Mh)^{3(\frac{1}{q}-\frac{1}{2})}\Big(\|\nabla(\Gamma_{l,h}^{{\rm DG}}-\Gamma_l)\|_{L_h^2(\Omega)}
+h|\Gamma_{l,h}^{{\rm DG}}-\Gamma_l|_{H_h^2(\Omega)}\nonumber\\
&+\Big(\sum_{F\in\mathcal{E}_h}h_F^{-1}\|[\Gamma_{l,h}^{{\rm DG}}-\Gamma_l]\|_{L^2(F)}^2\Big)^{1/2}\Big)
\leq Ch^{\frac{3}{q}-\frac{1}{2}}|\Gamma_l|_{H^2(\Omega)},\nonumber
\end{align}
where $\mathcal{T}_h(\Omega_{\ast}(x_l))=\{K\in\mathcal{T}_h: K\cap\Omega_{\ast}(x_l)\neq\emptyset\}$. Moreover, according to (\ref{ddd:2}) we have
\begin{align}\label{rG-L1-IPDG-proof:ad1}
|\Gamma_l|_{H^2(\Omega)}\leq C\begin{cases}
\|\nabla\cdot({\bm e}_i\delta_{h,z})\|_{L^2(\Omega)}\\
\|\delta_{h,K^+}-\delta_{h,K^-}\|_{L^2(\Omega)}\\
\|\delta_{h,x_{\ast}}\|_{L^2(\Omega)}
\end{cases}\leq C\begin{cases}
h^{-\frac{5}{2}},\quad & l=1,\\
h^{-\frac{3}{2}},\quad & l=0,\\
h^{-\frac{3}{2}},\quad & l=\ast.
\end{cases}
\end{align}
Therefore,
\begin{align}\label{rG-L1-IPDG-proof:2}
J_1\leq C\begin{cases}
h^{-\frac{3}{p}},\quad & l=1,\\
h^{1-\frac{3}{p}},\quad & l=0,\ast.
\end{cases}
\end{align}

As for $J_2$, using Lemma \ref{lee} and interpolation error estimates, we obtain
\begin{align}
J_2\leq& C\sum_{j=0}^I(hd_j^{-1})^{-s}d_j^{3(\frac{1}{q}-\frac{1}{2})}\Big(\|\nabla(\Gamma_{l,h}^{{\rm DG}}-\Gamma_l)\|_{L^2_h(\mathcal{T}_h(\Omega_{j}(x_l)))}\nonumber\\
&+\Big(\sum_{F\in\mathcal{E}_h\cap\mathcal{T}_h(\Omega_j(x_l))}
h_F\|\{\nabla(\Gamma_{l,h}^{{\rm DG}}-\Gamma_l)\cdot{\bm n}\}\|_{L^2(F)}^2+h_F^{-1}\|[\Gamma_{l,h}^{{\rm DG}}-\Gamma_l]\|_{L^2(F)}^2\Big)^{1/2}\Big)\nonumber\\
\leq&C\sum_{j=0}^I(hd_j^{-1})^{-s}d_j^{3(\frac{1}{q}-\frac{1}{2})}h|\Gamma_l|_{H^2(\Omega_j^2(x_l))}+C\sum_{j=0}^I(hd_j^{-1})^{-s}d_j^{3(\frac{1}{q}-\frac{1}{2})}d_j^{-1}
\|\Gamma_l-\Gamma_{l,h}^{{\rm DG}}\|_{L^2(\Omega_j^1(x_l))}.\nonumber
\end{align}
Moreover, from Lemma \ref{rG-lee} we know that
\begin{align*}
(hd_j^{-1})^{-s}d_j^{3(\frac{1}{q}-\frac{1}{2})}h|\Gamma_l|_{H^2(\Omega_j^2(x_l))}\leq C\begin{cases}
(hd_j^{-1})^{1+\frac{3}{p}-s}h^{-\frac{3}{p}},\quad & l=1,\\
(hd_j^{-1})^{1+\frac{3}{p}-s}h^{1-\frac{3}{p}},\quad & l=0,\ast.
\end{cases}
\end{align*}
Hence,
\begin{align}\label{rG-L1-IPDG-proof:3}
J_2\leq&C\sum_{j=0}^I(hd_j^{-1})^{-s}d_j^{3(\frac{1}{q}-\frac{1}{2})}\Big(\|\nabla(\Gamma_{l,h}^{{\rm DG}}-\Gamma_l)\|_{L^2_h(\mathcal{T}_h(\Omega_{j}(x_l)))}\nonumber\\
&+\Big(\sum_{F\in\mathcal{E}_h\cap\mathcal{T}_h(\Omega_j(x_l))}
h_F\|\{\nabla(\Gamma_{l,h}^{{\rm DG}}-\Gamma_l)\cdot{\bm n}\}\|_{L^2(F)}^2+h_F^{-1}\|[\Gamma_{l,h}^{{\rm DG}}-\Gamma_l]\|_{L^2(F)}^2\Big)^{1/2}\Big)\nonumber\\ \leq&C\sum_{j=0}^I(hd_j^{-1})^{-s}d_j^{3(\frac{1}{q}-\frac{1}{2})}d_j^{-1}
\|\Gamma_l-\Gamma_{l,h}^{{\rm DG}}\|_{L^2(\Omega_j^1(x_l))}+C\begin{cases}
h^{-\frac{3}{p}},\quad & l=1,\\
h^{1-\frac{3}{p}},\quad & l=0,\ast.
\end{cases}
\end{align}

To estimate $\|\Gamma_l-\Gamma_{l,h}^{{\rm DG}}\|_{L^2(\Omega^1_j(x_l))}$, let $W$ be the solution of the following problem:
\begin{align}
-\Delta W=\psi\quad {\rm in}~\Omega,\quad W=0\quad {\rm on}~\partial\Omega,\label{rG-L1-IPDG-proof:4}
\end{align}
for $\psi\in\mathcal{C}_0^{\infty}(\Omega_j^1(z))$ with $\|\psi\|_{L^2(\Omega_j^1(z))}=1$. Then, using Lemma \ref{dfp} we have
\begin{align*}
(\psi,\Gamma_l-\Gamma_{l,h}^{{\rm DG}})=a_h(W,\Gamma_l-\Gamma_{l,h}^{{\rm DG}})=a_h(W-\Pi_kW,\Gamma_l-\Gamma_{l,h}^{{\rm DG}}).
\end{align*}
On the subdomain $\Omega_j^3(x_l)$, the interpolation error estimates yield
\begin{align}
(\psi,\Gamma_l-\Gamma_{l,h}^{{\rm DG}})|_{\Omega_j^3(x_l)}\leq& Ch\Big(\Big(\sum_{F\in\mathcal{E}_h}\big(h_F
\|\{\nabla(\Gamma_l-\Gamma_{l,h}^{{\rm DG}})\cdot{\bm n}\}\|^2_{L^2(F\cap\Omega_j^3(x_l))}\nonumber\\
&+h_F^{-1}\|[\Gamma_l-\Gamma_{l,h}^{{\rm DG}}]\|_{L^2(F\cap\Omega_j^3(x_l))}^2\big)\Big)^{1/2}+\|\nabla(\Gamma_l-\Gamma_{l,h}^{{\rm DG}})\|_{L_h^2(\Omega_j^3(x_l))}\Big).\nonumber
\end{align}
On the subdomain $\Omega\backslash\Omega_j^3(z)$, \cite[Lemma 3.8]{lc2026} and the interpolation error estimates presented in \cite[Lemma 4.1]{glrs2009} yield
\begin{align}
(\psi,\Gamma_l-\Gamma_{l,h}^{{\rm DG}})|_{\Omega\backslash\Omega_j^3(x_l)}\leq& \Big(\|\nabla(W-\Pi_kW)\|_{L_h^{\infty}(\Omega\backslash\Omega_j^3(x_l))}+
h^{-1}\|W-\Pi_kW\|_{L^{\infty}(\Omega\backslash\Omega_j^3(x_l))}\Big)\nonumber\\
&\times\Big(\|\nabla(\Gamma_l-\Gamma_{l,h}^{{\rm DG}})\|_{L^1_h(\Omega)}+\sum_{F\in\mathcal{E}_h}\big(h_F\|\{\nabla(\Gamma_l-\Gamma_{l,h}^{{\rm DG}})\cdot{\bm n}\}\|_{L^1(F)}\nonumber\\
&+\|[\Gamma_l-\Gamma_{l,h}^{{\rm DG}}]\|_{L^1(F)}\big)\Big)\nonumber\\
\leq&Ch^{\gamma}d_j^{-\gamma-\frac{1}{2}}\Big(\|\nabla(\Gamma_l-\Gamma_{l,h}^{{\rm DG}})\|_{L^1_h(\Omega)}\nonumber\\
&+\sum_{F\in\mathcal{E}_h}\big(h_F\|\{\nabla(\Gamma_l-\Gamma_{l,h}^{{\rm DG}})\cdot{\bm n}\}\|_{L^1(F)}
+\|[\Gamma_l-\Gamma_{l,h}^{{\rm DG}}]\|_{L^1(F)}\big)\Big).\nonumber
\end{align}
Therefore, using the duality argument, (\ref{ddd:2}), and Lemma \ref{leee}, we obtain
\begin{align}
&\sum_{j=0}^I(hd_j^{-1})^{-s}d_j^{3(\frac{1}{q}-\frac{1}{2})}d_j^{-1}
\|\Gamma_l-\Gamma_{l,h}^{{\rm DG}}\|_{L^2(\Omega_j^1(x_l))}\nonumber\\
\leq& Ch^{\frac{3}{q}-\frac{1}{2}}|\Gamma_l|_{H^2(\Omega)}+C\sum_{j=0}^I(hd_j^{-1})^{-s}d_j^{3(\frac{1}{q}-\frac{1}{2})}hd_j^{-1}
\Big(\Big(\sum_{F\in\mathcal{E}_h}\big(h_F
\|\{\nabla(\Gamma_l-\Gamma_{l,h}^{{\rm DG}})\cdot{\bm n}\}\|^2_{L^2(F\cap\Omega_j(x_l))}\nonumber\\
&+h_F^{-1}\|[\Gamma_l-\Gamma_{l,h}^{{\rm DG}}]\|_{L^2(F\cap\Omega_j(x_l))}^2\big)\Big)^{1/2}+\|\nabla(\Gamma_l-\Gamma_{l,h}^{{\rm DG}})\|_{L_h^2(\Omega_j(x_l))}\Big)+
Ch^{-\frac{3}{p}}\sum_{j=0}^I(hd_j^{-1})^{\frac{3}{p}+\gamma-s}\nonumber\\
&\times\Big(\|\nabla(\Gamma_l-\Gamma_{l,h}^{{\rm DG}})\|_{L^1_h(\Omega)}
+\sum_{F\in\mathcal{E}_h}\big(h_F\|\{\nabla(\Gamma_l-\Gamma_{l,h}^{{\rm DG}})\cdot{\bm n}\}\|_{L^1(F)}
+\|[\Gamma_l-\Gamma_{l,h}^{{\rm DG}}]\|_{L^1(F)}\big)\Big),\nonumber
\end{align}
which, together with (\ref{rG-L1-IPDG-proof:ad1}) and (\ref{rG-L1-IPDG-proof:3}), leads to
\begin{align}
J_2\leq& Ch^{-\frac{3}{p}}\Big(\sum_{F\in\mathcal{E}_h}\big(h_F\|\{\nabla(\Gamma_l-\Gamma_{l,h}^{{\rm DG}})\cdot{\bm n}\}\|_{L^1(F)}+\|[\Gamma_l-\Gamma_{l,h}^{{\rm DG}}]\|_{L^1(F)}\big)
\label{rG-L1-IPDG-proof:5}\\
&+\|\nabla(\Gamma_l-\Gamma_{l,h}^{{\rm DG}})\|_{L^1_h(\Omega)}\Big)+C\begin{cases}
h^{-\frac{3}{p}},\quad &l=1,\\
h^{1-\frac{3}{p}},\quad & l=0,\ast,
\end{cases}\nonumber
\end{align}
by setting $M$ to be large enough.

Similar to the above proof that uses the dyadic decomposition of the domain $\Omega$ and the local energy estimates presented in Lemma \ref{rG-lee}, we can also prove
\begin{align*}
\|\nabla(\Gamma_l-\Gamma_{l,h}^{{\rm DG}})\|_{L^1_h(\Omega)}
+\sum_{F\in\mathcal{E}_h}\big(h_F\|\{\nabla(\Gamma_l-\Gamma_{l,h}^{{\rm DG}})\cdot{\bm n}\}\|_{L^1(F)}
+\|[\Gamma_l-\Gamma_{l,h}^{{\rm DG}}]\|_{L^1(F)}\big)\leq \begin{cases}
C,\quad & l=1,\\
Ch,\quad &l=0,\ast,
\end{cases}
\end{align*}
which, together with (\ref{rG-L1-IPDG-proof:1}), (\ref{rG-L1-IPDG-proof:2}) and (\ref{rG-L1-IPDG-proof:5}), concludes the proof.
Note that the above $L^1$ error estimates also can be found in \cite[Lemma 4.8 and Lemma A.2]{lq2026} for $l\in \{0,\ast\}$ and $l=1$ respectively.
\end{proof}

Based on (\ref{rG:1}), (\ref{rG0:1}), (\ref{rG1:1}), and Lemma \ref{rG-L1-IPDG}, for $2\leq p<\frac{1}{\epsilon}$ and $\frac{3}{p}<s<\gamma+\frac{3}{p}$, we have
\begin{align}
|\nabla(\Pi_ku-u_h^{{\rm DG}})(z)|\leq& Ch^{-\frac{3}{p}}\Big(\Big(\sum_{F\in\mathcal{E}_h}h_F^{1-p}\|\sigma_z^s[\Pi_ku-u]\|_{L^p(F)}^p+
h_F\|\sigma_z^s\{\nabla(\Pi_ku-u)\cdot{\bm n}\}\|_{L^p(F)}^p\Big)^{1/p}\nonumber\\
&+\|\sigma_z^s\nabla(\Pi_ku-u)\|_{L_h^p(\Omega)}\Big),\label{rG:2}
\end{align}
if $z\in K_z$ with $K_z\in\mathcal{T}_h$, and
\begin{align}
|[\Pi_ku-u_h^{{\rm DG}}](z)|\leq& Ch^{1-\frac{3}{p}}\Big(\Big(\sum_{F\in\mathcal{E}_h}h_F^{1-p}\|\sigma_{z}^s[\Pi_ku-u]\|_{L^p(F)}^p+
h_F\|\sigma_{z}^s\{\nabla(\Pi_ku-u)\cdot{\bm n}\}\|_{L^p(F)}^p\Big)^{1/p}\nonumber\\
&+\|\sigma_{z}^s\nabla(\Pi_ku-u)\|_{L_h^p(\Omega)}+h^{-1}\|\sigma_{z}^s(\Pi_ku-u)\|_{L^p(\Omega)}\Big),\label{rG:3}
\end{align}
if $z\in F$ with $F\in\mathcal{E}_h$. Then, using (\ref{rG:2}), (\ref{rG:3}), and Lemma \ref{fubini1}, we obtain
\begin{align*}
&\|\nabla(\Pi_ku-u_h^{\rm DG})\|_{L^p(\Omega)}^p+\sum_{F\in\mathcal{E}_h}h_F^{1-p}\|[\Pi_ku-u_h^{\rm DG}]\|_{L^p(F)}^p\\
=&\int_{\Omega}|\nabla(\Pi_ku-u_h^{\rm DG})(z)|^pdz+\sum_{F\in\mathcal{E}_h}h_F^{1-p}\int_F|[\Pi_ku-u_h^{\rm DG}](z)|^p dz\\
\leq&Ch^{-3}\int_{\Omega}\Big(\sum_{F\in\mathcal{E}_h}\big(h_F^{1-p}\|\sigma_z^s[\Pi_ku-u]\|_{L^p(F)}^p+
h_F\|\sigma_z^s\{\nabla(\Pi_ku-u)\cdot{\bm n}\}\|_{L^p(F)}^p\big)\nonumber\\
&+\|\sigma_z^s\nabla(\Pi_ku-u)\|_{L_h^p(\Omega)}^p\Big) dz+Ch^{-2}\sum_{F\in\mathcal{E}_h}\int_F\Big(\|\sigma_{z}^s\nabla(\Pi_ku-u)\|_{L_h^p(\Omega)}^p+h^{-p}\|\sigma_{z}^s(\Pi_ku-u)\|_{L^p(\Omega)}^p\nonumber\\
&+\sum_{F\in\mathcal{E}_h}\big(h_F^{1-p}\|\sigma_{z}^s[\Pi_ku-u]\|_{L^p(F)}^p+
h_F\|\sigma_{z}^s\{\nabla(\Pi_ku-u)\cdot{\bm n}\}\|_{L^p(F)}^p\big)\Big)dz\\
\leq& C\Big(h^{-p}\|u-\Pi_ku\|_{L^p(\Omega)}^p+\|\nabla(u-\Pi_ku)\|_{L_h^p(\Omega)}^p+h^{1+p\epsilon}|\nabla(u-\Pi_ku)|_{W^{\frac{1}{p}+\epsilon,p}_h(\Omega)}^p\Big),
\end{align*}
which, together with Lemma \ref{tra} with $\beta=1-\frac{1}{p}$ and the triangle inequality, yields (\ref{W1p-IPDG:1}) for the case $2\leq p<\frac{1}{\epsilon}$.

Next, we prove the case $\frac{1}{1-\epsilon}<p<2$. To this end, Let $v_h\in V_h^{{\rm DG}}$ be the solution of the following problem:
\begin{align*}
a_h(v_h,w_h)=&(|\nabla(\Pi_ku-u_h^{{\rm DG}})|^{p-2}\nabla(\Pi_ku-u_h^{{\rm DG}}),\nabla w_h)_{\mathcal{T}_h}\nonumber\\
&+\sum_{F\in\mathcal{E}_h}h_F^{1-p}\int_F|[\Pi_ku-u_h^{{\rm DG}}]|^{p-2}[\Pi_ku-u_h^{{\rm DG}}][w_h] ds,\quad \forall w_h\in V_h^{{\rm DG}}.
\end{align*}
Then, Lemma \ref{dfp}, the inf-sup condition (\ref{inf-sup-1}), and the H\"{o}lder's inequality yield
\begin{align*}
|||\Pi_ku-u_h^{{\rm DG}}|||_{W_h^{1,p}(\Omega)}^p=a_h(v_h,\Pi_ku-u_h^{\rm DG})=a_h(v_h,\Pi_ku-u),
\end{align*}
and
\begin{align*}
|||v_h|||_{W_h^{1,q}(\Omega)}\leq C|||\Pi_ku-u_h^{{\rm DG}}|||^{p-1}_{W_h^{1,p}(\Omega)}.
\end{align*}
Consequently, using Lemma \ref{tra} with $\beta=\epsilon$ and $\beta=1-\frac{1}{p}$, we have
\begin{align*}
&|||\Pi_ku-u_h^{{\rm DG}}|||_{W_h^{1,p}(\Omega)}\leq C\frac{a_h(v_h,\Pi_ku-u)}{|||v_h|||_{W_h^{1,q}(\Omega)}}\\
\leq& C\Big(\|\nabla(u-\Pi_ku)\|_{L_h^p(\Omega)}+h^{\frac{1}{p}+\epsilon}|\nabla(u-\Pi_ku)|_{W_h^{\frac{1}{p}+\epsilon,p}(\Omega)}\Big),
\end{align*}
which, together with Lemma \ref{tra} with $\beta=1-\frac{1}{p}$ and the triangle inequality, derives the result (\ref{W1p-IPDG:1}) for the case $\frac{1}{1-\epsilon}<p<2$. We conclude the proof.

\subsection{Proof of Theorem \ref{W1p-HDG-pre}}\label{sec5:2}
For $z\in K_z$ with $K_z\in\mathcal{T}_h$, let $\Gamma(z,x)$ be the regularized Green's function defined in (\ref{rG}). Then, by setting $\boldsymbol{\mathcal{R}}_1(z,x)=-\nabla\Gamma(z,x)$ and $\boldsymbol{\mathcal{R}}_2(z,x)=-\nabla\Gamma(z,x)-{\bm e}_i\delta_{h,z}$, we arrive at the following two mixed formulations:
\begin{align}
\boldsymbol{\mathcal{R}}_1+\nabla\Gamma=0\quad {\rm in}~\Omega,\quad \nabla\cdot\boldsymbol{\mathcal{R}}_1=\nabla\cdot({\bm e}_i\delta_{h,z})
\quad {\rm in}~\Omega,\quad \Gamma=0\quad {\rm on}~\partial\Omega,\label{rG-mixed1}
\end{align}
and
\begin{align}
\boldsymbol{\mathcal{R}}_2+\nabla\Gamma=-{\bm e}_i\delta_{h,z}\quad {\rm in}~\Omega,\quad \nabla\cdot\boldsymbol{\mathcal{R}}_2=0
\quad {\rm in}~\Omega,\quad \Gamma=0\quad {\rm on}~\partial\Omega.\label{rG-mixed2}
\end{align}
For the above two mixed formulations (\ref{rG-mixed1}) and (\ref{rG-mixed2}), we let $(\boldsymbol{\mathcal{R}}_{1,h},\Gamma_{1,h},\widehat{\Gamma}_{1,h})\in {\bm W}_h\times V_h\times\widehat{V}_h$ and $(\boldsymbol{\mathcal{R}}_{2,h},\Gamma_{2,h},\widehat{\Gamma}_{2,h})\in {\bm W}_h\times V_h\times\widehat{V}_h$ be their HDG approximations such that
\begin{align}
\mathcal{B}(\boldsymbol{\mathcal{R}}_{1,h},\Gamma_{1,h},\widehat{\Gamma}_{1,h}; {\bm w}_h,v_h,\widehat{v}_h)=-(\nabla\cdot({\bm e}_i\delta_{h,z}),v_h)_{\mathcal{T}_h},
\quad \forall ({\bm w}_h,v_h,\widehat{v}_h)\in {\bm W}_h\times V_h\times\widehat{V}_h,\label{rG-HDG1}
\end{align}
and
\begin{align}
\mathcal{B}(\boldsymbol{\mathcal{R}}_{2,h},\Gamma_{2,h},\widehat{\Gamma}_{2,h}; {\bm w}_h,v_h,\widehat{v}_h)=-({\bm e}_i\delta_{h,z},{\bm w}_h)_{\mathcal{T}_h},
\quad \forall ({\bm w}_h,v_h,\widehat{v}_h)\in {\bm W}_h\times V_h\times\widehat{V}_h.\label{rG-HDG2}
\end{align}
Then, using Lemma \ref{dfp}, (\ref{ddd:2}), and the definitions of $\mathcal{B}$ and $\delta_{h,z}$, we have for $2\leq p<\frac{1}{\epsilon}$
\begin{align}
&({\bm e}_i\cdot\nabla(\Pi_Vu-u_h))(z)=-(\nabla\cdot({\bm e}_i\delta_{h,z}),\Pi_Vu-u_h)_{\mathcal{T}_h}\nonumber\\
=&\mathcal{B}(\boldsymbol{\mathcal{R}}_{1,h}-\boldsymbol{\mathcal{R}}_{1},\Gamma_{1,h}-\Gamma,\widehat{\Gamma}_{1,h}-\Gamma;
\Pi_W{\bm q}-{\bm q},\Pi_Vu-u,\Pi_{\widehat{V}}u-u)\nonumber\\
&+\mathcal{B}(\boldsymbol{\mathcal{R}}_{1},\Gamma,\Gamma;\Pi_W{\bm q}-{\bm q},\Pi_Vu-u,\Pi_{\widehat{V}}u-u)\nonumber\\
=&(\boldsymbol{\mathcal{R}}_{1,h}-\boldsymbol{\mathcal{R}}_{1},\Pi_W{\bm q}-{\bm q})_{\mathcal{T}_h}+(\nabla(\Gamma_{1,h}-\Gamma),\Pi_W{\bm q}-{\bm q})_{\mathcal{T}_h}\nonumber\\
&-\langle\Gamma_{1,h}-\widehat{\Gamma}_{1,h},(\Pi_W{\bm q}-{\bm q})\cdot{\bm n}\rangle_{\partial\mathcal{T}_h}+(\boldsymbol{\mathcal{R}}_{1,h}-\boldsymbol{\mathcal{R}}_{1},\nabla(\Pi_Vu-u))_{\mathcal{T}_h}\label{rG-HDG1:1}\\
&-\langle(\boldsymbol{\mathcal{R}}_{1,h}-\boldsymbol{\mathcal{R}}_{1})\cdot{\bm n}, (\Pi_Vu-u)\rangle_{\partial\mathcal{T}_h}
-\langle\tau(\mathcal{L}\Gamma_{1,h}-\widehat{\Gamma}_{1,h}),\mathcal{L}(\Pi_Vu-u)\rangle_{\partial\mathcal{T}_h}\nonumber\\
&-(\nabla\cdot({\bm e}_i\delta_{h,z}),\Pi_Vu-u)_{\mathcal{T}_h}\nonumber\\
\leq&C\Big(\|\sigma_z^s(\Pi_W{\bm q}-{\bm q})\|_{L^p(\Omega)}+\|\sigma_z^s\nabla(\Pi_Vu-u)\|_{L_h^p(\Omega)}\nonumber\\
&+\Big(\sum_{K\in\mathcal{T}_h}h_K\|\sigma_z^s(\Pi_W{\bm q}-{\bm q})\cdot{\bm n}\|^p_{L^p(\partial K)}+h_K^{1-p}\|\sigma_z^s(\Pi_Vu-u)\|_{L^p(\partial K)}^p\Big)^{1/p}\Big)\nonumber\\
&\times\Big(h^{-\frac{3}{p}}+\|\sigma_z^{-s}(\boldsymbol{\mathcal{R}}_{1,h}-\boldsymbol{\mathcal{R}}_{1})\|_{L^q(\Omega)}
+\|\sigma_z^{-s}\nabla(\Gamma_{1,h}-\Gamma)\|_{L_h^q(\Omega)}\nonumber\\
&+\Big(\sum_{K\in\mathcal{T}_h}h_K\|\sigma_z^{-s}(\boldsymbol{\mathcal{R}}_{1,h}-\boldsymbol{\mathcal{R}}_{1})\cdot{\bm n})\|_{L^q(\partial K)}^q+h_K^{1-q}\|\sigma_z^{-s}(\Gamma_{1,h}-\widehat{\Gamma}_{1,h})\|_{L^q(\partial K)}^q\Big)^{1/q}\Big),\nonumber
\end{align}
and
\begin{align}
&-({\bm e}_i\cdot(\Pi_W{\bm q}-{\bm q}_h))(z)=-({\bm e}_i\delta_{h,z},\Pi_W{\bm q}-{\bm q}_h)_{\mathcal{T}_h}\nonumber\\
=&\mathcal{B}(\boldsymbol{\mathcal{R}}_{2,h}-\boldsymbol{\mathcal{R}}_{2},\Gamma_{2,h}-\Gamma,\widehat{\Gamma}_{2,h}-\Gamma;
\Pi_W{\bm q}-{\bm q},\Pi_Vu-u,\Pi_{\widehat{V}}u-u)\nonumber\\
&+\mathcal{B}(\boldsymbol{\mathcal{R}}_{2},\Gamma,\Gamma;\Pi_W{\bm q}-{\bm q},\Pi_Vu-u,\Pi_{\widehat{V}}u-u)\nonumber\\
\leq&C\Big(\|\sigma_z^s(\Pi_W{\bm q}-{\bm q})\|_{L^p(\Omega)}+\|\sigma_z^s\nabla(\Pi_Vu-u)\|_{L_h^p(\Omega)}\label{rG-HDG2:1}\\
&+\Big(\sum_{K\in\mathcal{T}_h}h_K\|\sigma_z^s(\Pi_W{\bm q}-{\bm q})\cdot{\bm n}\|^p_{L^p(\partial K)}+h_K^{1-p}\|\sigma_z^s(\Pi_Vu-u)\|_{L^p(\partial K)}^p\Big)^{1/p}\Big)\nonumber\\
&\times\Big(h^{-\frac{3}{p}}+\|\sigma_z^{-s}(\boldsymbol{\mathcal{R}}_{2,h}-\boldsymbol{\mathcal{R}}_{2})\|_{L^q(\Omega)}
+\|\sigma_z^{-s}\nabla(\Gamma_{2,h}-\Gamma)\|_{L_h^q(\Omega)}\nonumber\\
&+\Big(\sum_{K\in\mathcal{T}_h}h_K\|\sigma_z^{-s}(\boldsymbol{\mathcal{R}}_{2,h}-\boldsymbol{\mathcal{R}}_{2})\cdot{\bm n})\|_{L^q(\partial K)}^q+h_K^{1-q}\|\sigma_z^{-s}(\Gamma_{2,h}-\widehat{\Gamma}_{2,h})\|_{L^q(\partial K)}^q\Big)^{1/q}\Big),\nonumber
\end{align}
where $s$ is any positive real number, $q$ is the conjugate number of $p$ such that $\frac{1}{p}+\frac{1}{q}=1$, and $\sigma_z$ is a weight function defined in (\ref{swf}).

Next, for $x_0\in F$ with $F\in\mathcal{E}_h^o$ and $x_{\ast}\in F$ with $F\in\mathcal{E}_h^{\partial}$, let $\Gamma_0(x_0,x)$ and
$\Gamma_{\ast}(x_{\ast},x)$ be the interface Green's functions defined in (\ref{rG0}) and (\ref{rG1}) respectively. Then, by setting $\boldsymbol{\mathcal{R}}_0=-\nabla\Gamma_0$
and $\boldsymbol{\mathcal{R}}_{\ast}=-\nabla\Gamma_{\ast}$, we arrive at the following two mixed formulations:
\begin{align}
\boldsymbol{\mathcal{R}}_0+\nabla\Gamma_0=0\quad {\rm in}~\Omega,\quad \nabla\cdot\boldsymbol{\mathcal{R}}_0=\delta_{h,K^+}-\delta_{h,K^-}\quad
{\rm in}~\Omega,\quad \Gamma_0=0\quad {\rm on}~\partial\Omega,\label{rG0-mixed}
\end{align}
and
\begin{align}
\boldsymbol{\mathcal{R}}_{\ast}+\nabla\Gamma_{\ast}=0\quad {\rm in}~\Omega,\quad \nabla\cdot\boldsymbol{\mathcal{R}}_{\ast}=\delta_{h,x_{\ast}}\quad
{\rm in}~\Omega,\quad \Gamma_{\ast}=0\quad {\rm on}~\partial\Omega,\label{rG1-mixed}
\end{align}
For the above two mixed formulations (\ref{rG0-mixed}) and (\ref{rG1-mixed}), let $(\boldsymbol{\mathcal{R}}_{0,h},\Gamma_{0,h},\widehat{\Gamma}_{0,h})\in {\bm W}_h\times V_h\times \widehat{V}_h$ and $(\boldsymbol{\mathcal{R}}_{\ast,h},\Gamma_{\ast,h},\widehat{\Gamma}_{\ast,h})\in {\bm W}_h\times V_h\times \widehat{V}_h$ be their HDG approximations such that
\begin{align}
\mathcal{B}(\boldsymbol{\mathcal{R}}_{0,h},\Gamma_{0,h},\widehat{\Gamma}_{0,h}; {\bm w}_h,v_h,\widehat{v}_h)=-(\delta_{h,K^+}-\delta_{h,K^-},v_h)_{\mathcal{T}_h},
\quad \forall ({\bm w}_h,v_h,\widehat{v}_h)\in {\bm W}_h\times V_h\times\widehat{V}_h,\label{rG0-HDG}
\end{align}
and
\begin{align}
\mathcal{B}(\boldsymbol{\mathcal{R}}_{\ast,h},\Gamma_{\ast,h},\widehat{\Gamma}_{\ast,h}; {\bm w}_h,v_h,\widehat{v}_h)=-(\delta_{h,x_{\ast}},v_h)_{\mathcal{T}_h},
\quad \forall ({\bm w}_h,v_h,\widehat{v}_h)\in {\bm W}_h\times V_h\times\widehat{V}_h.\label{rG1-HDG}
\end{align}
Similar to (\ref{rG-HDG1:1}) and (\ref{rG-HDG2:1}), by virtue of Lemma \ref{dfp}, (\ref{ddd:2}), and (\ref{rG0-mixed})-(\ref{rG1-HDG}), we obtain
for $2\leq p<\frac{1}{\epsilon}$
\begin{align}
&-[\Pi_Vu-u_h](x_0)=-(\delta_{h,K^+}-\delta_{h,K^-}, \Pi_Vu-u_h)_{\mathcal{T}_h}\nonumber\\
=&\mathcal{B}(\boldsymbol{\mathcal{R}}_{0,h}-\boldsymbol{\mathcal{R}}_0,\Gamma_{0,h}-\Gamma_0,\widehat{\Gamma}_{0,h}-\Gamma_0;\Pi_W{\bm q}-{\bm q},\Pi_Vu-u,\Pi_{\widehat{V}}u-u)\nonumber\\
&-(\delta_{h,K^+}-\delta_{h,K^-},\Pi_Vu-u)+\langle\tau(\mathcal{L}\Gamma_0-\Gamma_0),\Pi_{\widehat{V}}u-u\rangle_{\partial\mathcal{T}_h}\nonumber\\
\leq&C\Big(\|\sigma_{x_0}^s(\Pi_W{\bm q}-{\bm q})\|_{L^p(\Omega)}+\|\sigma_{x_0}^s\nabla(\Pi_Vu-u)\|_{L_h^p(\Omega)}+h^{-1}\|\sigma_{x_0}^s(\Pi_Vu-u)\|_{L^p(\Omega)}\label{rG0-HDG:1}\\
&+\Big(\sum_{K\in\mathcal{T}_h}h_K\|\sigma_{x_0}^s(\Pi_W{\bm q}-{\bm q})\cdot{\bm n}\|^p_{L^p(\partial K)}+h_K^{1-p}\|\sigma_{x_0}^s(\Pi_Vu-u)\|_{L^p(\partial K)}^p\Big)^{1/p}\Big)\nonumber\\
&\times\Big(h^{1-\frac{3}{p}}+\|\sigma_{x_0}^{-s}(\boldsymbol{\mathcal{R}}_{0,h}-\boldsymbol{\mathcal{R}}_{0})\|_{L^q(\Omega)}
+\|\sigma_{x_0}^{-s}\nabla(\Gamma_{0,h}-\Gamma_{0})\|_{L_h^q(\Omega)}\nonumber\\
&+\Big(\sum_{K\in\mathcal{T}_h}h_K\|\sigma_{x_0}^{-s}(\boldsymbol{\mathcal{R}}_{0,h}-\boldsymbol{\mathcal{R}}_{0})\cdot{\bm n})\|_{L^q(\partial K)}^q+h_K^{1-q}\|\sigma_{x_0}^{-s}(\Gamma_{0,h}-\widehat{\Gamma}_{0,h})\|_{L^q(\partial K)}^q\Big)^{1/q}\Big),\nonumber
\end{align}
and
\begin{align}
&-[\Pi_Vu-u_h](x_{\ast})=-(\delta_{h,x_{\ast}}, \Pi_Vu-u_h)_{\mathcal{T}_h}\nonumber\\
=&\mathcal{B}(\boldsymbol{\mathcal{R}}_{\ast,h}-\boldsymbol{\mathcal{R}}_{\ast},\Gamma_{\ast,h}-\Gamma_{\ast},\widehat{\Gamma}_{\ast,h}-\Gamma_{\ast};
\Pi_W{\bm q}-{\bm q},\Pi_Vu-u,\Pi_{\widehat{V}}u-u)\nonumber\\
&-(\delta_{h,x_{\ast}},\Pi_Vu-u)+\langle\tau(\mathcal{L}\Gamma_{\ast}-\Gamma_{\ast}),\Pi_{\widehat{V}}u-u\rangle_{\partial\mathcal{T}_h}\nonumber\\
\leq&C\Big(\|\sigma_{x_{\ast}}^s(\Pi_W{\bm q}-{\bm q})\|_{L^p(\Omega)}+\|\sigma_{x_{\ast}}^s\nabla(\Pi_Vu-u)\|_{L_h^p(\Omega)}+h^{-1}\|\sigma_{x_{\ast}}^s(\Pi_Vu-u)\|_{L^p(\Omega)}\label{rG1-HDG:1}\\
&+\Big(\sum_{K\in\mathcal{T}_h}h_K\|\sigma_{x_{\ast}}^s(\Pi_W{\bm q}-{\bm q})\cdot{\bm n}\|^p_{L^p(\partial K)}+h_K^{1-p}\|\sigma_{x_{\ast}}^s(\Pi_Vu-u)\|_{L^p(\partial K)}^p\Big)^{1/p}\Big)\nonumber\\
&\times\Big(h^{1-\frac{3}{p}}+\|\sigma_{x_{\ast}}^{-s}(\boldsymbol{\mathcal{R}}_{\ast,h}-\boldsymbol{\mathcal{R}}_{\ast})\|_{L^q(\Omega)}
+\|\sigma_{x_{\ast}}^{-s}\nabla(\Gamma_{\ast,h}-\Gamma_{\ast})\|_{L_h^q(\Omega)}\nonumber\\
&+\Big(\sum_{K\in\mathcal{T}_h}h_K\|\sigma_{x_{\ast}}^{-s}(\boldsymbol{\mathcal{R}}_{\ast,h}-\boldsymbol{\mathcal{R}}_{\ast})\cdot{\bm n})\|_{L^q(\partial K)}^q+h_K^{1-q}\|\sigma_{x_{\ast}}^{-s}(\Gamma_{\ast,h}-\widehat{\Gamma}_{\ast,h})\|_{L^q(\partial K)}^q\Big)^{1/q}\Big).\nonumber
\end{align}

To further estimate (\ref{rG-HDG1:1}), (\ref{rG-HDG2:1}), (\ref{rG0-HDG:1}) and (\ref{rG1-HDG:1}), based on Lemma \ref{rG-lee}, we prove the following results.
\begin{lemma}\label{rG01-HDG-L1}
For $z\in K_z$ with $K_z\in\mathcal{T}_h$, let $(\boldsymbol{\mathcal{R}}_1,\Gamma)$ and $(\boldsymbol{\mathcal{R}}_2,\Gamma)$ be the solutions of problems
(\ref{rG-mixed1}) and (\ref{rG-mixed2}), and let $(\boldsymbol{\mathcal{R}}_{1,h},\Gamma_{1,h},\widehat{\Gamma}_{1,h})$ and $(\boldsymbol{\mathcal{R}}_{2,h},\Gamma_{2,h},\widehat{\Gamma}_{2,h})$ be their HDG approximations defined in (\ref{rG-HDG1}) and (\ref{rG-HDG2}). Then, we have for $l\in\{1,2\}$
\begin{align}
&\Big(\sum_{K\in\mathcal{T}_h}h_K\|\sigma_z^{-s}(\boldsymbol{\mathcal{R}}_{l,h}-\boldsymbol{\mathcal{R}}_{l})\cdot{\bm n})\|_{L^q(\partial K)}^q+h_K^{1-q}\|\sigma_z^{-s}(\Gamma_{l,h}-\widehat{\Gamma}_{l,h})\|_{L^q(\partial K)}^q\Big)^{1/q}\nonumber\\
&+\|\sigma_z^{-s}(\boldsymbol{\mathcal{R}}_{l,h}-\boldsymbol{\mathcal{R}}_{l})\|_{L^q(\Omega)}
+\|\sigma_z^{-s}\nabla(\Gamma_{l,h}-\Gamma)\|_{L_h^q(\Omega)}\leq Ch^{-\frac{3}{p}}.
\end{align}
Furthermore, for $x_0\in F$ with $F\in\mathcal{E}_h^o$ and $x_{\ast}\in F$ with $F\in\mathcal{E}_h^{\partial}$, let $(\boldsymbol{\mathcal{R}}_0,\Gamma_0)$
and $(\boldsymbol{\mathcal{R}}_{\ast},\Gamma_{\ast})$ be the solutions of problems (\ref{rG0-mixed}) and (\ref{rG1-mixed}), and let
$(\boldsymbol{\mathcal{R}}_{0,h},\Gamma_{0,h},\widehat{\Gamma}_{0,h})$ and $(\boldsymbol{\mathcal{R}}_{\ast,h},\Gamma_{\ast,h},\widehat{\Gamma}_{\ast,h})$
be their HDG approximations defined in (\ref{rG0-HDG}) and (\ref{rG1-HDG}). Then, we have for $l\in\{0,\ast\}$
\begin{align}
&\Big(\sum_{K\in\mathcal{T}_h}h_K\|\sigma_{x_{l}}^{-s}(\boldsymbol{\mathcal{R}}_{l,h}-\boldsymbol{\mathcal{R}}_{l})\cdot{\bm n})\|_{L^q(\partial K)}^q+h_K^{1-q}\|\sigma_{x_{l}}^{-s}(\Gamma_{l,h}-\widehat{\Gamma}_{l,h})\|_{L^q(\partial K)}^q\Big)^{1/q}\nonumber\\
&+\|\sigma_{x_{l}}^{-s}(\boldsymbol{\mathcal{R}}_{l,h}-\boldsymbol{\mathcal{R}}_{l})\|_{L^q(\Omega)}
+\|\sigma_{x_{l}}^{-s}\nabla(\Gamma_{l,h}-\Gamma_l)\|_{L_h^q(\Omega)}\leq Ch^{1-\frac{3}{p}}.
\end{align}
Here, $s$ is a real number such that $\frac{3}{p}<s<\gamma+\frac{3}{p}$ for $2\leq p<\frac{1}{\epsilon}$, $\gamma\in(0,1)$ is a positive constant depending on the domain $\Omega$, and $q$ is the conjugate number of $p$ such that $\frac{1}{p}+\frac{1}{q}=1$.
\end{lemma}
\begin{proof}
In the proof, we let $x_1=x_2=z$ and $\Gamma_1=\Gamma_2=\Gamma$. Then, for $l\in\{0,1,2,\ast\}$, the dyadic decomposition of the domain $\Omega$, (\ref{rG-L1-IPDG-proof:ad2}), and (\ref{rG-L1-IPDG-proof:ad3}) yield
\begin{align}
&\Big(\sum_{K\in\mathcal{T}_h}h_K\|\sigma_{x_{l}}^{-s}(\boldsymbol{\mathcal{R}}_{l,h}-\boldsymbol{\mathcal{R}}_{l})\cdot{\bm n})\|_{L^q(\partial K)}^q+h_K^{1-q}\|\sigma_{x_{l}}^{-s}(\Gamma_{l,h}-\widehat{\Gamma}_{l,h})\|_{L^q(\partial K)}^q\Big)^{1/q}\nonumber\\
&+\|\sigma_{x_{l}}^{-s}(\boldsymbol{\mathcal{R}}_{l,h}-\boldsymbol{\mathcal{R}}_{l})\|_{L^q(\Omega)}
+\|\sigma_{x_{l}}^{-s}\nabla(\Gamma_{l,h}-\Gamma_l)\|_{L_h^q(\Omega)}\label{rG01-HDG-L1-proof:1}\\
&\leq CM^s\mathcal{F}_{\Omega_{\ast}(x_l)}+C\sum_{j=0}^I(hd_j^{-1})^{-s}\mathcal{F}_{\Omega_j(x_l)}= J_3+J_4,\nonumber
\end{align}
where
\begin{align*}
\mathcal{F}_{\Omega_{\ast}(x_l)}=&\Big(\sum_{K\in\mathcal{T}_h}h_K\|(\boldsymbol{\mathcal{R}}_{l,h}-\boldsymbol{\mathcal{R}}_{l})\cdot{\bm n})\|_{L^q(\partial K\cap\Omega_{\ast}(x_l))}^q+h_K^{1-q}\|\Gamma_{l,h}-\widehat{\Gamma}_{l,h}\|_{L^q(\partial K\cap\Omega_{\ast}(x_l))}^q\Big)^{1/q}\\
&+\|\boldsymbol{\mathcal{R}}_{l,h}-\boldsymbol{\mathcal{R}}_{l}\|_{L^q(\Omega_{\ast}(x_l))}
+\|\nabla(\Gamma_{l,h}-\Gamma_l)\|_{L_h^q(\Omega_{\ast}(x_l))}.
\end{align*}
On the other hand, for each $K\in\mathcal{T}_h$, Lemma \ref{inequality} yields
\begin{align}
h_K^{-1/2}\|\Gamma_{l,h}-\widehat{\Gamma}_{l,h}\|_{L^2(\partial K)}\leq& C\Big(\|\boldsymbol{\mathcal{R}}_l-\Pi_W\boldsymbol{\mathcal{R}}_l\|_{L^2(K)}+
\|\boldsymbol{\mathcal{R}}_l-\boldsymbol{\mathcal{R}}_{l,h}\|_{L^2(K)}\nonumber\\
&+\|\sqrt{\tau}(\mathcal{L}\Gamma_{l,h}-\widehat{\Gamma}_{l,h})\|_{L^2(\partial K)}+\|\nabla(\Gamma_l-\Pi_V\Gamma_l)\|_{L^2(K)}\label{rG01-HDG-L1-proof:2}\\
&+\|\nabla(\Gamma_l-\Pi_k\Gamma_l)\|_{L^2(K)}\Big).\nonumber
\end{align}
Note that the above estimate has been established in \cite[(4.47)]{lq2026}.

Similar to the estimate of $J_1$ in (\ref{rG-L1-IPDG-proof:2}), using Lemma \ref{leee}, Lemma \ref{inv}, Lemma \ref{tra} with $\beta=1-\frac{1}{q}$, (\ref{ddd:2}), (\ref{rG01-HDG-L1-proof:2}) and interpolation error estimates, we have
\begin{align}
J_3\leq& CM^s(Mh)^{3(\frac{1}{q}-\frac{1}{2})}\Big(\mathcal{K}_{\mathcal{T}_h(\Omega_{\ast}(x_l))}
+\|\boldsymbol{\mathcal{R}}_l-\Pi_W\boldsymbol{\mathcal{R}}_l\|_{L^2(\mathcal{T}_h(\Omega_{\ast}(x_l)))}\nonumber\\
&+h\|\nabla(\boldsymbol{\mathcal{R}}_l-\Pi_W\boldsymbol{\mathcal{R}}_l)\|_{L_h^2(\mathcal{T}_h(\Omega_{\ast}(x_l)))}+
\|\nabla(\Gamma_l-\Pi_V\Gamma_l)\|_{L_h^2(\mathcal{T}_h(\Omega_{\ast}(x_l)))}\nonumber\\
&+\|\nabla(\Gamma_l-\Pi_k\Gamma_l)\|_{L_h^2(\mathcal{T}_h(\Omega_{\ast}(x_l)))}\Big)\label{rG01-HDG-L1-proof:3}\\
\leq& CM^s(Mh)^{3(\frac{1}{q}-\frac{1}{2})}h|\Gamma_l|_{H^2(\Omega)}\leq C
\begin{cases}
h^{-\frac{3}{p}},\quad &l=1,2,\\
h^{1-\frac{3}{p}},\quad & l=0,\ast,
\end{cases}\nonumber
\end{align}
where
\begin{align*}
\mathcal{K}_{\mathcal{T}_h(\Omega_{\ast}(x_l))}=&\|\boldsymbol{\mathcal{R}}_{l,h}
-\boldsymbol{\mathcal{R}}_{l}\|_{L^2(\mathcal{T}_h(\Omega_{\ast}(x_l)))}
+\|\nabla(\Gamma_{l,h}-\Gamma_l)\|_{L_h^2(\mathcal{T}_h(\Omega_{\ast}(x_l)))}\\
&+\Big(\sum_{K\in\mathcal{T}_h(\Omega_{\ast}(x_l))}\|\sqrt{\tau}(\mathcal{L}\Gamma_{l,h}-\widehat{\Gamma}_{l,h})\|_{L^2(\partial K)}^2\Big)^{1/2}.
\end{align*}

As for $J_4$, using (\ref{rG01-HDG-L1-proof:2}), Lemma \ref{lee}-\ref{inv}, Lemma \ref{tra} with $\beta=1-\frac{1}{q}$, Lemma \ref{rG-lee} and interpolation error estimates, we have
\begin{align}
J_4\leq&C\sum_{j=0}^I(hd_j^{-1})^{-s}d_j^{3(\frac{1}{q}-\frac{1}{2})}\Big(\mathcal{K}_{\mathcal{T}_h(\Omega_j(x_l))}
+\|\boldsymbol{\mathcal{R}}_l-\Pi_W\boldsymbol{\mathcal{R}}_l\|_{L^2(\mathcal{T}_h(\Omega_j(x_l)))}\nonumber\\
&+h\|\nabla(\boldsymbol{\mathcal{R}}_l-\Pi_W\boldsymbol{\mathcal{R}}_l)\|_{L^2(\mathcal{T}_h(\Omega_j(x_l)))}+
\|\nabla(\Gamma_l-\Pi_V\Gamma_l)\|_{L^2(\mathcal{T}_h(\Omega_j(x_l)))}\nonumber\\
&+\|\nabla(\Gamma_l-\Pi_k\Gamma_l)\|_{L^2(\mathcal{T}_h(\Omega_j(x_l)))}\Big)\nonumber\\
\leq& C(Mh)^{3(\frac{1}{q}-\frac{1}{2})}\mathcal{K}_{\mathcal{T}_h(\Omega_{\ast}(x_l))}+ C\sum_{j=0}^I(hd_j^{-1})^{-s}d_j^{3(\frac{1}{q}-\frac{1}{2})}\Big((hd_j^{-1})\mathcal{K}_{\mathcal{T}_h(\Omega_j(x_l))}\label{rG01-HDG-L1-proof:4}\\
&+h|\Gamma_l|_{H^2(\Omega_j^2(x_l))}+d_j^{-1}\|\Gamma_l-\Gamma_{l,h}\|_{L^2(\Omega_j^1(x_l))}\Big)\nonumber\\
\leq &C\sum_{j=0}^I(hd_j^{-1})^{-s}d_j^{3(\frac{1}{q}-\frac{1}{2})}\Big((hd_j^{-1})\mathcal{K}_{\mathcal{T}_h(\Omega_j(x_l))}
+d_j^{-1}\|\Gamma_l-\Gamma_{l,h}\|_{L^2(\Omega_j^1(x_l))}\Big)\nonumber\\
&+C\Big(1+\sum_{j=0}^I(hd_j^{-1})^{1+\frac{3}{p}-s}\Big)\begin{cases}
h^{-\frac{3}{p}},\quad &l=1,2,\\
h^{1-\frac{3}{p}},\quad&l=0,\ast.
\end{cases}\nonumber
\end{align}

To estimate $\|\Gamma_l-\Gamma_{l,h}\|_{L^2(\Omega^1_j(x_l))}$, let $(W,\boldsymbol{\mathcal{A}})$ be the solution of the following problem:
\begin{align*}
\boldsymbol{\mathcal{A}}+\nabla W=0\quad {\rm in}~\Omega,\quad \nabla\cdot\boldsymbol{\mathcal{A}}=\psi\quad {\rm in}~\Omega,\quad W=0\quad {\rm on}~\partial\Omega,
\end{align*}
for $\psi\in\mathcal{C}_0^{\infty}(\Omega_j^1(x_l))$ with $\|\psi\|_{L^2(\Omega_j^1(x_l))}=1$. Then, similar to that below \cite[(4.49)]{lq2026}, using
Lemma \ref{dfp} we have
\begin{align*}
-(\psi,\Gamma_l-\Gamma_{l,h})=&(\boldsymbol{\mathcal{A}}-\Pi_W\boldsymbol{\mathcal{A}},\boldsymbol{\mathcal{R}}_l-\boldsymbol{\mathcal{R}}_{l,h})_{\mathcal{T}_h}+
(\nabla(W-\Pi_VW),\boldsymbol{\mathcal{R}}_l-\boldsymbol{\mathcal{R}}_{l,h})_{\mathcal{T}_h}\\
&-\langle W-\Pi_VW,(\boldsymbol{\mathcal{R}}_l-\boldsymbol{\mathcal{R}}_{l,h})\cdot{\bm n}\rangle_{\partial\mathcal{T}_h}+(\boldsymbol{\mathcal{A}}
-\Pi_W\boldsymbol{\mathcal{A}},\nabla(\Gamma_l-\Gamma_{l,h}))_{\mathcal{T}_h}\\
&+\langle(\boldsymbol{\mathcal{A}}-\Pi_W\boldsymbol{\mathcal{A}})\cdot{\bm n},\Gamma_{l,h}-\widehat{\Gamma}_{l,h}\rangle_{\partial\mathcal{T}_h}
+\langle\tau(\Pi_{\widehat{V}}W-\mathcal{L}\Pi_VW),\mathcal{L}\Gamma_{l,h}-\widehat{\Gamma}_{l,h}\rangle_{\partial\mathcal{T}_h}.
\end{align*}
On the subdomain $\Omega_j^3(x_l)$, Lemma \ref{inv}, Lemma \ref{tra} with $\beta=\frac{1}{2}$, (\ref{rG01-HDG-L1-proof:2}), and interpolation error estimates yield
\begin{align*}
-(\psi,\Gamma_l-\Gamma_{l,h})|_{\Omega_j^3(x_l)}\leq& Ch\Big(\|\boldsymbol{\mathcal{R}}_l-\Pi_W\boldsymbol{\mathcal{R}}_l\|_{L^2(\mathcal{T}_h(\Omega_j^3(x_l)))}
+h\|\nabla(\boldsymbol{\mathcal{R}}_l-\Pi_W\boldsymbol{\mathcal{R}}_l)\|_{L^2_h(\mathcal{T}_h(\Omega_j^3(x_l)))}\\
&+\|\nabla(\Gamma_l-\Pi_V\Gamma_l)\|_{L^2_h(\mathcal{T}_h(\Omega_j^3(x_l)))}\\
&+\|\nabla(\Gamma_l-\Pi_k\Gamma_l)\|_{L^2_h(\mathcal{T}_h(\Omega_j^3(x_l)))}
+\mathcal{K}_{\mathcal{T}_h(\Omega_j^3(x_l))}\Big).
\end{align*}
On the subdomain $\Omega\backslash\Omega^3_j(x_l)$, \cite[Lemma 3.8]{lc2026} and the interpolation error estimates presented in \cite[Lemma 4.1]{glrs2009}
yield
\begin{align*}
-(\psi,\Gamma_l-\Gamma_{l,h})|_{\Omega\backslash\Omega_j^3(x_l)}\leq& Ch^{\gamma}d_j^{-\gamma-\frac{1}{2}}\Big(\|\boldsymbol{\mathcal{R}}_l-\boldsymbol{\mathcal{R}}_{l,h}\|_{L^1(\Omega)}+\sum_{K\in\mathcal{T}_h}
h_K\|(\boldsymbol{\mathcal{R}}_l-\boldsymbol{\mathcal{R}}_{l,h})\cdot{\bm n}\|_{L^1(\partial K)}\\
&+\|\nabla(\Gamma_l-\Gamma_{l,h})\|_{L^1_h(\Omega)}+\sum_{K\in\mathcal{T}_h}\|\Gamma_{l,h}-\widehat{\Gamma}_{l,h}\|_{L^1(\partial K)}\Big).
\end{align*}
Thus, using the duality argument, Lemma \ref{leee}, Lemma \ref{rG-lee}, and interpolation error estimates, we have
\begin{align}
&\sum_{j=0}^I(hd_j^{-1})^{-s}d_j^{3(\frac{1}{q}-\frac{1}{2})}d_j^{-1}\|\Gamma_l-\Gamma_{l,h}\|_{L^2(\Omega_j^1(x_l))}\nonumber\\
\leq& C\sum_{j=0}^I(hd_j^{-1})^{-s}d_j^{3(\frac{1}{q}-\frac{1}{2})}h^2d_j^{-1}|\Gamma_l|_{H^2(\Omega_j^4(x_l))}
+C(Mh)^{3(\frac{1}{q}-\frac{1}{2})}\mathcal{K}_{\mathcal{T}_h(\Omega_{\ast}(x_l))}\nonumber\\
&+C\sum_{j=0}^I(hd_j^{-1})^{-s}d_j^{3(\frac{1}{q}-\frac{1}{2})}hd_j^{-1}\mathcal{K}_{\mathcal{T}_h(\Omega_j(x_l))}
+Ch^{-\frac{3}{p}}\sum_{j=0}^I(hd_j^{-1})^{\gamma+\frac{3}{p}-s}\Big(\|\boldsymbol{\mathcal{R}}_l-\boldsymbol{\mathcal{R}}_{l,h}\|_{L^1(\Omega)}\nonumber\\
&+\|\nabla(\Gamma_l-\Gamma_{l,h})\|_{L^1_h(\Omega)}+\sum_{K\in\mathcal{T}_h}
\left(h_K\|(\boldsymbol{\mathcal{R}}_l-\boldsymbol{\mathcal{R}}_{l,h})\cdot{\bm n}\|_{L^1(\partial K)}+\|\Gamma_{l,h}-\widehat{\Gamma}_{l,h}\|_{L^1(\partial K)}\right)\Big)\nonumber\\
\leq&C\sum_{j=0}^I(hd_j^{-1})^{-s}d_j^{3(\frac{1}{q}-\frac{1}{2})}hd_j^{-1}\mathcal{K}_{\mathcal{T}_h(\Omega_j(x_l))}
+Ch^{-\frac{3}{p}}\Big(\|\boldsymbol{\mathcal{R}}_l-\boldsymbol{\mathcal{R}}_{l,h}\|_{L^1(\Omega)}+\|\nabla(\Gamma_l-\Gamma_{l,h})\|_{L^1_h(\Omega)}\nonumber\\
&+\sum_{K\in\mathcal{T}_h}\left(h_K\|(\boldsymbol{\mathcal{R}}_l-\boldsymbol{\mathcal{R}}_{l,h})\cdot{\bm n}\|_{L^1(\partial K)}
+\|\Gamma_{l,h}-\widehat{\Gamma}_{l,h}\|_{L^1(\partial K)}\right)\Big)\nonumber\\
&+C\left(1+\sum_{j=0}^I(hd_j^{-1})^{2+\frac{3}{p}-s}\right)\begin{cases}
h^{-\frac{3}{p}},\quad & l=1,2,\\
h^{1-\frac{3}{p}},\quad& l=0,\ast,
\end{cases}\nonumber
\end{align}
which, together with (\ref{rG01-HDG-L1-proof:4}), leads to
\begin{align}
J_4\leq& Ch^{-\frac{3}{p}}\Big(\|\boldsymbol{\mathcal{R}}_l-\boldsymbol{\mathcal{R}}_{l,h}\|_{L^1(\Omega)}
+\|\nabla(\Gamma_l-\Gamma_{l,h})\|_{L^1_h(\Omega)}+\sum_{K\in\mathcal{T}_h}\Big(h_K\|(\boldsymbol{\mathcal{R}}_l-\boldsymbol{\mathcal{R}}_{l,h})\cdot{\bm n}\|_{L^1(\partial K)}\nonumber\\
&+\|\Gamma_{l,h}-\widehat{\Gamma}_{l,h}\|_{L^1(\partial K)}\Big)\Big)+C\begin{cases}
h^{-\frac{3}{p}},\quad & l=1,2,\\
h^{1-\frac{3}{p}},\quad& l=0,\ast,
\end{cases}\label{rG01-HDG-L1-proof:5}
\end{align}
by setting $M$ to be large enough.

Similar to the above proof that uses the dyadic decomposition of the domain $\Omega$ and the local energy estimates provided in Lemma \ref{rG-lee}, we also can prove
\begin{align*}
&\|\boldsymbol{\mathcal{R}}_l-\boldsymbol{\mathcal{R}}_{l,h}\|_{L^1(\Omega)}
+\|\nabla(\Gamma_l-\Gamma_{l,h})\|_{L^1_h(\Omega)}+\sum_{K\in\mathcal{T}_h}\Big(h_K\|(\boldsymbol{\mathcal{R}}_l-\boldsymbol{\mathcal{R}}_{l,h})\cdot{\bm n}\|_{L^1(\partial K)}\\
&+\|\Gamma_{l,h}-\widehat{\Gamma}_{l,h}\|_{L^1(\partial K)}\Big)\leq \begin{cases}
C,\quad & l=1,2,\\
Ch,\quad& l=0,\ast,
\end{cases}
\end{align*}
which, together with (\ref{rG01-HDG-L1-proof:1}), (\ref{rG01-HDG-L1-proof:3}), and (\ref{rG01-HDG-L1-proof:5}), concludes the proof. Note that
for $l\in\{0,\ast\}$ and $l\in\{1,2\}$, the above $L^1$ error estimates have been established in \cite[Lemma 4.8 and Lemma A.2]{lq2026} respectively.
\end{proof}

Based on (\ref{rG-HDG1:1}), (\ref{rG-HDG2:1}), (\ref{rG0-HDG:1}), (\ref{rG1-HDG:1}), and Lemma \ref{rG01-HDG-L1}, for $2\leq p<\frac{1}{\epsilon}$ and $\frac{3}{p}<s<\gamma+\frac{3}{p}$ we obtain
\begin{align}
&|\nabla(\Pi_Vu-u_h)(z)|+|(\Pi_W{\bm q}-{\bm q}_h)(z)|\nonumber\\
\leq& Ch^{-\frac{3}{p}}\Big(\|\sigma_z^s(\Pi_W{\bm q}-{\bm q})\|_{L^p(\Omega)}+\|\sigma_z^s\nabla(\Pi_Vu-u)\|_{L_h^p(\Omega)}\label{rG-HDG12:1}\\
&+\Big(\sum_{K\in\mathcal{T}_h}h_K\|\sigma_z^s(\Pi_W{\bm q}-{\bm q})\cdot{\bm n}\|^p_{L^p(\partial K)}+h_K^{1-p}\|\sigma_z^s(\Pi_Vu-u)\|_{L^p(\partial K)}^p\Big)^{1/p}\Big),\nonumber
\end{align}
if $z\in K_z$ with $K_z\in\mathcal{T}_h$, and
\begin{align}
|[\Pi_Vu-u_h](z)|\leq& Ch^{1-\frac{3}{p}}\Big(\|\sigma_{z}^s(\Pi_W{\bm q}-{\bm q})\|_{L^p(\Omega)}+\|\sigma_{z}^s\nabla(\Pi_Vu-u)\|_{L_h^p(\Omega)}\nonumber\\
&+h^{-1}\|\sigma_{z}^s(\Pi_Vu-u)\|_{L^p(\Omega)}
+\Big(\sum_{K\in\mathcal{T}_h}h_K\|\sigma_{z}^s(\Pi_W{\bm q}-{\bm q})\cdot{\bm n}\|^p_{L^p(\partial K)}\label{rG-HDG0:1}\\
&+h_K^{1-p}\|\sigma_{z}^s(\Pi_Vu-u)\|_{L^p(\partial K)}^p\Big)^{1/p}\Big),\nonumber
\end{align}
if $z\in F$ with $F\in\mathcal{E}_h$.

Using $|\nabla\sigma_z^s|\leq Ch^{-1}\sigma_z^s$ and Lemma \ref{tra} with $\beta=\epsilon$ and $\beta=1-\frac{1}{p}$,  we have
\begin{align*}
&\sum_{K\in\mathcal{T}_h}\Big(h_K\|\sigma_z^s(\Pi_W{\bm q}-{\bm q})\cdot{\bm n}\|^p_{L^p(\partial K)}+h_K^{1-p}\|\sigma_z^s(\Pi_Vu-u)\|_{L^p(\partial K)}^p\Big) \\
&+\|\sigma_z^s(\Pi_W{\bm q}-{\bm q})\|^p_{L^p(\Omega)}+\|\sigma_z^s\nabla(\Pi_Vu-u)\|^p_{L_h^p(\Omega)}+h^{-p}\|\sigma_{z}^s(\Pi_Vu-u)\|^p_{L^p(\Omega)}\\
\leq&C\Big(\|\sigma_z^s(\Pi_W{\bm q}-{\bm q})\|^p_{L^p(\Omega)}+h^{1+p\epsilon}|\sigma_z^s(\Pi_W{\bm q}-{\bm q})|_{W_h^{\frac{1}{p}+\epsilon,p}(\Omega)}^p\\
&+h^{-p}\|\sigma_z^s(\Pi_Vu-u)\|_{L^p(\Omega)}^p+\|\sigma_z^s\nabla(\Pi_Vu-u)\|^p_{L_h^p(\Omega)}\Big),
\end{align*}
which, together with (\ref{rG-HDG12:1}), (\ref{rG-HDG0:1}), and Lemma \ref{fubini}, leads to
\begin{align*}
&\|\nabla(\Pi_Vu-u_h)\|_{L^p_h(\Omega)}^p+\|\Pi_W{\bm q}-{\bm q}_h\|_{L^p(\Omega)}^p+\sum_{F\in\mathcal{E}_h}h_F^{1-p}\|[\Pi_Vu-u_h]\|_{L^p(F)}^p\\
\leq& Ch^{-3}\int_{\Omega}\Big(\|\sigma_z^s(\Pi_W{\bm q}-{\bm q})\|^p_{L^p(\Omega)}+h^{1+p\epsilon}|\sigma_z^s(\Pi_W{\bm q}-{\bm q})|_{W_h^{\frac{1}{p}+\epsilon,p}(\Omega)}^p\\
&+h^{-p}\|\sigma_z^s(\Pi_Vu-u)\|_{L^p(\Omega)}^p+\|\sigma_z^s\nabla(\Pi_Vu-u)\|^p_{L_h^p(\Omega)}\Big)dz\\
&+Ch^{-2}\sum_{F\in\mathcal{E}_h}\int_F\Big(\|\sigma_z^s(\Pi_W{\bm q}-{\bm q})\|^p_{L^p(\Omega)}+h^{1+p\epsilon}|\sigma_z^s(\Pi_W{\bm q}-{\bm q})|_{W_h^{\frac{1}{p}+\epsilon,p}(\Omega)}^p\\
&+h^{-p}\|\sigma_z^s(\Pi_Vu-u)\|_{L^p(\Omega)}^p+\|\sigma_z^s\nabla(\Pi_Vu-u)\|^p_{L_h^p(\Omega)}\Big)dz\\
\leq& C\Big(\|\Pi_W{\bm q}-{\bm q}\|^p_{L^p(\Omega)}+h^{1+p\epsilon}|\Pi_W{\bm q}-{\bm q}|_{W_h^{\frac{1}{p}+\epsilon,p}(\Omega)}^p\\
&+h^{-p}\|\Pi_Vu-u\|_{L^p(\Omega)}^p+\|\nabla(\Pi_Vu-u)\|^p_{L_h^p(\Omega)}\Big).
\end{align*}
Consequently, the above estimate, Lemma \ref{tra} with $\beta=1-\frac{1}{p}$, and the triangle inequality derive (\ref{W1p-HDG-pre:1}) for the case
$2\leq p<\frac{1}{\epsilon}$.

To prove (\ref{W1p-HDG-pre:1}) for the case $\frac{1}{1-\epsilon}<p<2$, let $(\boldsymbol{\psi}_h,\phi_h,\widehat{\phi}_h)\in {\bm W}_h\times V_h\times \widehat{V}_h$ be the solution of the following problem:
\begin{align*}
&\mathcal{B}(\boldsymbol{\psi}_h,\phi_h,\widehat{\phi}_h;{\bm w}_h,v_h,\widehat{v}_h)\\
=&(|\Pi_W{\bm q}-{\bm q}_h|^{p-2}(\Pi_W{\bm q}-{\bm q}_h),{\bm w}_h)_{\mathcal{T}_h}+(|\nabla(\Pi_Vu-u_h)|^{p-2}\nabla(\Pi_Vu-u_h),\nabla v_h)_{\mathcal{T}_h}\\
&+\sum_{F\in\mathcal{E}_h}h_F^{1-p}\langle|[\Pi_Vu-u_h]|^{p-2}[\Pi_Vu-u_h],[v_h]\rangle_{F},\quad \forall ({\bm w}_h,v_h,\widehat{v}_h)\in{\bm W}_h\times V_h\times \widehat{V}_h.
\end{align*}
Then, Lemma \ref{dfp}, the inf-sup condition (\ref{inf-sup-HDG0:1}), and the H\"{o}lder's inequality yield
\begin{align*}
\|\Pi_W{\bm q}-{\bm q}_h\|_{L^p(\Omega)}^p+|||\Pi_Vu-u_h|||_{W_h^{1,p}(\Omega)}^p=&\mathcal{B}(\boldsymbol{\psi}_h,\phi_h,\widehat{\phi}_h;\Pi_W{\bm q}-{\bm q}_h,\Pi_Vu-u_h,\Pi_{\widehat{V}}u-\widehat{u}_h)\\
=&\mathcal{B}(\boldsymbol{\psi}_h,\phi_h,\widehat{\phi}_h;\Pi_W{\bm q}-{\bm q},\Pi_Vu-u,\Pi_{\widehat{V}}u-u),
\end{align*}
and
\begin{align*}
|||(\boldsymbol{\psi}_h,\phi_h,\widehat{\phi}_h)|||_{W_{h^{\prime}}^{1,q}(\Omega)}\leq C\Big(\|\Pi_W{\bm q}-{\bm q}_h\|_{L^p(\Omega)}^{p-1}+|||\Pi_Vu-u_h|||_{W_h^{1,p}(\Omega)}^{p-1}\Big),
\end{align*}
where the mesh-dependent norms $|||\cdot|||_{W_h^{1,p}(\Omega)}$ and $|||(\cdot,\cdot,\cdot)|||_{W_{h^{\prime}}^{1,q}(\Omega)}$ are defined in Theorem \ref{W1p-IPDG} and Corollary \ref{W1p-HDG-c1}.
Thus, using Lemma \ref{inv}, Lemma \ref{tra} with $\beta=\epsilon$ and $\beta=1-\frac{1}{p}$, and the definition of $\mathcal{B}$, we have
\begin{align*}
&\|\Pi_W{\bm q}-{\bm q}_h\|_{L^p(\Omega)}+|||\Pi_Vu-u_h|||_{W_h^{1,p}(\Omega)}\\
\leq& C\frac{\mathcal{B}(\boldsymbol{\psi}_h,\phi_h,\widehat{\phi}_h;\Pi_W{\bm q}-{\bm q},\Pi_Vu-u,\Pi_{\widehat{V}}u-u)}{|||(\boldsymbol{\psi}_h,\phi_h,\widehat{\phi}_h)|||_{W_{h^{\prime}}^{1,q}(\Omega)}}\\
\leq&C\Big(\|\nabla(u-\Pi_Vu)\|_{L^p_h(\Omega)}+\|{\bm q}-\Pi_W{\bm q}\|_{L^p(\Omega)}+h^{\frac{1}{p}+\epsilon}|{\bm q}-\Pi_W{\bm q}|_{W_h^{\frac{1}{p}+\epsilon,p}(\Omega)}\Big),
\end{align*}
which, together with Lemma \ref{tra} for $\beta=1-\frac{1}{p}$ and the triangle inequality, obtains the result (\ref{W1p-HDG-pre:1}) for the case $\frac{1}{1-\epsilon}<p<2$. We conclude the proof.

\section{Numerical experiments}\label{sec6}
In this section, we verify the theoretical results obtained in this paper through a simple example, in which the exact solution and the domain are chosen
as
\begin{align*}
u=\sin(2\pi x_1)\sin(2\pi x_2)\sin(2\pi x_3), \quad \Omega=(0,1)^3.
\end{align*}
For ease of presentation, we adopt the following notations in this section:
\begin{align*}
E^{\rm HDG}_{1,p}=&\|{\bm q}-{\bm q}_h\|_{L^p(\Omega)}+|||u-u_h|||_{W_h^{1,p}(\Omega)},\\
E^{\rm HDG}_{2,p}=&\|{\bm q}-{\bm q}_h\|_{L^p(\Omega)}+\|\nabla(u-u_h)\|_{L_h^p(\Omega)}+\Big(\sum_{K\in\mathcal{T}_h}\tau^{p-1}\|\mathcal{L}u_h-\widehat{u}_h\|_{L^p(\partial K)}^p\Big)^{1/p},\\
E^{\rm HDG}_{3,p}=&|||({\bm q}-{\bm q}_h,u-u_h,u-\widehat{u}_h)|||_{W_{h^{\prime}}^{1,p}(\Omega)},\quad E^{\rm DG}_p=|||u-u_h^{\rm DG}|||_{W_h^{1,p}(\Omega)}.
\end{align*}

In Table 1, we present the convergence histories and convergence rates of the error $E_p^{\rm DG}$ for the IPDG method with $p\in\{\frac{6}{5},4,8\}$ and $k\in\{1,2,3\}$. In Table 2 and 3, we present the convergence histories and convergence rates of the errors $E_{1,p}^{\rm HDG}$, $E_{2,p}^{\rm HDG}$
and $E_{3,p}^{\rm HDG}$ for the hybrid mixed DG and HDG-1 methods with $p\in\{\frac{6}{5},4,8\}$ and $k=2$. From these results, we observe that the convergence
rates are numerically optimal.

In Table 4, we also provide the convergence histories and convergence rates of the errors $E_{1,p}^{\rm HDG}$, $E_{2,p}^{\rm HDG}$ and $E_{3,p}^{\rm HDG}$ for the HDG-2 method with $p\in\{\frac{6}{5},4,8\}$ and $k=2$. The results show that the convergence rates of $E_{1,p}^{\rm HDG}$ and $E_{2,p}^{\rm HDG}$ are numerically optimal. For the error $E_{3,p}^{\rm HDG}$, the convergence rate is only suboptimal. This is due to the suboptimal
convergence rate of $\big(\sum_{K\in\mathcal{T}_h}h_K^{1-p}\|u_h-\widehat{u}_h\|_{L^p(\partial K)}^p\big)^{1/p}$ as shown in Table 5, where a polynomial $u_h$ of degree $k+1$ is used to approximate the polynomial $\widehat{u}_h$ of degree $k$ on each element boundaries.

\begin{center}
\begin{tabular}{c|c|c|c|c|c|c|c}
  \hline
   p & h & $E_p^{\rm DG}~(k=1)$ & Order & $E_{p}^{\rm DG}~(k=2)$ & Order & $E_{p}^{\rm DG}~(k=3)$ & Order \\
  \hline
                & $\frac{1}{32}$ &5.1011e-1& &2.2098e-2 & &6.7094e-4&  \\
   $\frac{6}{5}$& $\frac{1}{64}$ &2.5579e-1&1.00&5.5413e-3 &2.00&8.3621e-5&3.00\\
                & $\frac{1}{128}$&1.2800e-1&1.00&1.3868e-3 &2.00&1.0442e-5&3.00\\
  \hline
     & $\frac{1}{32}$ &4.5451e-1& &2.6024e-2 & &1.0734e-3&  \\
   4 & $\frac{1}{64}$ &2.2841e-1&1.00&6.5517e-3 &1.99&1.3435e-4&3.00\\
     & $\frac{1}{128}$&1.1437e-1&1.00&1.6413e-3 &2.00&1.6794e-5&3.00\\
 \hline
     & $\frac{1}{32}$ &5.5749e-1& &4.0735e-2 & &2.1044e-3&  \\
   8 & $\frac{1}{64}$ &2.8021e-1&1.00&1.0253e-2 &1.99&2.6349e-4&3.00\\
     & $\frac{1}{128}$&1.4031e-1&1.00&2.5685e-3 &2.00&3.2939e-5&3.00\\
 \hline
\end{tabular}
\vskip 1mm
Table 1. The convergence histories and convergence rates for the IPDG method.
\end{center}
\vskip 5mm

\begin{center}
\begin{tabular}{c|c|c|c|c|c|c|c}
  \hline
   p & h & $E_{1,p}^{\rm HDG}~(k=2)$ & Order & $E_{2,p}^{\rm HDG}~(k=2)$ & Order & $E_{3,p}^{\rm HDG}~(k=2)$ & Order\\
  \hline
                & $\frac{1}{32}$ &2.0384e-2 & &1.5942e-2&  &2.1503e-2&\\
   $\frac{6}{5}$& $\frac{1}{64}$ &5.0688e-3 &2.01&3.9536e-3&2.01&5.3685e-3&2.00\\
                & $\frac{1}{128}$&1.2632e-3 &2.01&9.8406e-4&2.01&1.3410e-3&2.00\\
  \hline
     & $\frac{1}{32}$ &2.3007e-2 & &2.2997e-2&  &2.2626e-2&\\
   4 & $\frac{1}{64}$ &5.7181e-3 &2.01&5.7157e-3&2.01&5.6695e-3&2.00\\
     & $\frac{1}{128}$&1.4245e-3 &2.01&1.4239e-3&2.01&1.4182e-3&2.00\\
 \hline
     & $\frac{1}{32}$ &3.7016e-2 & &3.7016e-2&  &3.6414e-2&\\
   8 & $\frac{1}{64}$ &9.2028e-3 &2.01&9.2028e-3&2.01&9.1273e-3&2.00\\
     & $\frac{1}{128}$&2.2928e-3 &2.01&2.2928e-3&2.01&2.2833e-3&2.00\\
 \hline
\end{tabular}
\vskip 1mm
Table 2. The convergence histories and convergence rates for the hybrid mixed DG method.
\end{center}
\vskip 5mm

\begin{center}
\begin{tabular}{c|c|c|c|c|c|c|c}
  \hline
   p & h & $E_{1,p}^{\rm HDG}~(k=2)$ & Order & $E_{2,p}^{\rm HDG}~(k=2)$ & Order & $E_{3,p}^{\rm HDG}~(k=2)$ & Order\\
  \hline
                & $\frac{1}{32}$ &2.5067e-2 & &2.2025e-2&  &2.8156e-2&\\
   $\frac{6}{5}$& $\frac{1}{64}$ &5.7314e-3 &2.01&4.8004e-3&2.20&6.2820e-3&2.16\\
                & $\frac{1}{128}$&1.3903e-3 &2.01&1.1291e-3&2.09&1.5091e-3&2.06\\
  \hline
     & $\frac{1}{32}$ &2.6141e-2 & &2.7792e-2&  &2.5187e-2&\\
   4 & $\frac{1}{64}$ &6.2318e-3 &2.07&6.4373e-3&2.11&6.1114e-3&2.04\\
     & $\frac{1}{128}$&1.5314e-3 &2.03&1.5568e-3&2.05&1.5161e-3&2.01\\
 \hline
     & $\frac{1}{32}$ &4.1540e-2 & &4.3668e-2&  &3.9761e-2&\\
   8 & $\frac{1}{64}$ &9.9508e-3 &2.06&1.0217e-2&2.10&9.7281e-3&2.03\\
     & $\frac{1}{128}$&2.4464e-3 &2.02&2.4798e-3&2.04&2.4186e-3&2.01\\
 \hline
\end{tabular}
\vskip 1mm
Table 3. The convergence histories and convergence rates for the HDG-1 method.
\end{center}
\vskip 5mm

\begin{center}
\begin{tabular}{c|c|c|c|c|c|c|c}
  \hline
   p & h & $E_{1,p}^{\rm HDG}~(k=2)$ & Order & $E_{2,p}^{\rm HDG}~(k=2)$ & Order & $E_{3,p}^{\rm HDG}~(k=2)$ & Order\\
  \hline
                & $\frac{1}{32}$ &4.7093e-3 & &6.1122e-3&  &1.2610e-2&\\
   $\frac{6}{5}$& $\frac{1}{64}$ &5.8814e-4 &3.00&7.6476e-4&3.00&2.6957e-3&2.23\\
                & $\frac{1}{128}$&7.3489e-5 &3.00&9.5622e-5&3.00&6.2687e-4&2.10\\
  \hline
     & $\frac{1}{32}$ &4.7579e-3 & &5.4480e-3&  &4.8210e-3&\\
   4 & $\frac{1}{64}$ &5.9529e-4 &3.00&6.8156e-4&3.00&1.0084e-3&2.26\\
     & $\frac{1}{128}$&7.4416e-5 &3.00&8.5215e-5&3.00&2.4750e-4&2.02\\
 \hline
     & $\frac{1}{32}$ &6.8739e-3 & &7.5876e-3&  &6.1003e-3&\\
   8 & $\frac{1}{64}$ &8.5966e-4 &3.00&9.4894e-4&3.00&1.3610e-3&2.16\\
     & $\frac{1}{128}$&1.0747e-4 &3.00&1.1863e-4&3.00&3.3956e-4&2.00\\
 \hline
\end{tabular}
\vskip 1mm
Table 4. The convergence histories and convergence rates for the HDG-2 method.
\end{center}
\vskip 5mm

\begin{center}
\begin{tabular}{c|c|c|c|c|c|c|c}
  \hline
   p & h & $\|{\bm q}-{\bm q}_h\|_{L^p(\Omega)}$ & Order & $\|\nabla(u-u_h)\|_{L_h^p(\Omega)}$ & Order & $\Big(\sum_{K\in\mathcal{T}_h}h_K^{1-p}\|u_h-\widehat{u}_h\|_{L^p(\partial K)}^p\Big)^{1/p}$ & Order\\
  \hline
                & $\frac{1}{32}$ &4.9612e-4 & &3.3457e-3&  &1.0207e-2&\\
   $\frac{6}{5}$& $\frac{1}{64}$ &6.2154e-5 &3.00&4.1848e-4&3.00&2.4283e-3&2.07\\
                & $\frac{1}{128}$&7.7751e-6 &3.00&5.2321e-5&3.00&5.9751e-4&2.02\\
  \hline
     & $\frac{1}{32}$ &6.4962e-4 & &4.1076e-3&  &3.9973e-3&\\
   4 & $\frac{1}{64}$ &8.1300e-5 &3.00&5.1382e-4&3.00&9.9094e-4&2.01\\
     & $\frac{1}{128}$&1.0163e-5 &3.00&6.4243e-5&3.00&2.4722e-4&2.00\\
 \hline
     & $\frac{1}{32}$ &1.1713e-3 & &5.7026e-3&  &5.4684e-3&\\
   8 & $\frac{1}{64}$ &1.4653e-4 &3.00&7.1313e-4&3.00&1.3600e-3&2.01\\
     & $\frac{1}{128}$&1.8315e-5 &3.00&8.9156e-5&3.00&3.3956e-4&2.00\\
 \hline
\end{tabular}
\vskip 1mm
Table 5. The convergence histories and convergence rates for the HDG-2 method and $k=2$.
\end{center}

\section{Conclusions}\label{sec7}
In this paper, we have established optimal discrete $W^{1,p}$ error estimates and discrete $W^{1,q}$-$W^{1,p}$ inf-sup stability for a symmetric IPDG method, a hybrid mixed DG method, and two representative HDG methods for the Poisson problem on convex polyhedral domains. For any fixed $\epsilon\in(0,\frac{1}{2})$, these results hold for $\frac{1}{1-\epsilon}<p<\frac{1}{\epsilon}$ and are free of logarithmic factors.
The key to our analysis is the use of interface Green's functions developed in \cite{lq2026}, combined with Fubini-type weighted estimates established in this paper. A numerical example is also presented to illustrate the theoretical findings.
These results provide a basis for further investigations of DG methods on convex polyhedral domains. In particular, the logarithm-free discrete $W^{1,p}$ error estimates and inf-sup stability may facilitate the analysis of $W^{2,p}$ error estimates for DG approximations of elliptic equations in non-divergence form and the establishment of weak discrete maximum principles for DG methods.

\smallskip
\smallskip

\appendix

\section{}\label{apendixA}
In this section, we are devoted to proving the following two important lemmas.
\begin{lemma}\label{fubini}
Let $1\leq p<\infty$, $\frac{3}{p}<s$ and $l\in(0,1)$. Then, for $v\in L^p(\Omega)$ we have
\begin{align}
\sum_{F\in\mathcal{E}_h}\int_F\|\sigma_z^sv\|_{L^p(\Omega)}^p dz\leq Ch^2\|v\|_{L^p(\Omega)}^p,\quad \int_{\Omega}\|\sigma_z^sv\|_{L^p(\Omega)}^p dz\leq Ch^3\|v\|_{L^p(\Omega)}^p,\label{fubini:1}
\end{align}
where $C$ is the positive constant depending on $s$ and $p$. Furthermore, for $v\in W_h^{1+l,p}(\Omega)$ we have
\begin{align}
\sum_{F\in\mathcal{E}_h}\int_F|\sigma_z^s\nabla v|_{W_h^{l,p}(\Omega)}^p dz\leq& Ch^2|\nabla v|_{W_h^{l,p}(\Omega)}^p+Ch^{2-pl}\|\nabla v\|^p_{L^p_h(\Omega)},\label{fubini:2}
\end{align}
and
\begin{align}
\int_{\Omega}|\sigma_z^s\nabla v|_{W_h^{l,p}(\Omega)}^p dz\leq& Ch^3|\nabla v|^p_{W_h^{l,p}(\Omega)}+Ch^{3-pl}\|\nabla v\|_{L_h^p(\Omega)}^p,\label{fubini:3}
\end{align}
where $C$ is the positive constant depending on $s$, $p$ and $l$. Here, $\sigma_z$ is a weight function defined in (\ref{swf}).
\end{lemma}
\begin{proof}
By a simple calculation, we have
\begin{align*}
\sum_{F\in\mathcal{E}_h}\int_F\|\sigma_z^sv\|_{L^p(\Omega)}^p dz=\int_{\Omega}|v(x)|^p\left(\sum_{F\in\mathcal{E}_h}\int_F\left(\frac{h}{|x-z|+h}\right)^{ps} dz\right)dx,
\end{align*}
and
\begin{align*}
\int_{\Omega}\|\sigma_z^sv\|_{L^p(\Omega)}^p dz=\int_{\Omega}|v(x)|^p\left(\int_{\Omega}\left(\frac{h}{|x-z|+h}\right)^{ps} dz
\right)dx.
\end{align*}
Thus, to obtain (\ref{fubini:1}), we only need to prove
\begin{align}
\sum_{F\in\mathcal{E}_h}\int_F\left(\frac{h}{|x-z|+h}\right)^{ps} dz\leq Ch^2,\label{fubini-proof:2}
\end{align}
and
\begin{align}
\int_{\Omega}\left(\frac{h}{|x-z|+h}\right)^{ps} dz\leq Ch^3.\label{fubini-proof:ad2}
\end{align}

In the proof of this lemma, we only prove (\ref{fubini-proof:2}), since (\ref{fubini-proof:ad2}) can be proved similarly. For the given $x\in\Omega$, let
\begin{align*}
\widehat{\Omega}_{\ast}(x)=&\{y\in \cup_{F\in\mathcal{E}_h}F: |y-x|\leq 2h\},\\
\widehat{\Omega}_j(x)=&\{y\in\cup_{F\in\mathcal{E}_h}F: 2^jh\leq |y-x|\leq 2^{j+1}h\},\quad j=1,\cdots,J-1,
\end{align*}
where $J=\log_2\left(\frac{{\rm diam}(\Omega)}{h}\right)$ such that $2^J h={\rm diam}(\Omega)$. Then,
\begin{align*}
\sum_{F\in\mathcal{E}_h}\int_F\left(\frac{h}{|x-z|+h}\right)^{ps} dz\leq& \int_{\widehat{\Omega}_{\ast}} dz+\sum_{j=1}^{J-1}\int_{\widehat{\Omega}_j}
\left(\frac{1}{2^j+1}\right)^{ps}dz\\
\leq&C2^3h^{2}+C\sum_{j=1}^{J-1}2^{3j}h^{2}2^{-jps}\\
\leq& Ch^{2}\left(1+\sum_{j=1}^{J-1}2^{-j(ps-3)}\right).
\end{align*}
Due to $ps>3$, we obtain (\ref{fubini-proof:2}) via the above estimate.

On the other hand, according to the definition of the mesh-dependent seminorm $|\cdot|_{W_h^{l,p}(\Omega)}$ for $l\in(0,1)$, we have
\begin{align}
|\sigma_z^s\nabla v|_{W_h^{l,p}(\Omega)}^p=&\sum_{K\in\mathcal{T}_h}\int_K\int_K\frac{|\sigma_z^s(x)\nabla v(x)-\sigma_z^s(y)\nabla v(y)|^p}{|x-y|^{3+pl}}dxdy\nonumber\\
\leq&C\sum_{K\in\mathcal{T}_h}\int_K\int_K\frac{\sigma_z^{ps}(x)|\nabla v(x)-\nabla v(y)|^p}{|x-y|^{3+pl}}dxdy\label{fubini-proof:3}\\
&+C\sum_{K\in\mathcal{T}_h}\int_K\int_K\frac{|\sigma_z^s(x)-\sigma_z^s(y)|^p|\nabla v(y)|^p}{|x-y|^{3+pl}}dxdy=I_1+I_2.\nonumber
\end{align}
For $I_1$, from (\ref{fubini-proof:2}) and (\ref{fubini-proof:ad2}), we have
\begin{align}
\sum_{F\in\mathcal{E}_h}\int_F I_1 dz\leq Ch^2|\nabla v|^p_{W_h^{l,p}(\Omega)},\quad {\rm and}\quad \int_{\Omega} I_1 dz\leq Ch^3|\nabla v|^p_{W_h^{l,p}(\Omega)}.
\label{fubini-proof:4}
\end{align}
As for $I_2$, using the integral mean value theorem and the fact that $|\nabla\sigma_z^s|\leq Ch^{-1}\sigma_z^s$ and $\sigma_z(x)\thickapprox\sigma_z(y)$ for $x,y\in K$, we know that there exists a constant $m\in (0,1)$ such that
\begin{align*}
|\sigma_z^s(x)-\sigma_z^s(y)|\leq& Ch^{-1}|x-y|\sigma_z^s(\xi)\\
\leq& Ch^{-1}|x-y|\left(\sigma_z^s(x)+\sigma_z^s(y)\right)\\
\leq&Ch^{-1}|x-y|\sigma_z^s(y),
\end{align*}
where $\xi=mx+(1-m)y$,
which, together with coordinate transformation, leads to
\begin{align}
I_2\leq Ch^{-p}\sum_{K\in\mathcal{T}_h}\int_K|\nabla v(y)|^p\sigma_z^{sp}(y)\left(\int_K|x-y|^{-3-p(l-1)} dx\right)dy\leq Ch^{-pl}\|\sigma_z^s\nabla v\|_{L_h^p(\Omega)}^p.\label{fubini-proof:5}
\end{align}
Then, (\ref{fubini-proof:3})-(\ref{fubini-proof:5}), together with (\ref{fubini:1}), yield the desired results (\ref{fubini:2}) and (\ref{fubini:3}).
\end{proof}

\begin{lemma}\label{fubini1}
Let $2\leq p<\infty$, $\frac{3}{p}<s$, $0<\beta<1-\frac{1}{p}$ and $w\in W^{1+\frac{1}{p}+\beta,p}_h(\Omega)$, then we have
\begin{align}
&\int_{\Omega}\sum_{F\in\mathcal{E}_h}\Big(h_F^{1-p}\|\sigma_z^s[w]\|_{L^p(F)}^p+
h_F\|\sigma_z^s\{\nabla w\cdot{\bm n}\}\|_{L^p(F)}^p\Big)+\|\sigma_z^s\nabla w\|_{L_h^p(\Omega)}^p dz\nonumber\\
\leq& Ch^3\left(h^{-p}\|w\|_{L^p(\Omega)}^p+\|\nabla w\|_{L_h^p(\Omega)}^p+h^{1+p\beta}|\nabla w|_{W_h^{\frac{1}{p}+\beta,p}(\Omega)}^p\right),\label{fubini1:1}
\end{align}
and
\begin{align}
&\sum_{F\in\mathcal{E}_h}\int_{F}\sum_{F\in\mathcal{E}_h}\Big(h_F^{1-p}\|\sigma_z^s[w]\|_{L^p(F)}^p+
h_F\|\sigma_z^s\{\nabla w\cdot{\bm n}\}\|_{L^p(F)}^p\Big)dz\nonumber\\
&+\sum_{F\in\mathcal{E}_h}\int_{F}h^{-p}\|\sigma_z^sw\|_{L^p(\Omega)}^p+\|\sigma_z^s\nabla w\|_{L_h^p(\Omega)}^p dz\label{fubini1:2}\\
\leq& Ch^2\left(h^{-p}\|w\|_{L^p(\Omega)}^p+\|\nabla w\|_{L_h^p(\Omega)}^p+h^{1+p\beta}|\nabla w|_{W_h^{\frac{1}{p}+\beta,p}(\Omega)}^p\right),\nonumber
\end{align}
where $\sigma_z$ is a weight function defined in (\ref{swf}).
\end{lemma}
\begin{proof}
Here we only prove (\ref{fubini1:1}), since (\ref{fubini1:2}) can be proved similarly.
Using Lemma \ref{tra}, Lemma \ref{fubini}, and the fact that $|\nabla\sigma_z^s|\leq Ch^{-1}\sigma_z^s$, we obtain
\begin{align*}
&\int_{\Omega}\sum_{F\in\mathcal{E}_h}\Big(h_F^{1-p}\|\sigma_z^s[w]\|_{L^p(F)}^p+
h_F\|\sigma_z^s\{\nabla w\cdot{\bm n}\}\|_{L^p(F)}^p\Big)+\|\sigma_z^s\nabla w\|_{L_h^p(\Omega)}^p dz\\
\leq&C\int_{\Omega}h^{-p}\|\sigma_z^sw\|_{L^p(\Omega)}^p+\|\nabla(\sigma_z^sw)\|^p_{L_h^p(\Omega)}+\|\sigma_z^s\nabla w\|_{L_h^p(\Omega)}^pdz\\
&+C\int_{\Omega}\|\sigma_z^s\nabla w\|_{L_h^p(\Omega)}^p+h^{1+p\beta}|\sigma_z^s\nabla w|_{W^{\frac{1}{p}+\beta,p}_h(\Omega)}^pdz\\
\leq& C\int_{\Omega}h^{-p}\|\sigma_z^sw\|_{L^p(\Omega)}^p+\|\sigma_z^s\nabla w\|_{L_h^p(\Omega)}^p+h^{1+p\beta}|\sigma_z^s\nabla w|_{W^{\frac{1}{p}+\beta,p}_h(\Omega)}^pdz\\
\leq& Ch^3\left(h^{-p}\|w\|_{L^p(\Omega)}^p+\|\nabla w\|_{L_h^p(\Omega)}^p+h^{1+p\beta}|\nabla w|_{W_h^{\frac{1}{p}+\beta,p}(\Omega)}^p\right),
\end{align*}
which is the desired result (\ref{fubini1:1}). We conclude the proof.
\end{proof}

\section{}\label{apendixB}
In this section, we establish the inf-sup stability for the symmetric IPDG method, the hybrid mixed DG method, and two representative HDG methods introduced in this paper. We note that these inf-sup stability
are of independent interest in the field of numerical analysis.

\subsection{Inf-sup stability for IPDG methods}
First of all, we establish the inf-sup stability for the IPDG method. Before this, we introduce a projection operator $\mathcal{I}_h^{\rm DG}: L^1(\Omega)
\rightarrow V_h^{\rm DG}$ such that on each element $K\in\mathcal{T}_h$,
\begin{align}
(\mathcal{I}_h^{\rm DG}w,v_h)_K=(w,v_h)_K,\quad \forall v_h\in\mathcal{P}^k(K),~w\in L^1(\Omega).\label{po-IPDG}
\end{align}
Associated with the projection operator $\mathcal{I}_h^{\rm DG}$, we have the following stability and interpolation error estimates .
\begin{lemma}\label{po-IPDG-is}
Let the projection operator $\mathcal{I}_h^{\rm DG}$ be defined as above, then for $1\leq p\leq \infty$ and $w\in L^p(\Omega)$ we have
\begin{align}
\|\mathcal{I}_h^{{\rm DG}}w\|_{L^p(K)}\leq C\|w\|_{L^p(K)},\label{po-IPDG-is:1}
\end{align}
and
\begin{align}
\|w-\mathcal{I}_h^{{\rm DG}}w\|_{L^p(K)}\leq C\inf_{w_h\in \mathcal{P}^k(K)}\|w-w_h\|_{L^p(K)},\label{po-IPDG-is:2}
\end{align}
for each element $K\in\mathcal{T}_h$.
\end{lemma}
\begin{proof}
For $2\leq p\leq \infty$, Lemma \ref{inv}, (\ref{po-IPDG}), and the H\"{o}lder inequality yield
\begin{align}
\|\mathcal{I}_h^{{\rm DG}}w\|_{L^p(K)}\leq Ch_K^{3(\frac{1}{p}-\frac{1}{2})}\|\mathcal{I}_h^{{\rm DG}}w\|_{L^2(K)}\leq Ch_K^{3(\frac{1}{p}-\frac{1}{2})}\|w\|_{L^2(K)}\leq C\|w\|_{L^p(K)}.\label{po-IPDG-is-proof:1}
\end{align}
For $1\leq p<2$, let $q$ be the number such that $\frac{1}{p}+\frac{1}{q}=1$. Due to $1\leq p< 2$, we have $2< q\leq \infty$. Then, (\ref{po-IPDG}) and
(\ref{po-IPDG-is-proof:1}) yield
\begin{align*}
\|\mathcal{I}_h^{{\rm DG}}w\|_{L^p(K)}=\sup_{0\neq v\in L^q(K)}\frac{(\mathcal{I}_h^{\rm DG}w,v)_K}{\|v\|_{L^q(K)}}=\sup_{0\neq v\in L^q(K)}\frac{(w,\mathcal{I}_h^{\rm DG}v)_K}{\|v\|_{L^q(K)}}\leq C\|w\|_{L^p(K)}.
\end{align*}
Thus, the stability result (\ref{po-IPDG-is:1}) holds for $1\leq p\leq \infty$.

Next, we prove the interpolation error estimate (\ref{po-IPDG-is:2}). For any $w_h\in \mathcal{P}^k(K)$, (\ref{po-IPDG-is:1}) obtains
\begin{align*}
\|w-\mathcal{I}_h^{{\rm DG}}w\|_{L^p(K)}\leq\|w-w_h\|_{L^p(K)}+\|\mathcal{I}_h^{\rm DG}(w-w_h)\|_{L^p(K)}
\leq C\|w-w_h\|_{L^p(K)}.
\end{align*}
We conclude the proof.
\end{proof}

Based on the projection operator $\mathcal{I}_h^{\rm DG}$ and Lemma \ref{po-IPDG-is}, we now prove the following inf-sup stability.
\begin{lemma}\label{inf-sup}
For any given constant $\epsilon\in(0,\frac{1}{2})$ and $\frac{1}{1-\epsilon}<p<\frac{1}{\epsilon}$, there exists a positive constant $C$ independent of the mesh size such that
\begin{align}
|||v_h|||_{W^{1,q}_h(\Omega)}\leq C\sup_{0\neq w_h\in V_h^{{\rm DG}}}\frac{a_h(v_h,w_h)}{|||w_h|||_{W_h^{1,p}(\Omega)}},\quad \forall v_h\in V_h^{{\rm DG}},\label{inf-sup-1}
\end{align}
where $q$ is the conjugate number of $p$ such that $\frac{1}{p}+\frac{1}{q}=1$ and $|||\cdot|||_{W_h^{1,p}(\Omega)}$ is the mesh-dependent norm defined in
Theorem \ref{W1p-IPDG}.
\end{lemma}
\begin{proof}
To prove (\ref{inf-sup-1}) for the case $\frac{1}{1-\epsilon}<p\leq 2$, let $v\in H_0^1(\Omega)$ be the solution of the following problem
\begin{align*}
-\Delta v=\mathcal{L}_h v_h
\end{align*}
where $\mathcal{L}_hv_h\in V_h^{{\rm DG}}$ satisfies
\begin{align*}
(\mathcal{L}_hv_h,\phi_h)=a_h(v_h,\phi_h),\quad \forall\phi_h\in V_h^{{\rm DG}}.
\end{align*}
By the definition of $\mathcal{L}_hv_h$, $v_h\in V_h^{{\rm DG}}$ is the IPDG approximation of $v$.
Since $2\leq q<\frac{1}{\epsilon}$ for the case $\frac{1}{1-\epsilon}<p\leq 2$, (\ref{W1p-IPDG:1}) yields
\begin{align}
|||v_h|||_{W_h^{1,q}(\Omega)}\leq& |||v-v_h|||_{W_h^{1,q}(\Omega)}+\|\nabla v\|_{L^q(\Omega)}\nonumber\\
\leq& C\left(\|\nabla v\|_{L^{q}(\Omega)}+h^{\frac{1}{q}+\epsilon}|\nabla v|_{W^{\frac{1}{q}+\epsilon,q}(\Omega)}\right).\label{inf-sup:1}
\end{align}
From the inverse estimates, (\ref{model-r:3}), and (\ref{model-r:2}) with $s=1$, we have
\begin{align}
&\|\nabla v\|_{L^{q}(\Omega)}+h^{\frac{1}{q}+\epsilon}|\nabla v|_{W^{\frac{1}{q}+\epsilon,q}(\Omega)}\nonumber\\
\leq& C\Big(\|\mathcal{L}_hv_h\|_{W^{-1,q}(\Omega)}+h^{\frac{1}{q}+\epsilon}\|\mathcal{L}_hv_h\|_{W^{\frac{1}{q}+\epsilon-1,q}(\Omega)}\Big)
\leq C\|\mathcal{L}_hv_h\|_{W^{-1,q}(\Omega)}\label{inf-sup:2}\\
=&C\sup_{0\neq \phi\in W_0^{1,p}(\Omega)}\frac{(\mathcal{L}_hv_h,\phi)}{\|\nabla\phi\|_{L^p(\Omega)}}
=C\sup_{0\neq \phi\in W_0^{1,p}(\Omega)}\frac{(\mathcal{L}_hv_h,\mathcal{I}_h^{{\rm DG}}\phi)}{\|\nabla\phi\|_{L^p(\Omega)}}\nonumber\\
=&C\sup_{0\neq\phi \in W_0^{1,p}(\Omega)}\frac{a_h(v_h,\mathcal{I}_h^{\rm DG}\phi)}{\|\nabla\phi\|_{L^p(\Omega)}}.\nonumber
\end{align}
In addition, Lemma \ref{inv}, Lemma \ref{tra} with $\beta=1-\frac{1}{p}$, and Lemma \ref{po-IPDG-is} yield
\begin{align*}
|||\mathcal{I}_h^{\rm DG}\phi|||_{W_h^{1,p}(\Omega)}\leq& |||\mathcal{I}_h^{\rm DG}\phi-\phi|||_{W_h^{1,p}(\Omega)}+\|\nabla\phi\|_{L^p(\Omega)}\\
\leq& C\left(\|\nabla(\phi-\mathcal{I}_h^{\rm DG}\phi)\|_{L_h^p(\Omega)}+h^{-1}\|\phi-\mathcal{I}_h^{\rm DG}\phi\|_{L^p(\Omega)}\right)+\|\nabla\phi\|_{L^p(\Omega)}\\
\leq& C\|\nabla\phi\|_{L^p(\Omega)},
\end{align*}
which, together with (\ref{inf-sup:2}), leads to
\begin{align}
&\|\nabla v\|_{L^{q}(\Omega)}+h^{\frac{1}{q}+\epsilon}|\nabla v|_{W^{\frac{1}{q}+\epsilon,q}(\Omega)}\nonumber\\
\leq& C\sup_{0\neq \phi \in W_0^{1,p}(\Omega)}\frac{a_h(v_h,\mathcal{I}_h^{\rm DG}\phi)}{|||\mathcal{I}_h^{\rm DG}\phi|||_{W_h^{1,p}(\Omega)}}
\leq C\sup_{0\neq w_h\in V_h^{{\rm DG}}}\frac{a_h(v_h,w_h)}{|||w_h|||_{W_h^{1,p}(\Omega)}}.\label{inf-sup:3}
\end{align}
Combining (\ref{inf-sup:1}) and (\ref{inf-sup:3}), we derive (\ref{inf-sup-1}) for the case $\frac{1}{1-\epsilon}<p\leq 2$.

With (\ref{inf-sup-1}) for the case $\frac{1}{1-\epsilon}<p\leq 2$ established, by the proof of Theorem \ref{W1p-IPDG} presented in Subsection \ref{sec5:1}, we know that (\ref{W1p-IPDG:1})
holds for $\frac{1}{1-\epsilon}<p<2$, and consequently, (\ref{inf-sup:1}) holds for $\frac{1}{1-\epsilon}<q<2$. Then, using the same argument as above, we can also establish (\ref{inf-sup-1}) (with the proof omitted) for the case $2<p<\frac{1}{\epsilon}$. We conclude the proof.
\end{proof}

\subsection{Inf-sup stability for the hybrid mixed DG and HDG methods}
Next, we establish the inf-sup stability for the hybrid mixed DG method and two representative HDG methods.
To this end, we introduce three projection operators $\mathcal{I}_h^W: (L^1(\Omega))^3\rightarrow {\bm W}_h$, $\mathcal{I}_h^V: L^1(\Omega)\rightarrow
V_h$, and $\mathcal{I}_h^{\widehat{V}}: L^1(\mathcal{E}_h)\rightarrow \widehat{V}_h$. On each element $K\in\mathcal{T}_h$ and face $F\in\mathcal{E}_h$,
these projection operators are defined by
\begin{align}
(\mathcal{I}_h^W{\bm w},{\bm v}_h)_K=&({\bm w},{\bm v}_h)_K,\quad \forall {\bm v}_h\in {\bm W}(K),~{\bm w}\in(L^1(\Omega))^3,\label{po-HDG:1}\\
(\mathcal{I}_h^V w,v_h)_K=&(w,v_h)_K,\quad \forall v_h\in V(K),~w\in L^1(\Omega),\label{po-HDG:2}\\
\langle\mathcal{I}_h^{\widehat{V}}\mu,\widehat{v}_h\rangle_F=&\langle\mu,\widehat{v}_h\rangle_F,\quad \forall \widehat{v}_h\in\mathcal{P}^k(F),~\mu\in L^1(\mathcal{E}_h).\label{po-HDG:3}
\end{align}
Associated with these three projection operators, we have the following stability and interpolation error estimates.
\begin{lemma}\label{po-HDG-is}
Let the projection operators $\mathcal{I}_h^W$, $\mathcal{I}_h^V$, and $\mathcal{I}_h^{\widehat{V}}$ be defined as above, then for $1\leq p\leq \infty$ and $({\bm w}, w, \mu)\in (L^p(\Omega))^3\times L^p(\Omega)\times L^p(\mathcal{E}_h)$ we have
\begin{align}
\|\mathcal{I}_h^W{\bm w}\|_{L^p(K)}\leq C\|{\bm w}\|_{L^p(K)},\quad \|\mathcal{I}_h^Vw\|_{L^p(K)}\leq C\|w\|_{L^p(K)},\quad
\|\mathcal{I}_h^{\widehat{V}}\mu\|_{L^p(F)}\leq C\|\mu\|_{L^p(F)},\label{po-HDG-is:1}
\end{align}
and
\begin{align}
\|{\bm w}-\mathcal{I}_h^W{\bm w}\|_{L^p(K)}\leq& C\inf_{{\bm w}_h\in {\bm W}(K)}\|{\bm w}-{\bm w}_h\|_{L^p(K)},\\
\|w-\mathcal{I}_h^Vw\|_{L^p(K)}\leq& C\inf_{w_h\in V(K)}\|w-w_h\|_{L^p(K)},\\
\|\mu-\mathcal{I}_h^{\widehat{V}}\mu\|_{L^p(F)}\leq& C\inf_{\mu_h\in\mathcal{P}^k(F)}\|\mu-\mu_h\|_{L^p(F)},
\end{align}
for each element $K\in\mathcal{T}_h$ and face $F\in\mathcal{E}_h$.
\end{lemma}
\begin{proof}
Since the proof is very similar to that in Lemma \ref{po-IPDG-is}, we omit it here.
\end{proof}

Based on Lemma \ref{po-HDG-is}, we now establish the following inf-sup stability for the hybrid mixed DG method and two representative HDG methods.

\begin{lemma}\label{inf-sup-HDG0}
For any given constant $\epsilon\in (0,\frac{1}{2})$ and $\frac{1}{1-\epsilon}<p<\frac{1}{\epsilon}$, there exists a positive constant $C$ independent of
the mesh size such that the following inf-sup stability holds:
\begin{align}
|||({\bm w}_{1,h},v_{1,h},\widehat{v}_{1,h})|||_{W_{h^{\prime}}^{1,q}(\Omega)}\leq C\sup_{0\neq ({\bm w}_{2,h},v_{2,h},\widehat{v}_{2,h})\in{\bm V}_h}
\frac{\mathcal{B}({\bm w}_{1,h},v_{1,h},\widehat{v}_{1,h};{\bm w}_{2,h},v_{2,h},\widehat{v}_{2,h})}{|||({\bm w}_{2,h},v_{2,h},\widehat{v}_{2,h})|||_{W_{h^{\prime}}^{1,p}(\Omega)}},\label{inf-sup-HDG0:1}
\end{align}
for any $({\bm w}_{1,h},v_{1,h},\widehat{v}_{1,h})\in\widetilde{{\bm V}}_h:=\{({\bm w}_h,v_h,\widehat{v}_h)\in{\bm V}_h:\mathcal{B}({\bm w}_h,v_h,\widehat{v}_h;0,0,\widehat{\phi}_h)=0,~\forall \widehat{\phi}_h\in\widehat{V}_h\}$, where ${\bm V}_h={\bm W}_h\times V_h\times\widehat{V}_h$, $\mathcal{B}$ is defined in Subsection \ref{sec3:2}, $q$ is the conjugate number of $p$ such that $\frac{1}{p}+\frac{1}{q}=1$, and the mesh-dependent norm
$|||(\cdot,\cdot,\cdot)|||_{W_{h^{\prime}}^{1,p}(\Omega)}$ is defined in Corollary \ref{W1p-HDG-c1}.
\end{lemma}
\begin{proof}
To prove (\ref{inf-sup-HDG0:1}) for the case $\frac{1}{1-\epsilon}<p\leq2$, let $({\bm w}_{h},v_{h})\in {\bm W}_h\times V_h$ for the hybrid mixed DG method and $({\bm w}_{h},v_{h})\in (\mathcal{C}(\overline{\Omega}))^3\cap{\bm W}_h\times V_h$ for the HDG-1 and HDG-2 methods be the solution of the following problem:
\begin{align}
({\bm w}_{h},\boldsymbol{\psi}_h)-(v_{h},\phi_h)=\mathcal{B}({\bm w}_{1,h},v_{1,h},\widehat{v}_{1,h};\boldsymbol{\psi}_h,\phi_h,\widehat{\phi}_h),\quad \forall
(\boldsymbol{\psi}_h,\phi_h,\widehat{\phi}_h)\in{\bm V}_h.\label{inf-sup-HDG0-proof:1}
\end{align}
Let $({\bm w}_1,v_1)$ be the solution of the following problem associated with ${\bm w}_{h}$ and $v_{h}$:
\begin{align*}
{\bm w}_1+\nabla v_1={\bm w}_{h}\quad {\rm in}~\Omega,\quad \nabla\cdot{\bm w}_1=v_{h}\quad {\rm in}~\Omega,\quad v_1=0\quad {\rm on}~\partial\Omega,
\end{align*}
Then, we know that $({\bm w}_{1,h},v_{1,h},\widehat{v}_{1,h})$ is the HDG approximation of $({\bm w}_1,v_1)$. Since $2\leq q<\frac{1}{\epsilon}$ for the case $\frac{1}{1-\epsilon}<p\leq 2$, (\ref{W1p-HDG-c1:1}) yields
\begin{align}
|||({\bm w}_{1,h},v_{1,h},\widehat{v}_{1,h})|||_{W_{h^{\prime}}^{1,q}(\Omega)}\leq& |||({\bm w}_1-{\bm w}_{1,h},v_1-v_{1,h},v_1-\widehat{v}_{1,h})|||_{W_{h^{\prime}}^{1,q}(\Omega)}\nonumber\\
&+\|{\bm w}_1\|_{L^q(\Omega)}+\|\nabla v_1\|_{L^q(\Omega)}\label{inf-sup-HDG0-proof:2}\\
\leq& C\left(\|\nabla v_1\|_{L^q(\Omega)}+\|{\bm w}_1\|_{L^q(\Omega)}+h^{\frac{1}{q}+\epsilon}|{\bm w}_1|_{W^{\frac{1}{q}+\epsilon,q}(\Omega)}\right).\nonumber
\end{align}
Using the inverse estimates, (\ref{inf-sup-HDG0-proof:1}), (\ref{model-r:3}), and (\ref{model-r:2}) with $s=1$, we have
\begin{align}
&\|\nabla v_1\|_{L^q(\Omega)}+\|{\bm w}_1\|_{L^q(\Omega)}+h^{\frac{1}{q}+\epsilon}|{\bm w}_1|_{W^{\frac{1}{q}+\epsilon,q}(\Omega)}\nonumber\\
\leq& C\Big(\|v_{h}-\nabla\cdot{\bm w}_{h}\|_{W^{-1,q}(\Omega)}+h^{\frac{1}{q}+\epsilon}
\|v_{h}-\nabla\cdot{\bm w}_{h}\|_{W^{\frac{1}{q}+\epsilon-1,q}(\Omega)}\Big)\nonumber\\
\leq&C\|v_{h}-\nabla\cdot{\bm w}_{h}\|_{W^{-1,q}(\Omega)}
= C\sup_{0\neq \phi\in W_0^{1,p}(\Omega)}\frac{({\bm w}_{h},\nabla\phi)+(v_{h},\phi)}{\|\nabla\phi\|_{L^p(\Omega)}}\label{inf-sup-HDG0-proof:3}\\
\leq&C\sup_{0\neq \phi\in W_0^{1,p}(\Omega),\boldsymbol{\psi}=\nabla\phi}\frac{({\bm w}_{h},\mathcal{I}_h^W\boldsymbol{\psi})+(v_{h},\mathcal{I}_h^V\phi)}{\|\nabla\phi\|_{L^p(\Omega)}+\|\boldsymbol{\psi}\|_{L^p(\Omega)}}\nonumber\\
=&C\sup_{0\neq \phi\in W_0^{1,p}(\Omega),\boldsymbol{\psi}=\nabla\phi}\frac{\mathcal{B}({\bm w}_{1,h},v_{1,h},\widehat{v}_{1,h};\mathcal{I}_h^W\boldsymbol{\psi},-\mathcal{I}_h^V\phi,-\mathcal{I}_h^{\widehat{V}}\phi)}{\|\nabla\phi\|_{L^p(\Omega)}+\|\boldsymbol{\psi}\|_{L^p(\Omega)}}.\nonumber
\end{align}
In addition, Lemma \ref{inv}, Lemma \ref{tra} with $\beta=1-\frac{1}{p}$, Lemma \ref{po-IPDG-is}, and Lemma \ref{po-HDG-is} yield
\begin{align*}
&|||(\mathcal{I}_h^W\boldsymbol{\psi},-\mathcal{I}_h^V\phi,-\mathcal{I}_h^{\widehat{V}}\phi)|||_{W_{h^{\prime}}^{1,p}(\Omega)}\\
\leq& C\|\boldsymbol{\psi}\|_{L^p(\Omega)}+\|\nabla\phi\|_{L^p(\Omega)}+C\Big(\|\boldsymbol{\psi}-I_h^W\boldsymbol{\psi}\|_{L^p(\Omega)}+\|\nabla(\mathcal{I}_h^V\phi-\phi)\|_{L_h^p(\Omega)}\\
&+h^{-1}\|\phi-\mathcal{I}_h^V\phi\|_{L^p(\Omega)}+\|\nabla(\mathcal{I}_h^{\rm DG}\phi-\phi)\|_{L_h^p(\Omega)}+h^{-1}\|\phi-\mathcal{I}_h^{\rm DG}\phi\|_{L^p(\Omega)}\Big)\\
\leq &C\left(\|\boldsymbol{\psi}\|_{L^p(\Omega)}+\|\nabla\phi\|_{L^p(\Omega)}\right),
\end{align*}
which, together with (\ref{inf-sup-HDG0-proof:3}), leads to
\begin{align}
&\|\nabla v_1\|_{L^q(\Omega)}+\|{\bm w}_1\|_{L^q(\Omega)}+h^{\frac{1}{q}+\epsilon}|{\bm w}_1|_{W^{\frac{1}{q}+\epsilon,q}(\Omega)}\nonumber\\
\leq& C\sup_{0\neq\phi\in W_0^{1,p}(\Omega),\boldsymbol{\psi}=\nabla\phi}\frac{\mathcal{B}({\bm w}_{1,h},v_{1,h},\widehat{v}_{1,h};
\mathcal{I}_h^W\boldsymbol{\psi},-\mathcal{I}_h^V\phi,-\mathcal{I}_h^{\widehat{V}}\phi)}
{|||(\mathcal{I}_h^W\boldsymbol{\psi},-\mathcal{I}_h^V\phi,-\mathcal{I}_h^{\widehat{V}}\phi)|||_{W_{h^{\prime}}^{1,p}(\Omega)}}\label{inf-sup-HDG0-proof:4}\\
\leq& C\sup_{0\neq({\bm w}_{2,h},v_{2,h},\widehat{v}_{2,h})\in{\bm V}_h}\frac{\mathcal{B}({\bm w}_{1,h},v_{1,h},\widehat{v}_{1,h};
{\bm w}_{2,h},v_{2,h},\widehat{v}_{2,h})}
{|||({\bm w}_{2,h},v_{2,h},\widehat{v}_{2,h})|||_{W_{h^{\prime}}^{1,p}(\Omega)}}.\nonumber
\end{align}
Thus, combining (\ref{inf-sup-HDG0-proof:2}) and (\ref{inf-sup-HDG0-proof:4}), we derive (\ref{inf-sup-HDG0:1}) for the case
$\frac{1}{1-\epsilon}<p\leq 2$.

With the inf-sup stability (\ref{inf-sup-HDG0:1}) for the case $\frac{1}{1-\epsilon}<p\leq 2$ established, we know that Theorem \ref{W1p-HDG-pre} holds for
$\frac{1}{1-\epsilon}<p<2$, and consequently, Corollary \ref{W1p-HDG-c1} holds for $\frac{1}{1-\epsilon}<p<2$. Therefore, using the same argument as above,
we can also establish (with the proof omitted) the inf-sup stability (\ref{inf-sup-HDG0:1}) for the case $2<p<\frac{1}{\epsilon}$. 
\end{proof}

\section{}\label{apendixC}

\subsection{Proof of (\ref{W1p-HDG-pre:2})}\label{apendixC:1}
In this subsection, we briefly explain how to obtain (\ref{W1p-HDG-pre:2}). For $z\in K_z$ with $K_z\in\mathcal{T}_h$, let $(\boldsymbol{\mathcal{R}}_2(z,x),\Gamma(z,x))$ be the regularized Green's function defined in (\ref{rG-mixed2}). Then, similar to (\ref{rG-HDG2:1}), using (\ref{proj_HDG-1}), (\ref{ddd:2}), (\ref{rG-mixed2}),
and (\ref{rG-HDG2}), we have for $2\leq p<\frac{1}{\epsilon}$
\begin{align*}
-({\bm e}_i\cdot(\Pi_W{\bm q}-{\bm q}_h))(z)=&(\boldsymbol{\mathcal{R}}_{2,h}-\boldsymbol{\mathcal{R}}_2,\Pi_W{\bm q}-{\bm q})_{\mathcal{T}_h}+
(\nabla(\Gamma_{2,h}-\Gamma),\Pi_W{\bm q}-{\bm q})_{\mathcal{T}_h}\\
&-\langle\Gamma_{2,h}-\widehat{\Gamma}_{2,h},(\Pi_W{\bm q}-{\bm q})\cdot{\bm n}\rangle_{\partial\mathcal{T}_h}-({\bm e}_i\delta_{h,z},\Pi_W{\bm q}-{\bm q})_{\mathcal{T}_h}\\
\leq&C\Big(\|\sigma_z^{-s}(\boldsymbol{\mathcal{R}}_2-\boldsymbol{\mathcal{R}}_{2,h})\|_{L^q(\Omega)}
+\|\sigma_z^{-s}\nabla(\Gamma-\Gamma_{2,h})\|_{L^q_h(\Omega)}\\
&+\Big(\sum_{K\in\mathcal{T}_h}h_K^{1-q}\|\sigma_z^{-s}(\Gamma_{2,h}-\widehat{\Gamma}_{2,h})\|_{L^q(\partial K)}^q\Big)^{1/q}+h^{-\frac{3}{p}}\Big)\\
&\times\Big(\|\sigma_z^s({\bm q}-\Pi_W{\bm q})\|_{L^p(\Omega)}+\Big(\sum_{K\in\mathcal{T}_h}h_K\|\sigma_z^s({\bm q}-\Pi_W{\bm q})\cdot{\bm n}\|_{L^p(\partial K)}^p\Big)^{1/p}\Big).
\end{align*}
where $s$ is any positive real number, $\sigma$ is a weight function define in (\ref{swf}), $q$ is the conjugate number of $p$ such that $\frac{1}{p}+\frac{1}{q}=1$, and ${\bm e}_i$ denotes the $i$-th standard basis vector of $\mathbb{R}^3$. Using the similar argument presented in the proof of Lemma \ref{rG01-HDG-L1}, we can obtain (with the proof omitted) for $\frac{3}{p}<s<\gamma+\frac{3}{p}$
\begin{align*}
&\|\sigma_z^{-s}(\boldsymbol{\mathcal{R}}_2-\boldsymbol{\mathcal{R}}_{2,h})\|_{L^q(\Omega)}
+\|\sigma_z^{-s}\nabla(\Gamma-\Gamma_{2,h})\|_{L^q_h(\Omega)}\\
&+\Big(\sum_{K\in\mathcal{T}_h}h_K^{1-q}\|\sigma_z^{-s}(\Gamma_{2,h}-\widehat{\Gamma}_{2,h})\|_{L^q(\partial K)}^q\Big)^{1/q}\leq Ch^{-\frac{3}{p}},
\end{align*}
where $\gamma\in(0,1)$ is a constant depending on the domain $\Omega$. Consequently, for $z\in K_z$ with $K_z\in\mathcal{T}_h$ and
$\frac{3}{p}<s<\gamma+\frac{3}{p}$ with $2\leq p<\frac{1}{\epsilon}$, we have by using Lemma \ref{tra} with $\beta=\epsilon$
\begin{align*}
|(\Pi_W{\bm q}-{\bm q}_h)(z)|\leq& Ch^{-\frac{3}{p}}\Big(\|\sigma_z^s({\bm q}-\Pi_W{\bm q})\|_{L^p(\Omega)}+\Big(\sum_{K\in\mathcal{T}_h}h_K\|\sigma_z^s({\bm q}-\Pi_W{\bm q})\cdot{\bm n}\|_{L^p(\partial K)}^p\Big)^{1/p}\Big)\\
\leq&Ch^{-\frac{3}{p}}\Big(\|\sigma_z^s({\bm q}-\Pi_W{\bm q})\|_{L^p(\Omega)}+h^{\frac{1}{p}+\epsilon}|\sigma_z^s({\bm q}-\Pi_W{\bm q})|_{W_h^{\frac{1}{p}+\epsilon,p}(\Omega)}\Big),
\end{align*}
which, together with the triangle inequality and Lemma \ref{fubini}, yields
\begin{align*}
\|{\bm q}-{\bm q}_h\|_{L^p(\Omega)}\leq C\left(\|{\bm q}-\Pi_W{\bm q}\|_{L^p(\Omega)}+h^{\frac{1}{p}+\epsilon}|{\bm q}-\Pi_W{\bm q}|_{W_h^{\frac{1}{p}+\epsilon,p}}\right),\quad 2\leq p<\frac{1}{\epsilon}.
\end{align*}

To prove (\ref{W1p-HDG-pre:2}) for the case $\frac{1}{1-\epsilon}<p<2$, let $(\boldsymbol{\psi}_h,\phi_h,\widehat{\phi}_h)\in{\bm W}_h\times V_h\times \widehat{V}_h$ be the solution of the following problem:
\begin{align}
\mathcal{B}(\boldsymbol{\psi}_h,\phi_h,\widehat{\phi}_h;{\bm w}_h,v_h,\widehat{v}_h)=(|\Pi_W{\bm q}-{\bm q}_h|^{p-2}(\Pi_W{\bm q}-{\bm q}_h),{\bm w}_h)_{\mathcal{T}_h},
\end{align}
for any $({\bm w}_h,v_h,\widehat{v}_h)\in {\bm W}_h\times V_h\times \widehat{V}_h$. Then, Lemma \ref{dfp}, the inf-sup condition (\ref{inf-sup-HDG0:1}), and the H\"{o}lder's inequality yield
\begin{align*}
\|\Pi_W{\bm q}-{\bm q}_h\|_{L^p(\Omega)}^p=&\mathcal{B}(\boldsymbol{\psi}_h,\phi_h,\widehat{\phi}_h;\Pi_W{\bm q}-{\bm q}_h,\Pi_Vu-u_h,\Pi_{\widehat{V}}u-\widehat{u}_h)\\
=&\mathcal{B}(\boldsymbol{\psi}_h,\phi_h,\widehat{\phi}_h;\Pi_W{\bm q}-{\bm q},\Pi_Vu-u,\Pi_{\widehat{V}}u-u),
\end{align*}
and
\begin{align*}
|||(\boldsymbol{\psi}_h,\phi_h,\widehat{\phi}_h)|||_{W_{h^{\prime}}^{1,q}(\Omega)}\leq& C\sup_{0\neq ({\bm w}_h,v_h,\widehat{v}_h)\in{\bm V}_h}\frac{(|\Pi_W{\bm q}-{\bm q}_h|^{p-2}(\Pi_W{\bm q}-{\bm q}_h),{\bm w}_h)_{\mathcal{T}_h}}{|||({\bm w}_h,v_h,\widehat{v}_h)|||_{W_{h^{\prime}}^{1,p}(\Omega)}}\\
\leq& C\|\Pi_W{\bm q}-{\bm q}_h\|_{L^p(\Omega)}^{p-1},
\end{align*}
where the mesh-dependent norm $|||(\cdot,\cdot,\cdot)|||_{W_{h^{\prime}}^{1,p}(\Omega)}$ is defined in Corollary \ref{W1p-HDG-c1} and the space ${\bm V}_h$
is define in Lemma \ref{inf-sup-HDG0}. Therefore, using Lemma \ref{tra} with $\beta=\epsilon$, (\ref{proj_HDG-1}), and the definition of $\mathcal{B}$, we obtain
\begin{align*}
\|\Pi_W{\bm q}-{\bm q}_h\|_{L^p(\Omega)}\leq& C\frac{\mathcal{B}(\boldsymbol{\psi}_h,\phi_h,\widehat{\phi}_h;\Pi_W{\bm q}-{\bm q},\Pi_Vu-u,\Pi_{\widehat{V}}u-u)}{|||(\boldsymbol{\psi}_h,\phi_h,\widehat{\phi}_h)|||_{W_{h^{\prime}}^{1,q}(\Omega)}}\\
=&\frac{(\boldsymbol{\psi}_h+\nabla\phi_h,\Pi_W{\bm q}-{\bm q})_{\mathcal{T}_h}-\langle\phi_h-\widehat{\phi}_h,(\Pi_W{\bm q}-{\bm q})\cdot{\bm n}\rangle
_{\partial\mathcal{T}_h}}{|||(\boldsymbol{\psi}_h,\phi_h,\widehat{\phi}_h)|||_{W_{h^{\prime}}^{1,q}(\Omega)}}\\
\leq& C\left(\|{\bm q}-\Pi_W{\bm q}\|_{L^p(\Omega)}+h^{\frac{1}{p}+\epsilon}|{\bm q}-\Pi_W{\bm q}|_{W_h^{\frac{1}{p}+\epsilon,p}}\right),
\end{align*}
which, together with the triangle inequality, leads to (\ref{W1p-HDG-pre:2}) for the case $\frac{1}{1-\epsilon}<p<2$.

\subsection{Proof of Corollary \ref{W1p-HDG-c}}\label{apendixC:2}
From (\ref{HDG:3}), we derive that on each interior face $F\in\mathcal{E}_h^o$ shared by the elements $K^+$ and $K^-$,
\begin{subequations}\label{W1p-HDG-c-proof:1}
\begin{align}
(\tau^++\tau^-)(\mathcal{L}u_h^+-\widehat{u}_h)=-[{\bm q}_h\cdot{\bm n}]+\tau^-\mathcal{L}[u_h],\\
(\tau^++\tau^-)(\mathcal{L}u_h^--\widehat{u}_h)=-[{\bm q}_h\cdot{\bm n}]-\tau^+\mathcal{L}[u_h],
\end{align}
\end{subequations}
where $v^+$ and $v^-$ denote the restrictions of $v$ on $K^+\cap F$ and $K^-\cap F$ for $v\in\{\tau,u_h\}$. Note that the derivation of the above two identities also can be found in \cite[Subsection 4.3]{lq2026}. Since $\tau=h_K^{-1}$ on each element $K\in\mathcal{T}_h$ for the HDG-2 method, identities (\ref{W1p-HDG-c-proof:1}) yield
\begin{align}
\tau^{p-1}\|\mathcal{L}u_h-\widehat{u}_h\|^p_{L^p(\partial K)}=&\sum_{F\in\partial K}h_K\|\tau(\mathcal{L}u_h-\widehat{u}_h)\|^p_{L^p(F)}\nonumber\\
\leq& C\sum_{F\in\partial K}\left(h\|[{\bm q}_h\cdot{\bm n}]\|_{L^p(F\backslash\mathcal{E}_h^{\partial})}^p+h_F^{1-p}\|[u_h]\|_{L^p(F)}^p\right),\quad \forall K\in\mathcal{T}_h.\label{W1p-HDG-c-proof:2}
\end{align}
For the hybrid mixed DG method, since $\tau=0$, estimate (\ref{W1p-HDG-c-proof:2}) holds trivially. For the HDG-1 method, since
$\tau=h_K^{-1}$ on $F_K^{\ast}$ and $\tau=0$ on $\partial K\backslash F_K^{\ast}$, identities (\ref{W1p-HDG-c-proof:1}) yield
\begin{align}
\tau^{p-1}\|\mathcal{L}u_h-\widehat{u}_h\|^p_{L^p(\partial K)}=&h_K\|\tau(\mathcal{L}u_h-\widehat{u}_h)\|^p_{L^p(F_K^{\ast})}\nonumber\\
\leq&C\left(h\|[{\bm q}_h\cdot{\bm n}]\|_{L^p(F_K^{\ast}\backslash\mathcal{E}_h^{\partial})}^p+h_F^{1-p}\|[u_h]\|_{L^p(F_{K}^{\ast})}^p\right),\quad \forall K\in\mathcal{T}_h.\label{W1p-HDG-c-proof:3}
\end{align}
Then, summing (\ref{W1p-HDG-c-proof:2}) and (\ref{W1p-HDG-c-proof:3}) over the mesh $\mathcal{T}_h$ and using Lemma \ref{inv} and Lemma \ref{tra} with
$\beta=\epsilon$, we obtain
\begin{align*}
\sum_{K\in\mathcal{T}_h}\tau^{p-1}\|\mathcal{L}u_h-\widehat{u}_h\|_{L^p(\partial K)}^p\leq& C\Big(\|{\bm q}-{\bm q}_h\|_{L^p(\Omega)}^p
+h^{1+p\epsilon}|{\bm q}-{\bm q}_h|_{W_h^{\frac{1}{p}+\epsilon,p}(\Omega)}^p\Big)\\
&+C\sum_{F\in\mathcal{E}_h}h_F^{1-p}\|[u_h]\|_{L^p(F)}^p\\
\leq&C\Big(\|{\bm q}-{\bm q}_h\|_{L^p(\Omega)}^p+\|{\bm q}-\Pi_W{\bm q}\|_{L^p(\Omega)}^p\\
&+h^{1+p\epsilon}|{\bm q}-\Pi_W{\bm q}|_{W_h^{\frac{1}{p}+\epsilon,p}(\Omega)}^p\Big)
+C\sum_{F\in\mathcal{E}_h}h_F^{1-p}\|[u_h]\|_{L^p(F)}^p,
\end{align*}
which, together with (\ref{W1p-HDG-pre:1}), obtains the result (\ref{W1p-HDG-c:1}).

\subsection{Proof of Corollary \ref{W1p-HDG-c1}}\label{apendixC:3}
By (\ref{HDG:1}) and the definitions of $\Pi_W$, $\Pi_V$, and $\Pi_{\widehat{V}}$, we have the following error equation:
\begin{align*}
(\Pi_W{\bm q}-{\bm q}_h,{\bm w}_h)_{\mathcal{T}_h}-(\Pi_Vu-u_h,\nabla\cdot{\bm w}_h)_{\mathcal{T}_h}+\langle\Pi_{\widehat{V}}u-\widehat{u}_h,{\bm w}_h\cdot{\bm n}\rangle_{\partial\mathcal{T}_h}=0,\quad \forall{\bm w}_h\in{\bm W}_h.
\end{align*}
Then, applying Lemma \ref{inequality} to the above error equation, we obtain for each element $K\in\mathcal{T}_h$
\begin{align*}
h_K^{-1}\|u_h-\widehat{u}_h\|_{L^2(\partial K)}^2=&h_K^{-1}\langle(\Pi_{\widehat{V}}u-\widehat{u}_h)-(\Pi_Vu-u_h),u_h-\widehat{u}_h\rangle_{\partial K}\\
&+h_K^{-1}\langle\Pi_Vu-\Pi_{\widehat{V}}u ,u_h-\widehat{u}_h\rangle_{\partial K}\\
\leq&Ch_K^{-1/2}\|u_h-\widehat{u}_h\|_{L^2(K)}\Big(\|\Pi_W{\bm q}-{\bm q}_h\|_{L^2(K)}+h_K^{-1/2}\|\Pi_Vu-\Pi_{\widehat{V}}u\|_{L^2(\partial K)}\\
&+\|\sqrt{\tau}(\mathcal{L}(\Pi_Vu-u_h)-(\Pi_{\widehat{V}}u-\widehat{u}_h))\|_{L^2(\partial K)}\Big),
\end{align*}
which, together with Lemma \ref{inv} and Lemma \ref{tra} for $\beta=1-\frac{1}{p}$, leads to
\begin{align}
h_K^{\frac{1}{p}-1}\|u_h-\widehat{u}_h\|_{L^p(\partial K)}\leq& Ch_K^{3(\frac{1}{p}-\frac{1}{2})}h_K^{-\frac{1}{2}}\|u_h-\widehat{u}_h\|_{L^2(\partial K)}\nonumber\\
\leq& Ch_K^{3(\frac{1}{p}-\frac{1}{2})}\Big(\|\Pi_W{\bm q}-{\bm q}_h\|_{L^2(K)}+h_K^{-1/2}\|\Pi_Vu-\Pi_{\widehat{V}}u\|_{L^2(\partial K)}\nonumber\\
&+\|\sqrt{\tau}(\mathcal{L}(\Pi_Vu-u_h)-(\Pi_{\widehat{V}}u-\widehat{u}_h))\|_{L^2(\partial K)}\Big)\nonumber\\
\leq& C\Big(\|\Pi_W{\bm q}-{\bm q}_h\|_{L^p(K)}+h_K^{\frac{1}{p}-1}\|\Pi_Vu-\Pi_{\widehat{V}}u\|_{L^p(\partial K)}\label{W1p-HDG-c1-proof:1}\\
&+\tau^{1-\frac{1}{p}}\|\mathcal{L}(\Pi_Vu-u_h)-(\Pi_{\widehat{V}}u-\widehat{u}_h)\|_{L^p(\partial K)}\Big)\nonumber\\
\leq&C\Big(\|{\bm q}-{\bm q}_h\|_{L^p(K)}+\|{\bm q}-\Pi_W{\bm q}\|_{L^p(K)}+\|\nabla(u-\Pi_Vu)\|_{L^p(K)}\nonumber\\
&+\|\nabla(u-\Pi_ku)\|_{L^p(K)}+\tau^{1-\frac{1}{p}}\|\mathcal{L}u_h-\widehat{u}_h\|_{L^p(\partial K)}\Big).\nonumber
\end{align}
Combining (\ref{W1p-HDG-c:1}) and (\ref{W1p-HDG-c1-proof:1}), we conclude the proof.

{\bf Acknowledgements}  All authors contributed equally. Haitao Leng was supported by Basic and
Applied Basic Research Foundation of Guangdong Province (Grant No. 2024A1515011163, 2025A1515010428).



\end{document}